\RequirePackage{fix-cm}
\documentclass[envcountsect,numbook]{svjour3_arxiv3}   
\smartqed  
\usepackage{graphicx}
\graphicspath{{./}}
\usepackage[utf8]{inputenc}
\usepackage{dsfont}
\usepackage{subfigure}
\usepackage{color}
\usepackage{url}
\usepackage{enumerate} 
\usepackage{booktabs}
\usepackage{tabularx}
\usepackage{ltablex}
\usepackage{mathptmx}
\usepackage{amssymb,amsmath,latexsym} 
\usepackage{mathtools}
\usepackage{longtable} 
\usepackage{textcomp}
\usepackage{lscape}
\usepackage[12pt]{extsizes}
\usepackage{geometry}
\usepackage{ifthen}
\usepackage{calrsfs}
\DeclareMathAlphabet{\pazocal}{OMS}{zplm}{m}{n}
\let\mathcal\pazocal
\usepackage{mathtools}
\mathtoolsset{showonlyrefs}

\allowdisplaybreaks

\spnewtheorem{theorem}{Theorem}{\bfseries}{\itshape}
\spnewtheorem{corollary}[theorem]{Corollary}{\bfseries}{\itshape}
\spnewtheorem{lemma}[theorem]{Lemma}{\bfseries}{\itshape}
\spnewtheorem{proposition}[theorem]{Proposition}{\bfseries}{\itshape}
\spnewtheorem{definition}[theorem]{Definition}{\bfseries}{\itshape}
\spnewtheorem{remark}[theorem]{Remark}{\bfseries}{\upshape}
\spnewtheorem{assumption}[theorem]{Assumption}{\bfseries}{\itshape}
\spnewtheorem{algo}[theorem]{Algorithm}{\bfseries}{\itshape}

\counterwithin{theorem}{section}
\smartqed

\renewcommand{\paragraph}[1]{{\smallskip\noindent\textbf{#1.}}}
\definecolor{myred}{rgb}{0.8,0,0}  
\newcommand{\neu}{\color{blue} }
{\vskip\baselineskip\noindent\textbf{Proof of {#1}:}}%
{\hspace*{.1pt}\hspace*{\fill}\BOX\vskip\baselineskip}

\usepackage{tikz}

\def \R{\mathbb{R}}               
\def \N{\mathbb{N}}               
\def \1{{\bf 1}}                
\def \0{{\bf 0}}

\newcommand\norm[1]{\left\lVert#1\right\rVert}

\def \trace{\operatorname{tr}}

\def \diag{\operatorname{diag}}

\definecolor{myred}{rgb}{0.9,0,0}  
\definecolor{mygreen}{rgb}{0,0.7,0}  
\definecolor{myblue}{rgb}{0.2,0,0.8}  
\definecolor{orange}{rgb}{1,0.6,0}  
\definecolor{olive}{rgb}{0.5,0.5,0}  
\definecolor{mylila}{rgb}{0.8,0.5,0.2}

\def \Domainspace{\mathcal{D}}  

\def \cgram{\mathcal{G}_C}
\def \ogram{\mathcal{G}_O}
\def \cgramtr{\overline{\mathcal{G}}_C}
\def \ogramtr{\overline{\mathcal{G}}_O}
\def \cfun{\mathcal{E}_C}
\def \ofun{\mathcal{E}_O}
\newcommand{\trans}{\mathcal{T}}
\def \omatrix{C}
\def \one{\mathds{1}}

\def \dimred{\ell}
\def \phx{\text{PHX} }
\def \phxk{\text{PHX}}
\def \phxs{\text{PHX}s }
\def \phxsk{\text{PHX}s}

\def \medium{M}
\def \fluid{F}
\def \outlet{O}
\def \inlet{I}
\def \bottom{B}
\def \interface{J}
\def \Qav{{\overline{Q}}{}}
\def \Qm{\Qav^\medium}
\def \Qf{\Qav^\fluid}

\def \Qout{\Qav^{\outlet}}
\def \Qin{Q^{\inlet}}
\def \QinC{\Qin_C}
\def \QinD{\Qin_D}
\def \Qbottom{\Qav^\bottom}
\def \Qdom{\Qav^\dom}
\def \Qg{Q^{G}}
\def \Qmm{Q^\medium}
\def \Qff{Q^\fluid}
\def \Dm{\Domainspace^\medium}
\def \Df{\Domainspace^\fluid}

\def \Din{\Domainspace^{I}}
\def \Dout{\Domainspace^{\outlet}}
\def \Dbottom{\Domainspace^\bottom}
\def \Dtop{\Domainspace^{T}}
\def \Dleft{\Domainspace^{L}}	
\def \Dright{\Domainspace^{R}}
\def \DInterface{\Domainspace^{\interface}}		
\def \DInterfaceL{\underline{\Domainspace}^{\interface}}
\def \DInterfaceU{{\overline{\Domainspace}}{}^{\interface}}		
\def \Ddom{\Domainspace^{\dom}}

\def \Rp{R^P}
\def \Rb{R^B}

\def \Gp{G^P}
\def \Gb{G^\bottom}
\def \Cav{\omatrix}
\def \OutputM{\Cav^{\medium}}
\def \OutputF{\Cav^{\fluid}}
\def \OutputOut{\Cav^{\outlet}}
\def \OutputBottom{\Cav^{\bottom}}
\def \Outputdom{\Cav^{\dom}}
\def \Ltwo{\mathcal{L}_2}
\def \selcrit{\mathcal{S}}  

\def \rhom{\rho^\medium}
\def \rhof{\rho^\fluid}
\def \kappam{\kappa^\medium}
\def \kappaf{\kappa^\fluid}
\def \cp{c_p}
\def \cpm{\cp^\medium}
\def \cpf{\cp^\fluid}
\def \am{a^\medium}
\def \af{a^\fluid}
\def \betam{\beta^\medium}

\newcommand{\mycaption}[1]{\caption{\footnotesize #1}}
\newcommand{\ncol}{q}
\newcommand{\vconst}{\overline{v}_0}

\newcommand{\heattransfer}{\lambda^{\!G}}
\newcommand{\Celsius}{{\text{\textdegree C}}}

\newcommand{\mat}[1]{{#1}}

\newcommand{\normalvec}{\mathfrak{n}}   

\newcommand{\dom}{\dagger}
\newcommand{\spacecontr}{\mathcal{Y}_C}
\newcommand{\spaceobs}{\mathcal{Y}_O}
\newcommand{\solfun}{\psi}
\newcommand{\textem}[1]{\textbf{#1}}

\journalname{Decisions in Economics and Finance}
\begin{document}

	\title{On the Input-Output Behavior of a Geothermal Energy Storage: Approximations by Model Order Reduction}
	\titlerunning{Approximating the Input-Output Behavior of a Geothermal Energy Storage} 

	\author{Paul Honore Takam  \and Ralf Wunderlich }
	
	\authorrunning{P.H.~Takam, R.~Wunderlich} 
	
	\institute{ Paul Honore Takam / Ralf Wunderlich \at
		Brandenburg University of Technology Cottbus-Senftenberg, Institute of Mathematics, P.O. Box 101344, 03013 Cottbus, Germany;  
		\email{\texttt{takam@b-tu.de} / \texttt{ralf.wunderlich@b-tu.de} }  
	}
	
	\date{}

	\maketitle
	
	\begin{abstract}{
			In this paper we consider a geothermal energy storage in which the  spatio-temporal temperature distribution  is modeled by a  heat equation with a convection term. Such storages often are embedded in residential heating systems and control and management  require the knowledge of some aggregated characteristics of that  temperature distribution in the storage. They describe the input-output behaviour of the storage and the associated energy flows and their response to charging and discharging processes. We aim  to derive an efficient approximative description  of these characteristics by a low-dimensional system of ODEs.  This leads  to a model order reduction problem for a large scale linear system of ODEs arising from the semi-discretization of the heat equation  combined with a linear algebraic output equation. In a first step we approximate the non time-invariant system of ODEs by a linear time-invariant  system. Then we apply Lyapunov balanced truncation model order reduction to approximate the output by a reduced-order system with only a few state equations but almost the same input-output behavior. The paper presents results of extensive numerical experiments showing the efficiency of the applied model order reduction methods. It turns out that only a few suitable chosen ODEs are sufficient to produce good approximations of the input-output behaviour of the storage.
		}
		\keywords{Geothermal energy storage\and Heat equation \and Large-scale systems \and Model order reduction \and Lyapunov balanced truncation  \and Gramians  }
		\subclass{93A15 
			\and 93B11 
			\and 93C05 
			\and 93C15 
			\and 37M99 
		}
	\end{abstract}
	
	\section{Introduction}
	\label{sec:intro}		
	Heating and cooling systems of single buildings as well as for district heating systems manage and mitigate temporal fluctuations of  heat supply and demand by using thermal storage facilities.  They  allow thermal energy  to be stored and to be used hours, days, weeks or months later. This is attractive  for space heating, domestic or process hot water production, or generating electricity. Note that  thermal energy may also be stored in the way of cold.  Thermal storages can significantly increase both the flexibility and the performance of district energy systems and enhancing the integration of intermittent  renewable energy 	sources into thermal networks (see Guelpa and  Verda \cite{guelpa2019thermal}, Kitapbayev et al.~\cite{KITAPBAYEV2015823}).

	Geothermal storages constitute an important class of thermal storages and enable an extremely efficient operation of heating and cooling systems in buildings. Further, they can be used  to mitigate peaks in the electricity grid by converting electrical into heat energy (power to heat). Pooling several geothermal storages within the  framework of a virtual power plant gives the necessary  capacity which allows to participate in the balancing energy market.
	
	\begin{figure}[!h]
		\centering
		\includegraphics[width=0.9\textwidth]{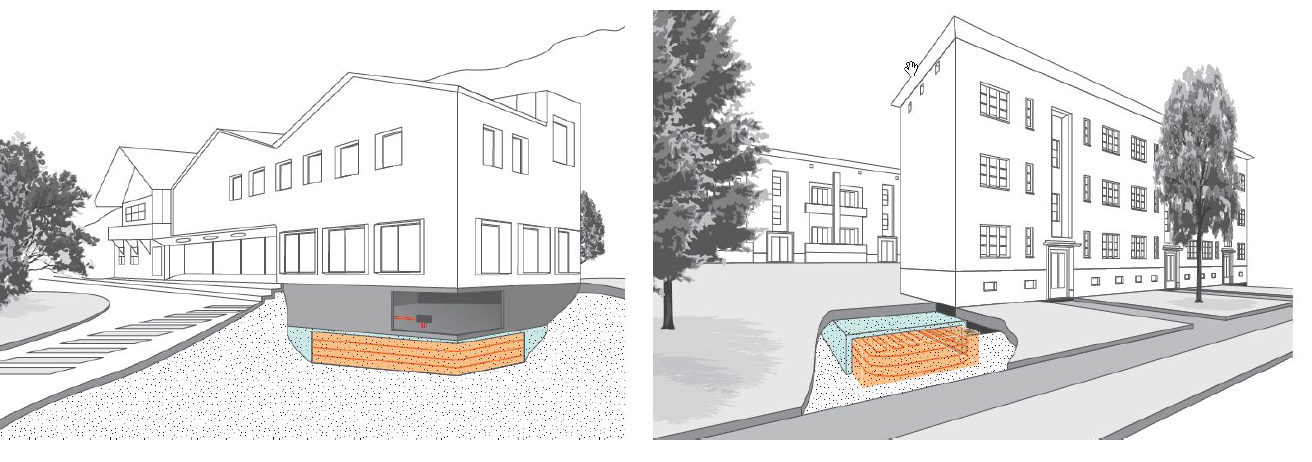}
		\caption[Geothermal storage]{Geothermal storage: in the new building, under a building (left) and in the renovation, aside of  the building (right), see
			\url{www.ezeit-ingenieure.eu}, \url{www.geo-ec.de}.}
		\label{fig:etank}
	\end{figure}

	The present  paper considers  geothermal storages as depicted in  Fig.\ref{fig:etank} and presented in detail in our companion papers \cite{takam2021shorta,takam2021shortb}. Such storages gain more and more importance and are quite attractive for residential heating systems since  construction and maintenance  are relatively inexpensive. Furthermore, they can be integrated both in new buildings and in renovations. For the storage a defined volume is filled with soil and insulated to the surrounding ground.  The thermal energy is stored by raising the temperature of the soil inside the storage. It is charged and discharged via pipe heat exchangers  (\phxk) filled with some fluid such as water. These \phxs are connected to a  residential heating system, in particular to internal storages such as water tanks, or directly to a solar collector. The fluid carrying the thermal energy is moved by pumps and heat pumps where the latter raise the temperature to a higher level using electrical energy.

	A special feature of the considered geothermal storage  is that it is not insulated at the bottom such that thermal energy can also flow into deeper layers   as it can be seen in Fig. \ref{etank_longtermsimu}. This allows for an  extension of the storage capacity since that heat can to some extent be retrieved if the storage is sufficiently discharged  (cooled) and a heat flux back to storage is induced. The unavoidable losses  due to the diffusion to the  environment can by compensated  since such  geothermal storages can benefit from higher temperatures in deeper layers of the ground  and therefore also serve as a production units  similar to  downhole heat exchangers. This becomes interesting during winter since in many regions in Europe the temperature in a depth of only 10 meter  over the year is almost constant around 10${}^\circ C$.	
	
	\begin{figure}[h!]
		\centering  
		\input{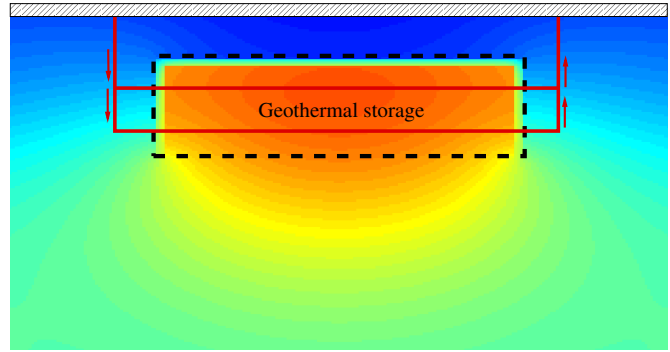}
		\caption{\label{etank_longtermsimu} 
			2D-model of a geothermal storage  insulated to the top and the sides (green bars) while open to the bottom and spatial temperature distribution.}
	\end{figure}
	
	For the operation of a geothermal storage within a residential system the controller or manager of that system need to know certain aggregated characteristics of the spatial temperature distribution in the storage, their dynamics and response to charging and discharging decisions. Note that the latter means to decide if the fluid is pumped through the \phxs to the storage or if it at rest and pumps are off. Further, if pumps are one one has to decide on an appropriate temperature of the fluid. 
	An example of such an aggregated characteristic is the average temperature in the storage medium from which one can derive the  amount of available thermal energy that can be stored in or extracted from the storage 
	Another example is the average temperature at the outlet which allows to determine  the amount of energy injected to or withdrawn from the storage. Further, the average temperature at the bottom of the storage allows to quantify the heat transfer to and from the ground via the open bottom boundary. For more details we refer to  Sec.~\ref{sec:Aggregate}.
	
	The above  aggregated characteristics can be computed by  postprocessing the spatio-temporal  temperature distribution in the storage. The latter can be obtained by solving the governing  linear heat equation with convection and appropriate boundary and interface conditions. This is explained in Secs.~\ref{Geothermal_S} and \ref{Discretization}. There   we work with a 2D model of the storage and use finite difference methods for the semi-discretization of the PDE w.r.t~the spatial variables.  That approach is also known as 'method of lines' and leads to a high-dimensional system of ODEs. We refer to our companion papers \cite{takam2021shorta,takam2021shortb} for details and an stability analysis of the  finite difference scheme  as well as results of extensive numerical experiments.
	Further, we refer to a previous study in  B\"ahr et al.~\cite{bahr2022efficient,bahr2017fast} where the storage was not considered isolated but embedded in the surrounding domain and the interaction between geothermal storage and the environment was studied. In that papers the focus was on the numerical simulation of the long-term behaviour of the spatial temperature distribution. For simplicity charging and discharging was described by a simple source term  but not by \phxsk. 
	The focus of the present paper  is  on the computation of the short-term behaviour of the spatial temperature distribution and its response to charging and discharging processes. We  choose the computational domain to be  the storage  depicted in Fig. \ref{etank_longtermsimu} by a dashed black rectangle. For the sake of simplicity we  do not consider the surrounding medium but  set appropriate boundary conditions to mimic the interaction between storage and  environment.  In addition to \cite{bahr2022efficient,bahr2017fast} we model \phxs which are used to charge and discharge the storage.
	
	The cost-optimal management of residential heating systems equipped with a geothermal storage  can be treated mathematically in terms of optimal control problems. This requires to model  the input-output behavior of the storage, i.e.~the dynamics of the above mentioned aggregated characteristics and their response to charging and discharging processes.  Working with the governing  PDE for the heat propagation or an approximating high-dimensional system of ODEs  becomes then intractable due to the curse of dimensionality.   Therefore one wishes  to describe the input-out behavior of the storage by a suitable low-dimensional system of ODEs with a sufficiently high approximation accuracy. This leads to a problem of model order reduction which is the focus in the present paper.

	In our model assumptions in Sec.~\ref{Geothermal_S} we restrict to the case of a piecewise constant velocity of the fluid in the \phxsk. This is often observed in real-world systems which operate with constant velocity during charging and discharging if pumps are on while the velocity is zero if pumps are off. Then the  high-dimensional system of ODEs constitutes a  system	of $n$  linear non-autonomous ODEs since the system matrices depend on time via the fluid velocity.  The latter varies over time and is only piecewise constant. Thus, the obtained linear system is not  linear time-invariant (LTI). The latter is a crucial assumption for many of model reduction methods.  In Sec.~\ref{sec:analog:system} we circumvent this problem by approximating the model for the geothermal storage by a so-called \emph{analogous model} which is LTI.
	The key idea for the construction of such an analogue is to  mimic the original model by a LTI system where pumps are always on such that  the fluid velocity is constant all the time. During the waiting periods we use at the inlet and outlet boundary the same type of boundary conditions  as during  charging and discharging. However,  we choose the inlet temperature to be equal to the average temperature in the \phxk. Numerical examples presented in \cite{takam2021shortb} show that the analogous system  approximates the original system quite well.
	
	For the derived linear LTI system we apply in Sec.~\ref{model_reduction} the Lyapunov balanced truncation model order  reduction method which is well suited for our  purposes.  It was first introduced by Mullis and Roberts \cite{mullis1976synthesis} and later in the linear systems and control literature by Moore \cite{moore1981principal}.   The idea of that method is first, to transform the system into an appropriate coordinate system for the state-space in which the states that are difficult to reach, that is, require a large input energy to be reached. They  are simultaneously difficult to observe, i.e., produce a small observation output energy. 
	Then, the reduced model is obtained by truncating the states which are simultaneously difficult to reach and to observe.
	Among the various model order reduction methods balanced truncation is characterized by the preservation of several system properties like stability and passivity, see Pernebo and Silverman~\cite{pernebo1982model}. Further, it provides error bounds that permit an appropriate choice of the dimension of the reduced-order model depending on the desired accuracy of the approximation, see Enns~\cite{enns1984model}. 
	
	Besides the Lyapunov balancing method, there exist other types of balancing techniques such as stochastic balancing,  bounded real balancing, positive real balancing and frequency weighted balancing, see Antoulas \cite{antoulas2005approximation} and Gugercin and  Antoulas~\cite{gugercin2004survey}. 
	Gosea et al. \cite{gosea2018balanced} considers balanced truncation for linear switched systems.
	In the book of Benner et al.~\cite{benner2005dimension}, an efficient implementation of model reduction methods such as modal truncation, balanced truncation, and other balancing-related truncation techniques is presented. Further, the authors discussed various aspects of balancing-related techniques for large-scale systems, structured systems, and descriptor systems. The results presented in~\cite{benner2005dimension} also cover the model reduction techniques for time-varying as well as the model reduction for second- and higher-order systems, which can be considered as one of the major research directions in dimension reduction for linear systems. In addition, surveys on system approximation and model reduction can be found in \cite{amsallem2012stabilization,antoulas2005approximation,benner2000balanced,besselink2013comparison,freund2000krylov,gugercin2004survey,kurschner2018balanced,mehrmann2005balanced,redmann2016balancing,volkwein2013proper} and the references therein. 
	
	\smallskip
	The rest of the paper is organized as follow. In Sec.~\ref{Geothermal_S} we describe 
	the mathematical modeling of the geothermal storage. We present the   heat equation with a convection term and appropriate boundary and interface conditions which governs the dynamics of the spatial temperature distribution in the  storage.  Sec.~\ref{Discretization} is devoted to  the finite difference semi-discretization of the heat equation. In Sec.~\ref{sec:Aggregate} we introduce aggregated characteristics of the spatio-temporal temperature distribution.  Sec.~\ref{sec:analog:system} derives the approximate LTI analogous model of the geothermal storage.  
	In Sec.~\ref{model_reduction} we start with  the formulation of the general model reduction problem. Then we  present the  Lyapunov balanced truncation  method. 
	Sec.~\ref{Numerical_exp} presents results of various numerical experiments where  the aggregated characteristics of the temperature distribution in storage for the original model are compared with the approximations obtained from reduced-order models. In Sec.~\ref{conclu} we present some conclusions. An appendix provides a list of frequently used notations and some proofs  which were removed from the main text.

	\section{Dynamics of the Geothermal Storage}
	\label{Geothermal_S}
	The setting is based  on our companion paper \cite[Sec.~2]{takam2021shorta}. For self-containedness  and the convenience of the reader,  we recall in this section the description of the model.
	The dynamics of the  spatial temperature distribution in a geothermal storage can be described mathematically by a linear heat equation with  convection term and appropriate boundary and interface conditions.    
	
	\subsection{2D-Model}
	\label{model2D}
	
	We assume that the domain of the geothermal storage is a cuboid and consider a two-dimensional rectangular cross-section.
	We denote by  $Q=Q(t,x,y)$ the temperature at time $t \in [0,T]$ in the point $(x,y)\in \Domainspace=(0,l_x) \times (0,l_y)$ with $l_x,l_y$ denoting the width and height of the storage. 
	The domain $\Domainspace$ and its boundary $\partial \Domainspace$ are  depicted in Fig.~\ref{bound_cond}.  $\Domainspace$ is divided into three parts. The first is $\Dm$ and filled with a homogeneous medium (soil) characterized by material parameters $\rhom, \kappam$ and $\cpm$ denoting  mass density,    thermal conductivity and   specific heat capacity, respectively. The second is $\Df$, it  represents the\phxs and is filled with a fluid (water) with constant material parameters $\rhof, \kappaf$ and $\cpf$. The fluid moves with time-dependent velocity $v_0(t)$ along the \phxs. For the sake of simplicity we  restrict to the case, often observed in applications, where the pumps moving the fluid are either on or off. Thus the velocity $v_0(t)$ is piecewise constant taking  values  $\vconst>0$ and zero, only.  Finally, the third part is the interface $\DInterface$ between $\Dm$ and $\Df$. 
	For the sake of simplicity  we neglect modeling the wall of the \phx and suppose perfect contact between the \phx and the soil. Details are given below in \eqref{Interface} and \eqref{eq: 13f}. Summarizing we make the following
	\begin{figure}[h!]		
		\begin{center}
			
			\includegraphics[width=0.8\linewidth,height=.6\linewidth]{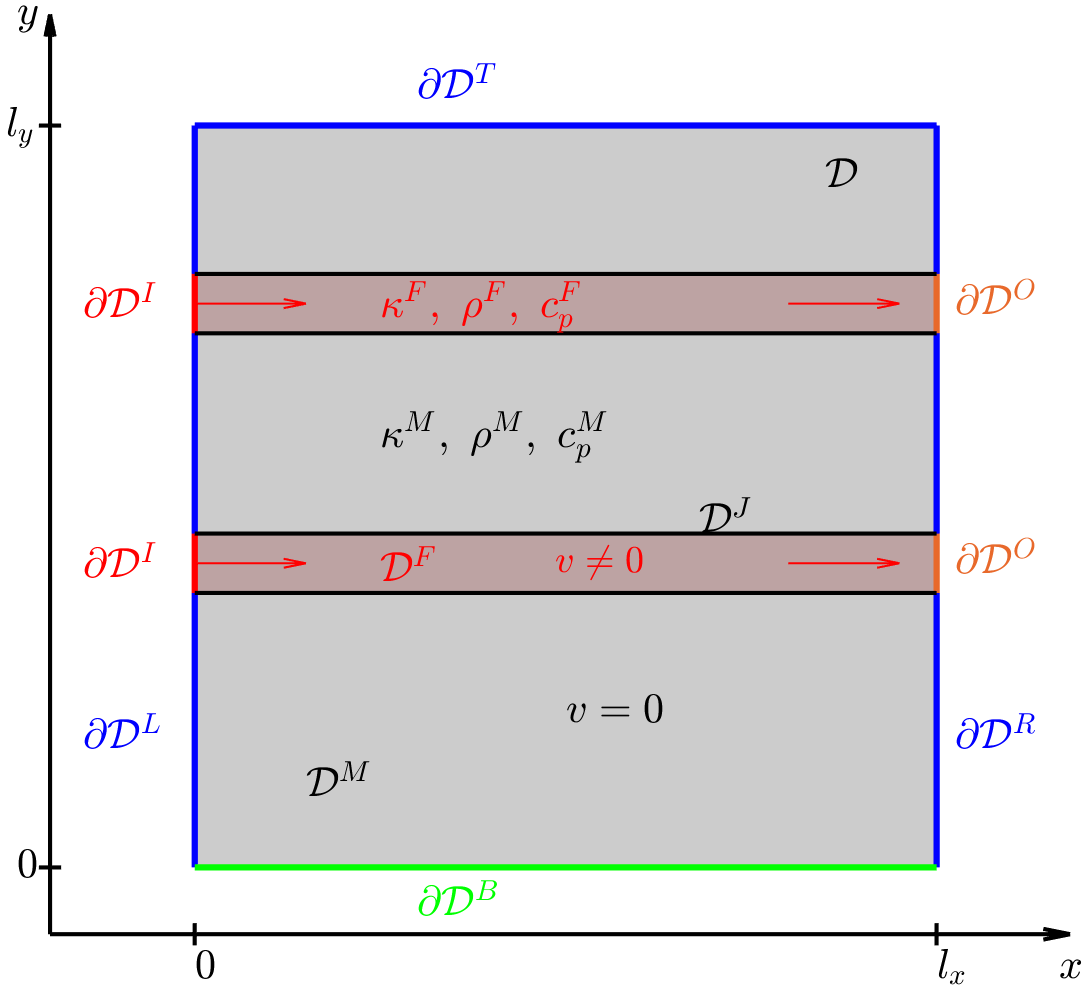}
		\end{center}
		\mycaption{\label{bound_cond} 2D-model of the geothermal storage:  decomposition of the domain $\Domainspace$ and the boundary $\partial \Domainspace $.
		}			
	\end{figure}
	\begin{assumption}
		\label{assum1}~
		\begin{itemize}
			\item [1.] Material parameters of the medium   $\rhom, \kappam, \cpm$ in the domain $\Dm$  and of the fluid  $\rhof, \kappaf, \cpf$ in the domain  $\Df$ are constants.
			\item [2.] Fluid velocity is piecewise constant, i.e. $v_0(t)=\begin{cases}
				\vconst>0,  &\text{pump~on}\\
				0, & \text{pump off}
			\end{cases}$
			\item [3.] Perfect contact at the interface between fluid and  medium.
		\end{itemize}
	\end{assumption}
	
	\begin{remark}\label{rem:3D} Results obtained for our 2D-model, where $\mathcal{D}$ represents the rectangular cross-section of a box-shaped storage can be extended to the 3D-case if we assume that the 3D storage domain is a cuboid of depth $l_z$ with an homogeneous temperature distribution in $z$-direction. 
		A \phx in  the 2D-model then represents a horizontal snake-shaped \phx densely filling a small layer of the storage. 
	\end{remark}

	\paragraph{Heat equation}
	The temperature $Q=Q(t,x,y)$ in the external storage is governed by the linear heat equation with convection term
	\begin{align}
		\rho \cp \frac{\partial Q}{\partial t}={\nabla \cdot (\kappa \nabla Q)}-
		{\rho v \cdot \nabla (\cp Q)},\quad  (t,x,y) \in (0,T]\times \Domainspace \setminus \DInterface,   \label{heat_eq}
	\end{align}
	where the first term  on the right hand side describes diffusion while the second represents convection of the moving fluid in the \phxsk.
	Further,
	$v=v(t,x,y)$ $=v_0(t)(v^x(x,y),v^y(x,y))^{\top}$  denotes  the velocity vector with $(v^x,v^y)^\top$ being the normalized directional vector of the flow. 
	According to Assumption \ref{assum1} the material parameters $\rho,\kappa, \cp$  
	depend on the position $(x,y)$ and take the values $\rhom,\kappam, \cpm$ for points in $\Dm$ (medium) and  $\rhof, \kappaf, \cpf$ in  $\Df$ (fluid).
	
	Note that there are no sources or sinks inside the storage and therefore the above heat equation appears without forcing term.
	Based on this assumption, the heat equation (\ref{heat_eq}) can be written as
	\begin{align}
		\frac{\partial Q}{\partial t}=a\Delta Q-
		{ v \cdot \nabla Q},\quad  (t,x,y) \in (0,T]\times \Domainspace \setminus \DInterface,   \label{heat_eq2}
	\end{align}
	where $\Delta=\frac{\partial^2}{\partial x^2}+\frac{\partial^2}{\partial y^2}$ is the Laplace operator, $\nabla=\big(\frac{\partial}{\partial x},\frac{\partial}{\partial y}\big)$ the gradient operator, and  $ a=a(x,y)$ is the thermal diffusivity which is piecewise constant with values  $a^\dom=\frac{\kappa^\dom}{\rho^\dom \cp^\dom}$  with $\dom=\medium$ for  $(x,y)\in \Dm $ and $\dom=\fluid$   for  $(x,y)\in \Df$, respectively.   
	The initial condition $Q(0,x,y)=Q_0(x,y)$ is given by the initial temperature distribution $Q_0$ of the storage.

	\subsection{Boundary { and Interface} Conditions}
	For the description of the boundary conditions we decompose the boundary $\partial\Domainspace$ into several subsets as depicted in Fig.~\ref{bound_cond} representing the insulation on the top and the side, the open bottom, the inlet and outlet of the \phxsk.  Further, we have to specify conditions at the interface between \phxs and  soil. The inlet, outlet and the interface  conditions model the heating and cooling of the storage via \phxsk. We distinguish between the two regimes 'pump on' and 'pump off' where for simplicity we assume perfect insulation at inlet and outlet if the pump is off.
	This leads to the following boundary conditions.
	
	\begin{itemize}				
		\item \textit{Homogeneous Neumann condition} describing perfect insulation on the top  and the side 
		\begin{align}\frac{\partial Q}{\partial \normalvec}=0,\qquad (x,y)\in 
			\partial \Dtop\cup \partial \Dleft \cup \partial \Dright, 
			\label{Neumann}
		\end{align}
		where $\partial \Dleft=\{0\} \times  [0,l_y] \backslash \partial \Din$, ~
		$\partial \Dright=\{l_x\} \times  [0,l_y] \backslash \partial \Dout, \partial\Dtop=[0,l_x] \times \{l_y\}$  and $\normalvec$ denotes the outer-pointing normal vector.
		\item \textit{Robin condition} describing heat transfer on the bottom 
		\begin{align}
			-\kappam\frac{\partial Q}{\partial \normalvec}=\heattransfer(Q-\Qg(t)), \qquad  (x,y)\in 
			\partial \Dbottom,
			\label{Robin}
		\end{align}
		with $\partial \Dbottom=(0,l_x) \times \{0\}$, where $\heattransfer>0$ denotes the  heat transfer coefficient  and $\Qg(t)$ the underground temperature.	
		\item  \textit{Dirichlet condition} at the inlet if the pump is on ($v_0(t)>0$), i.e.~the fluid arrives at the storage with a given temperature $\Qin(t)$. If the pump is off ($v_0(t)=0$), we set a homogeneous Neumann condition describing perfect insulation.
		\begin{align}
			\begin{cases}
				\begin{array}{rll}	
					Q &=\Qin(t), &\text{ ~pump on,} \\
					\frac{\partial Q}{\partial \normalvec}&=0, &\text{ ~pump off,} 
				\end{array}
			\end{cases} 
			\qquad  (x,y)\in 	\partial \Din.
			\label{input}
		\end{align}
		
		\item \textit{``Do Nothing'' condition} at the outlet in the following sense. If  the  pump is on ($v_0(t)>0$) then the total heat flux directed outwards can be decomposed into a diffusive  heat flux given by $k^{f}\frac{\partial Q}{\partial \normalvec}$ and a convective  heat flux given by $v_0(t) \rhof \cpf Q$.  In our model we can neglect the diffusive heat flux. This leads to a homogeneous Neumann condition
		\begin{align}\frac{\partial Q}{\partial \normalvec}=0,\qquad (x,y)\in 
			\partial \Dout. 
			\label{output}
		\end{align}
		If the pump is off then we  assume  perfect insulation which is also described by the above condition.	
		
		\item \textit{Smooth heat flux} at interface $\DInterface$ between fluid and soil leading to a coupling condition
		\begin{align}
			\kappaf\bigg(\frac{\partial \Qff}{\partial \normalvec }\bigg)=\kappam\bigg(\frac{\partial \Qmm}{\partial \normalvec }\bigg),
			\qquad  (x,y)\in \DInterface.
			\label{Interface}
		\end{align}	
		Here, $\Qff, \Qmm$ denote the temperature of the fluid inside the \phx and of the soil outside the \phxk, respectively.
		Moreover, we assume that the contact between the \phx and the medium is perfect which leads to a smooth transition of a temperature, i.e., we have 
		\begin{align}
			\Qff=\Qmm,\qquad  (x,y)\in \DInterface. \label{eq: 13f}
		\end{align} 
	\end{itemize}
	
	\section{Semi-Discretization of the Heat Equation}
	\label{Discretization}
	We now sketch  the discretization of the  heat equation \eqref{heat_eq2} together with the boundary and interface conditions given in \eqref{Neumann} through \eqref{eq: 13f}. For details we refer to our companion paper \cite[Sec.~3]{takam2021shorta}.
	We confine ourselves to a  semi-discretization in space and approximate only spatial derivatives by their respective finite differences.  This approach is also known as 'method of lines' and leads  to a high-dimensional system of ODEs  for the temperatures at the grid points. The latter will serve as starting point for the model reduction in Sec.~\ref{model_reduction}.   
	For the full discretization in which time is also  discretized we refer to our companion paper \cite[Sec.~4]{takam2021shortb} where we  derive an implicit finite difference scheme and study its stability.

	\subsection{Semi-Discretization of the Heat Equation}
	\label{Semi-Discret}
	\begin{figure}[h!]
		\centering
		
		\begin{center}
			\resizebox{0.7\textwidth}{0.5\textwidth}{%
				\begin{tikzpicture}[thick,scale=1, every node/.style={scale=1}]
					\draw[thick,->] (-1.5,-1.5) -- (8.7,-1.5) node[anchor=north west] {x};
					\draw[thick,->] (-1.5,-1.5) -- (-1.5,8.3) node[anchor=south east] {y};
					\draw[step=1.5cm,black,very thin] (-1.5,-1.5) grid (7.5,7.5);
					\node at (3,3) {$\bullet$};
					\node at (3.5,2.7) {$(i,j)$};
					\node at (3,4.5) {$\bullet$};
					\node at (3.8,4.8) {$(i,j+1)$};
					\node at (3,1.5) {$\bullet$};
					\node at (3.8,1.2) {$(i,j-1)$};
					\node at (4.5,3) {$\bullet$};
					\node at (5.3,2.7) {$(i+1,j)$};
					\node at (1.5,3) {$\bullet$};
					\node at (2.2,2.7) {$(i-1,j)$};
					\node at (6.9,-1.8) {$l_x=N_xh_x$};
					\node at (-2.4,7.5) {$l_y=N_yh_y$};
					\node at (7.5,7.5) {$\bullet$};
					\node at (6.8,7.1) {$(N_x,N_y)$};
					\node at (-1.5,-1.5) {$\bullet$};
					\node at (-1.0,-1.3) {$(0,0)$};	
					\node at (-1.5,7.5) {$\bullet$};
					\node at (-0.9, 7.2) {$(0,N_y)$};
					\node at (7.5,-1.5) {$\bullet$};
					\node at (6.9,-1.3) {$(N_x,0)$};
				\end{tikzpicture}
			}  
		\end{center}
		\mycaption{\label{grid}Computational grid.}	
	\end{figure}
	
	The spatial domain depicted in Fig.~\ref{bound_cond} is discretized by the means of a mesh   with grid points  $(x_i,y_j)$ as in Fig.~\ref{grid} where
	$	x_i =ih_x, ~~~y_j =jh_y,,\quad i ={ 0},...,N_x, ~~~j ={ 0},...,N_y.$
	Here,  $N_x$ and $N_y$ denote the  the number of grid points  while   $h_x={l_x}/{N_x}$ and $h_y={l_y}/{N_y}$  are the step sizes in $x$ and $y$-direction, respectively. 
	We denote by $Q_{ij}(t)\simeq Q(t,x_i,y_j)$ the semi-discrete approximation  of the temperature $Q$   and by $v_0(t)(v^x_{ij},v^y_{ij})^{\top} =v_0(t)(v^x(x_i,y_j),v^y(x_i,y_j))^{\top} =v(t,x_i,y_j)$  the velocity vector  at the grid point $(x_i,y_j)$ at time $t$.

	For the sake of simplification and tractability of our analysis we restrict to the following assumption on the arrangement of \phxs and impose  conditions on the location of grid points  along the \phxsk.
	\begin{assumption}\label{assum2}~%
		\begin{enumerate}
			\item  There are $n_P \in \N$  straight horizontal \phxsk, the fluid moves in positive $x$-direction.
			\item The diameter of the \phxs are such that the interior of \phxs contains grid points.
			\item Each interface between medium and fluid contains grid points.
		\end{enumerate}		
	\end{assumption}	
	
	We approximate the spatial derivatives in the heat equation, the boundary and interface conditions by finite differences as in \cite[Subsec.~3.1--3]{takam2021shorta} where we apply upwind techniques for the convection terms. The result is the system of ODEs \eqref{Matrix_form1} (given below) for a vector function $Y:[0,T]\to \R^n$ collecting  the semi-discrete approximations $Q_{ij}(t)$ of the temperature $Q(t,x_i,y_j)$ in the ``inner'' grid points, i.e., all grid points except those on the boundary $\partial \mathcal{D}$ and the interface $\DInterface$. 
	For a model with $n_P$ \phxs the dimension of $Y$ is   $n=(N_x-1)(N_y-2n_P-1)$, see \cite{takam2021shorta}.

	Using the above notation the semi-discretized  heat equation \eqref{heat_eq2} together with the given initial, boundary and interface conditions reads as
	\begin{align}
		\dot Y(t)= \mat{A}(t)Y(t)+\mat{B}(t)g(t), ~~t \in (0,T],
		\label{Matrix_form1}
	\end{align}
	with the initial condition $Y(0)=y_0$ where the vector $y_0\in \R^n$ contains the initial temperatures $Q_0(\cdot,\cdot)$ at the corresponding grid points.  
	The system matrix  $\mat{A}$  results from the spatial discretization of the convection and diffusion term in the heat equation (\ref{heat_eq2}) together with the Robin and linear heat flux boundary conditions. It has the tridiagonal structure  
	\begin{align}
		\label{matrix_A}
		\mat{A}=\begin{pmatrix}
			\mat{A}_{L} ~&~ \mat{D}^{+} ~& &&&\text{\LARGE0} \\
			\mat{D}^{-} ~&~ \mat{A}_{M} ~&~ \mat{D}^{+} \\
			& \mat{D}^- ~&~ \mat{A}_{M} ~&~ \mat{D}^{+} \\
			&& \ddots &\ddots & \ddots &\\
			& && \mat{D}^- ~&~ \mat{A}_{M}~ &~ \mat{D}^+ \\
			\text{\LARGE0}& & &&~ \mat{D}^- ~&~ \mat{A}_{R}
		\end{pmatrix}
	\end{align}
	and consists of $(N_x-1)\times( N_x-1)$ block matrices of dimension $\ncol=N_y-2n_P-1$.
	The block matrices $\mat{A}_{L},\mat{A}_{M},\mat{A}_{R}$ on the diagonal have a tridiagonal structure and are given
	in \cite[Table 3.1 and B.1]{takam2021shorta}. 		
	The block matrices on the subdiagonals $\mat{D}^{\pm} \in \R^{\ncol \times \ncol}$, $i=1,\ldots, N_x-1$, are  diagonal matrices and given in \cite[Eq.~(3.12)]{takam2021shorta}.

	As a result of the discretization of the  Dirichlet  condition at the inlet boundary and the Robin condition at the bottom boundary, we get the  function  $g:~[0,T] \to \R^2$ called input function and the  $n\times 2$ input matrix $\mat{B}$  called input matrix. The entries of the input matrix  $B_{l\neu r}, ~l=1,\ldots,n,~~ r=1,2,$ are derived in \cite[Subsec 3.4]{takam2021shorta} and are given by 
	\begin{align}\label{eq:input_matrix}
		\begin{array}{rl@{\hspace*{2em}}l}
			B_{l1}&=B_{l1}(t)=\begin{cases}
				\frac{\af}{h^2_x}+\frac{\vconst}{h_x}  & \text{pump on,}\\
				0 & \text{pump off,}
			\end{cases}
			& l=\mathcal{K}(1,j), (x_0,y_j)\in\Din,\\[3ex]	
			B_{l2}& = \frac{\heattransfer h_y}{\kappam+\heattransfer h_y}\betam,
			&  l=\mathcal{K}(i,1), (x_i,y_0)\in\Dbottom,
		\end{array}		
	\end{align}   
	with $\beta^\medium=\am/{h^2_y}$.
	The entries for other $l$ are zero. Here,  $\mathcal{K}$ denotes the mapping   $(i,j)\mapsto  l=\mathcal{K}(i,j)$  of pairs of indices   of  grid point $(x_i,y_j) \in\mathcal{D}$ to the single index $l \in \{1,\ldots,n \}$ of the corresponding entry in the vector $Y$. The input function  reads as
	\begin{align}
		g(t)=\begin{cases}
			(\Qin(t),~\Qg(t))^{\top}, & \quad \text{pump on},\\
			~~~~~(0,~~~~\Qg(t))^{\top}, & \quad \text{pump off}.
		\end{cases}
		\label{eq:input}
	\end{align}
	Recall that  $\Qin$ is the inlet temperature  of the pipe during pumping  and $\Qg$ is the underground temperature.

	\subsection{Stability of Matrix $\mat{A}$}
	\label{stability_A}	
	The finite difference  semi-discretization of the heat equation \eqref{heat_eq2} given by the  system of ODEs \eqref{Matrix_form1} is expected to preserve the dissipativity of the PDE. This property is related to the stability of the system matrix $\mat{A}=\mat{A}(t)$  in the sense that all eigenvalues  of $\mat{A}$ lie in left open complex half plane. That property will play a crucial role below in Sec.~\ref{model_reduction} where we study model reduction techniques for \eqref{Matrix_form1} based on balanced truncation.	The next theorem confirms the expectations on the stability of $\mat{A}$.	For the proof we refer to our companion paper \cite[Theorem 3.3]{takam2021shorta}.
	\begin{theorem}[Stability of Matrix $\mat{A}$]	\label{stable_m} \ \\
		Under Assumption \ref{assum1} on the model and Assumption \ref{assum2} on the discretization,  
		the matrix $\mat{A}=\mat{A}(t)$ given in \eqref{matrix_A} is stable for all $t\in[0,T]$, i.e.,  all eigenvalues $\lambda (\mat{A})$ of $\mat{A}$ lie in   left open complex half plane.  		
	\end{theorem}
	
	\section{Aggregated Characteristics}	
	\label{sec:Aggregate}
	The numerical methods introduced in Sec.~\ref{Discretization}	 allow the approximate  computation of  the  spatio-temporal temperature distribution in the geothermal storage. In many applications it is  not necessary to know  the complete information about that distribution. An example is the management and control of a storage which is embedded into a residential heating system. Here it is sufficient to know  only a few aggregated characteristics of the temperature distribution which can be computed via  post-processing. In this section we introduce some of these aggregated characteristics and describe their approximate computation  based on the solution vector $Y$ of the finite difference scheme.
	
	\subsection{Aggregated Characteristics Related to the Amount of Stored Energy}
	
	We start with aggregated characteristics given by the average temperature in some given subdomain of the storage which are related to the amount of stored energy in that domain. 
	
	Let  $\mathcal{B}\subset \mathcal{D}$ be a  generic  subset of the 2D computational domain. We denote by $|\mathcal{B}|=\iint_{\mathcal{B}} dxdy$ the area of $\mathcal{B}$. 
	Then
	$W_{\mathcal{B}}(t)=l_z \iint_{\mathcal{B}} \rho \cp Q(t,x,y) dxdy$ represents the thermal energy contained in the 3D spatial domain $\mathcal{B}\times[0,l_z]$ at time $t\in[0,T]$. Then for $0\le t_0<t_1\le T$ the difference $G_{\mathcal{B}}(t_0,t_1)=W_{\mathcal{B}}(t_1)-W_{\mathcal{B}}(t_0)$ is the gain of thermal energy during the period $[t_0,t_1]$. While positive values correspond to warming of $\mathcal{B}$,  negative values indicate cooling and  $-G_{\mathcal{B}}(t_0,t_1)$ represents the size of the  loss of thermal energy. 
	
	For $\mathcal{B}=\Ddom, \dom=\medium,\fluid$, we can use that the material parameters on $\mathcal{D}^\dom$ equal  the constants  $\rho=\rho^\dom,\cp=\cp^\dom$. Thus,  for the corresponding gain of thermal energy we obtain
	\begin{align}
		\nonumber
		G^\dom=G^\dom(t_0,t_1)& := G_{\Ddom}(t_0,t_1) = \rho^\dom \cp^\dom|\Ddom| l_z~(\Qdom(t_1)-\Qdom(t_0)),\\
		\text{where}\quad 
		\Qdom(t) &= \frac{1}{|\Ddom|} \iint_{\Ddom} Q(t,x,y) dxdy, \quad \dom=\medium,\fluid,
		\label{eq:av:medium:fluid}	
	\end{align}
	denotes the  average temperature in the medium ($\dom=\medium$) and the fluid  ($\dom=\fluid$), respectively.

	\subsection{Aggregated Characteristics Related to the Heat Flux at the Boundary}
	\label{subsec:AggCharBoundary}
	Now we consider the convective heat flux at the inlet and outlet boundary and the  heat transfer at the bottom boundary. 
	Let  $\mathcal{C}\subset \mathcal{\partial D} $ be a generic curve on the boundary,  then  we denote by $|\mathcal{C}|=\int_{\mathcal{C}} ds$ the curve length. 
	
	The rate  at which the energy is injected or withdrawn via the \phx is given by 
	\begin{align}
		\nonumber	
		\Rp(t)&=\rhof \cpf v_0(t) \Big[\int_{\Din}  Q(t,x,y)\, ds  -\int_{\Dout} Q(t,x,y)\, ds\Big]\\
		\label{Rout}	
		&=\rhof \cpf v_0(t) |\partial  \Dout|[\Qin(t) -\Qout(t)],\\
		\nonumber	
		\text{where}\quad 
		\Qout(t) &=	 \frac{1}{|\partial \Dout|} \int_{\partial \Dout} Q(t,x,y)ds
	\end{align}
	is the average temperature at the outlet boundary. Here, it is used that in our model we have  horizontal \phxs such that  $|\partial \Din|=|\partial \Dout|$ and a uniformly distributed  inlet temperature at the inlet boundary $\partial \Din$. For a given interval of time $[t_0,t_1]$ the quantity 
	$$\Gp=\Gp(t_0,t_1) =l_z\int_{t_0}^{t_1} \Rp(t)\, dt$$ 
	describes the amount of heat injected  ($G^P>0$)  to or withdrawn ($G^P<0$) from the storage due to convection of the   fluid. Note that the fluid moves at time $t$ with velocity $v_0(t)$ and arrives at the inlet with temperature $\Qin(t)$ while it leaves at the outlet with  the average temperature $\Qout(t)$.
	
	Next we look at the  diffusive heat transfer via the bottom boundary and define the rate 
	\begin{align}
		\nonumber
		\Rb(t)&=  \int_{\Dbottom}\kappam\frac{\partial Q}{\partial \normalvec}\, ds = \int_{\Dbottom}\heattransfer(\Qg(t)-Q(t,x,y))\, ds	\\[0.5ex]
		\label{RB}		
		& = \heattransfer |\partial  \Dbottom| (\Qg(t)- \Qbottom(t) ),\\
		\nonumber	
		\text{where}\quad 
		\Qbottom(t) &=	 \frac{1}{|\partial \Dbottom|} \int_{\partial \Dbottom} Q(t,x,y)ds,
	\end{align}
	is the average temperature at the bottom boundary.
	Note that the second equation in the first line  follows from the Robin boundary condition.
	The quantity $$\Gb=\Gb(t_0,t_1) =l_z\int_{t_0}^{t_1} \Rb(t)\, dt$$ describes the amount of  heat  transferred via the bottom boundary of the storage.

	\subsection{Numerical  Computation of Aggregated Characteristics}	 
	For the numerical simulation of the  aggregated characteristics introduced in the previous subsections these quantities have to be expressed in terms of the   finite difference  approximations of the temperature $Q=Q(t,x,y)$. Then one obtains approximations in terms of  the entries of the vector function $Y(t)$ satisfying  the system of ODEs \eqref{Matrix_form1} and containing the semi-discrete finite difference approximations of the temperature in the inner grid points of the computational domain $\mathcal{D}$. Recall that  the temperatures  on boundary and interface grid points can be determined by  linear combinations from the entries of $Y(t)$. 
	
	Let $\Qdom(t)$ be one of the average temperatures $\Qm,\Qf,\Qout,\Qbottom$. Then the numerical approximation of the defining single and double integrals by quadrature rules leads to approximations by linear combinations of the entries of $Y$  of the form 
	\begin{align}
		\label{aggr_char}
		\Qdom(t) & \approx    \Outputdom\,Y(t)
	\end{align}
	where $\Outputdom$ is some $1\times n$-matrix. 		
	For the details we refer to  our companion paper  \cite[Subsection 4.2, Appendix B]{takam2021shortb}.

	\section{Analogous Linear Time-Invariant  System}
	\label{sec:analog:system}
	Eq.~\eqref{Matrix_form1} represents a  system	of $n$  linear non-autonomous ODEs. Since some of the entries in the  matrices $\mat{A}$ and $\mat{B}$ resulting from the discretization of convection terms in the heat equation \eqref{heat_eq2} depend on the velocity $v_0(t)$, it follows that  $\mat{A}$ and $\mat{B}$  are time-dependent.
	Thus, \eqref{Matrix_form1} does not constitute a linear time-invariant (LTI) system. The latter is a crucial assumption for many of model reduction methods such as the Lyapunov balanced truncation technique that is  considered below in Sec.~\ref{model_reduction}.  We circumvent this problem by  replacing the model for the geothermal storage by a so-called \emph{analogous model} which is LTI. 
	
	The key idea for the construction of such an analogue is based on the observation that under the assumption of this paper our ``original model'' is already piecewise LTI. This is due to our assumption that the  fluid velocity is constant $\vconst$ during (dis)charging when the pump is on, and zero during waiting when the pump is off. This leads to the following  approximation of the original by an analogous model which is performed in two steps.
	
	\paragraph{Approximation Step 1}
	For the analogous model we assume that contrary to the original model the fluid is also moving with constant velocity $\vconst$ during pump-off periods. 
	During these waiting periods   in the original model the fluid is at rest and only subject to the diffusive propagation of heat.   In order to mimic that behavior of the resting fluid by a moving fluid we assume that the temperature $\Qin$ at the \phx inlet is equal to the average temperature of the fluid in the \phx $\Qf$. From a physical point of  view we will preserve the average temperature of the fluid but a potential  temperature gradient along the \phx is  not preserved  and replaced by an almost flat temperature distribution. It can be expected that the error induced by this ``mixing'' of the fluid temperature in the \phx is small after sufficiently long (dis)charging periods leading to saturation with an almost constant temperature along the \phxk.

	In the mathematical description by an initial boundary value problem for the heat equation \eqref{heat_eq2}, the above approximation leads to a modified boundary condition at the inlet. During waiting  the homogeneous Neumann boundary condition in \eqref{input} is replaced  by a non-local coupling condition such that the inlet boundary condition reads as 
	\begin{align}
		Q=
		\begin{cases}
			\begin{array}{ll}	
				\Qin(t), &\text{ ~pump on,} \\
				\Qf(t), &\text{ ~pump off,} 
			\end{array}
		\end{cases} 
		\qquad  (x,y)\in 	\partial \Din.
		\label{input_analog}
	\end{align}
	That condition is termed 'non-local' since the inlet temperature is not only specified by a  condition to the local temperature distribution at the inlet boundary $\partial \Din$ but it depends on the whole  spatial temperature distribution  in the fluid domain $\Df$. Semi-discretization of the above boundary condition using approximation  \eqref{eq:av:medium:fluid}	 of  the average fluid temperature $\Qf$ formally leads to a modification of the   input term $g(t)$ of  the system of ODEs \eqref{Matrix_form1} given in \eqref{eq:input}. That input term now reads as
	\begin{align}
		g(t)=\begin{cases}
			(\Qin(t),~\Qg(t))^{\top}, & \quad \text{pump on},\\
			~~(\OutputF \,Y,~~\Qg(t))^{\top}, & \quad \text{pump off}.
		\end{cases}
		\label{eq:input_analog}
	\end{align}
	Further, the non-zero entries  $B_{l1}$ of the input matrix $B$ given in \eqref{eq:input_matrix} are modified. They are now no longer time-dependent but given by the constant
	$		B_{l1}=\frac{\af}{h^2_x}+\frac{\vconst}{h_x} $ which was already used during pump-on periods.
	
	\paragraph{Approximation Step 2}
	From \eqref{eq:input_analog} it can be seen that the input term $g$  during pumping depends on the state vector $Y$ via $\OutputF \,Y$ and can no longer considered as exogenous.  Formally, the term  $\OutputF \,Y$ has to be included in $\mat{A}Y$ which would  lead to an additional contribution  to the system matrix  $\mat{A}$ given by $\mat{B}_{\bullet1}\OutputF$ where $\mat{B}_{\bullet1}$ denotes the first column of $\mat{B}$. Thus, the system matrix again would be time-dependent and the system not LTI. For the application of  model reduction methods   we therefore perform a second approximation step and treat  $\Qf$ as an exogenously given quantity  (such as $\Qin$ and $\Qg$). This leads to a tractable approach since the Lyapunov balanced truncation technique applied in the next section generates low-dimensional systems depending only on the system matrices $\mat{A,B}$ but not on the input term $g$. Further, from an algorithmic or implementation point of view this is not a problem since given the solution $Y$ of \eqref{Matrix_form1} at time $t$, the average fluid temperature $\Qf(t)$  can be computed as a linear combination of the entries of $Y(t)$. 
	
	\smallskip
	In our companion paper \cite[Sec.~6]{takam2021shortb} we present results of numerical experiments indicating that apart  from some small approximation errors in the \phx during waiting periods, in particular at the outlet, the other deviations are negligible. They also show  that during the (dis)charging periods the errors decrease and  vanish almost completely, i.e., in the long run there is no accumulation of errors.

	\section{Model Order Reduction}
	\label{model_reduction}
	\subsection{Problem}
	\label{MOP_intro}	
	In the previous sections we have seen that   the spatio-temporal temperature distribution describing the input-output behavior of the geothermal storage can be approximately computed by solving the system of ODEs \eqref{Matrix_form1} for the $n$-dimensional function $Y(t)$ resulting from semi-discretization of the heat equation \eqref{heat_eq2}.  Aggregated characteristics can be obtained by linear combinations of the entries of $Y$ in a  post-processing step, see Sec.~\ref{sec:Aggregate}.  In the following we will work with the approximation by an analogous system introduced in Sec.~\ref{sec:analog:system}. Then the input-output behavior of the geothermal storage can be described  by a LTI system, i.e.,~a pair of a linear autonomous differential and a linear algebraic equation which is well-known from linear system and control  theory  and of the form 
	\begin{align}
		\begin{array}{rl}
			\dot{Y}(t)&=\mat{A}\,Y(t)+ \mat{B}\,g(t), \\				
			Z(t)&=\mat{\omatrix}\, Y(t).
		\end{array}
		\label{sys_org}
	\end{align}
	Here, $\mat{A} \in \R^{n \times n}, \mat{B} \in \R^{n \times m}, \mat{\omatrix} \in \R^{n_0 \times n}$  for  $n,m,n_o\in\N$ are called  \textit{system, input, output matrix}, respectively. Further,   $~g: [0,T] \to \R^{m}$  is the \textit{input} (or control),  $Y:[0,T]\to \R^n $ the state and $Z:[0,T] \to  \R^{n_o}$ is the \textit{output}. Given some initial value $Y(0)=y_0$ the input-output behavior, i.e., the mapping of the input $g$ to the output $Z$ is fully described by the   triple of matrices $(\mat{A,B,\omatrix})$ which  is called \textit{realization} of the above system. 
	
	For the analogous system  $\mat{A}$ and $\mat{B}$ are constant matrices which are given in \eqref{matrix_A} and \eqref{eq:input_matrix} for the case of constant velocity $v_0(t)=\vconst$, i.e., the pump is on. From  Theorem \ref{stable_m} we know  that  $\mat{A}$ is stable. The input dimension is $m=2$ while the dimension $n$ of the state depends on the discretization of the spatial domain $\mathcal{D}$. The two entries   of the input $g$ are the temperatures at the inlet and of the underground at the bottom boundary. Thus $g$ is piecewise continuous and bounded. The output $Z$ contains the desired aggregated characteristics such as the average temperatures $\Qdom$ of the medium, the fluid, at the outlet or the bottom boundary.  The associated row matrices $\Outputdom$ of the approximation $\Qdom(t)= \Outputdom Y(t)$ given in \eqref{aggr_char} form the $n_o$ rows of the output matrix $C$. The output dimension ${n_o}$ is the number of characteristics included in the problem and typically small while the state dimension $n$  will be very large in order to obtain a reasonable accuracy for the semi-discretized solution of the heat equation \eqref{heat_eq2}.
	
	For the above systems with high-dimensional state the simulation of the input-output behavior  and the solution of optimal control problems suffer from the curse of dimensionality because of the computational complexity and storage requirements. This motivates us to apply model order reduction (MOR).

	The general goal of MOR is to approximate the high-dimensional linear time-invariant system \eqref{sys_org} given by the realization $(\mat{A,B,\omatrix})$ by a low-dimensional reduced-order system 	
	\begin{align}
		\begin{array}{rl}
			\dot{\widetilde{Y}}(t)&=\widetilde{\mat{A}}\,\widetilde{Y}(t)+ \widetilde{\mat{B}}g(t) \\[0.5ex]
			\widetilde{Z}(t)&=\widetilde{\mat{\omatrix}}\,\widetilde{Y}(t),
		\end{array}
		\label{sys_red}
	\end{align}
	i.e., a realization $(\widetilde{\mat{A}},\widetilde{\mat{B}},\widetilde{\mat{\omatrix}})$		
	where $\widetilde{\mat{A}} \in \R^{\dimred \times \dimred},~\widetilde{\mat{B}} \in \R^{\dimred \times m}, ~\widetilde{\mat{\omatrix}} \in \R^{n_o \times \dimred},~\widetilde{Y} \in \R^{\dimred},$ $\widetilde{Z} \in \R^{n_o}$ and  $ \dimred\ll n$ denotes the  the dimension of the reduced-order state. We notice that the input variable $g$ is the same for systems \eqref{sys_org} and \eqref{sys_red}.
	The reduced-order system should capture the most dominant dynamics of the original system in particular preserve the main physical system properties, e.g.,  stability. Further, it should provide a  reasonable approximation of the original  output $Z$ by $\widetilde Z$ to given input $g$ where the output error  $Z-\widetilde Z$ is measured using a suitable norm.
	In addition, the computation of the	reduced-order system should be numerically stable and efficient. 
	
	Below we will derive a reduced-order model using Lyapunov balanced truncation. That method belongs to projection-based methods which are sketched next.
	
	\subsection{Projection-Based Methods}  
	\label{subsec:projection}
	The underlying idea of projection-based methods is that the state dynamics can be well approximated by the dynamics of a projection of the $n$-dimensional state $Y$ onto a  a suitable low-dimensional subspace of $\R^n$ of dimension $\dimred<n$. Then the aim is to describe the dynamics of the projection by a $\dimred$-dimensional system of ODEs. Prominent examples are the modal truncation and balanced truncation method, u  see  Antouslas \cite[Secs. 7, 9.2]{antoulas2005approximation}.
	Here, the projection is found by applying a suitable linear state \textit{transformation} $\overline Y=\trans Y$ with some non-singular $n\times n$-matrix $\trans$ which allows to define the projection by \textit{truncation} of $\overline Y$, i.e., replacing the last $n-\dimred$ entries by zero.

	\paragraph{Transformation} The above mentioned transformation allows to derive the following alternative equivalent realization of the system \eqref{sys_org} which is proven in Appendix~\ref{append_d}. 
	\begin{lemma}~
		\label{transform_lemma}			
		Let $\trans$ be a $n \times n$ constant non-singular transformation matrix. If we define the transformation $\overline{Y}=\trans Y$, then the  state and  output equation in \eqref{sys_org} become
		\begin{equation}
			\label{sys_trans}
			\begin{aligned}									
				\dot{\overline{Y}}(t)&=\overline{\mat{A}}~\overline{Y}(t)+ \overline{\mat{B}}g(t), \\
				Z(t) &=\overline{\mat{\omatrix}}~\overline{Y}(t).										
			\end{aligned}							
		\end{equation} 	
		
		The realization of the system is given by $(\overline{\mat{A}},\overline{\mat{B}},\overline{\mat{\omatrix}})$ where
		$$\overline{\mat{A}}=\trans\mat{A}\trans^{-1}, ~ \quad \overline{\mat{B}}=\trans\mat{B} \quad \text{and} \quad \overline{\mat{\omatrix}}=\mat{\omatrix} \trans^{-1}.$$					
	\end{lemma}

	\paragraph{Truncation}		After transforming the system we proceed to the truncation step.
	It is assumed that the transformation $\trans$ is such that the first $\dimred$ entries of the transformed state $\overline{Y}$ forming the $\dimred$-dimensional vector $\overline{Y}_1$ represent  the ``most dominant'' states while the remaining $n-\dimred$ entries which are collected in the $(n-\dimred)$-dimensional vector $\overline{Y}_2$ comprise  the ``less dominant'' and ``negligible'' states.  Based on this decomposition of $\overline{Y}$ into $\overline{Y}_1$ and $\overline{Y}_2$   the following partition of system  \eqref{sys_trans} into blocks is obtained
	
	\begin{align*}
		\begin{pmatrix}
			\dot{\overline{Y}}_1\\\dot{\overline{Y}}_2
		\end{pmatrix}&=\begin{pmatrix}
			\overline{\mat{A}}_{11} & \overline{\mat{A}}_{12} \\ \overline{\mat{A}}_{21} & \overline{\mat{A}}_{22} 
		\end{pmatrix} \begin{pmatrix}
			\overline{Y}_1\\ \overline{Y}_2
		\end{pmatrix}+\begin{pmatrix}
			\overline{\mat{\mat{B}}}_1\\ \overline{\mat{B}}_2
		\end{pmatrix}g(t), \qquad 
		Z=\begin{pmatrix}
			\overline{\mat{\omatrix}}_1&\overline{\mat{\omatrix}}_2
		\end{pmatrix}\begin{pmatrix}
			\overline{Y}_1\\ \overline{Y}_2
		\end{pmatrix}.
	\end{align*}

	The above block partition of the equivalent realization \eqref{sys_trans} is used to define the reduced-order system by assuming that the truncation of  $\overline{Y}$ to the first $\dimred$ entries defines the desired projection. Then 
	replacing  the ``negligible'' states $ \overline{Y}_2$ by zero and keeping only the first $\dimred$ state equations for  $\overline{Y}_1$ leads to the  system 
	\begin{equation}
		\label{sys_trunc}
		\begin{aligned}
			&\dot{\overline{Y}}_1=\overline{\mat{A}}_{11}\overline{Y}_1 +\overline{\mat{B}}_1g(t), \qquad 
			\overline{Z}=\overline{\mat{\omatrix}}_1 \overline{Y}_1.
		\end{aligned}
	\end{equation}
	
	Thus is the desired reduced-order system \eqref{sys_red} is given by the realization 	$(\widetilde{\mat{A}},\widetilde{\mat{B}},\widetilde{\mat{\omatrix}})=(\overline{\mat{A}}_{11}, \overline{\mat{B}}_1, \overline{\mat{\omatrix}}_1)$ and approximates the output $Z$ of the original system \eqref{sys_org} by $\widetilde Z = \overline Z$.
	
	\subsection{Lyapunov Balanced Truncation}
	\label{Balanced_T}
	\subsubsection{Setting}
	The crucial question for the above introduced projection-based methods for MOR is the choice of a ``suitable'' transformation matrix $\trans$ which was left out and will be addressed in this subsection. The transformation should be such that the first  entries of the transformed state $\overline{Y}$ provide  the largest contribution to the input-output behavoir of the system. They carry the essential information for approximating the system output $Z$ to a given input $g$ with sufficiently high accuracy. On the other hand the last entries should deliver the smallest contribution  to the input-output behavoir and thus can be neglected.
	
	Lyapunov balanced truncation uses ideas from control theory in particular the notion of controllability and observability which we sketch below.
	The basic idea  is to define a transformation $\trans$ that ``balances'' the state in a way that the first $\dimred$ entries of $\overline{Y}$ are the  states which are the easiest to observe  and to reach.   Then states which are simultaneously difficult to reach and to observe can be neglected and are truncated.
	That method appears to be well-suited for the present problem of approximation input-output behavior of the geothermal storage. 
	Balanced truncation MOR was first presented by Moore  \cite{moore1981principal} who exploited results of  Mullis and Roberts \cite{mullis1976synthesis}. The preservation of stability was addressed by Pernebo and Silverman \cite{pernebo1982model}, error bounds derived by Enns \cite{enns1984model} and Glover \cite{glover1984all}
	
	In the following we will always assume that the linear system \eqref{sys_org} is stable, i.e., the system matrix $\mat{A}$ is stable. From  Theorem \ref{stable_m} it is known that this is the case in the  problem under consideration. This allows that the system dynamics is not only considered on  finite time intervals  $[0,T]$ but also on  $[0,\infty)$.

	Given the initial state $Y(0)=y_0$ and the control $g$, there exists a unique solution to the continuous-time dynamical system \eqref{sys_org} given by 
	\begin{align}
		Y(t) =\solfun(t;y_0,u)&:=e^{\mat{A}t}y_0+\int_{0}^{t} e^{\mat{A}(t-s)} \mat{B} g(s)ds, \label{solution} \\
		Z(t) =\mat{\omatrix}Y(t)&=\mat{\omatrix}e^{\mat{A}t}y_0 +\int_{0}^{t} \mat{\omatrix} e^{\mat{A}(t-s)} \mat{B} g(s)ds. \nonumber
	\end{align}
	Note that the first term in the above state equation is $\solfun(t;y_0,0)$ representing  the response  to an initial condition $y_0$ and to a zero input while the second term is  $\solfun(t;0,g)$ and represents  the response to a zero initial state and input $g$.

	\subsubsection{Controllability and Observability} 
	We now introduce some concepts from linear system theory which will play a crucial role in the derivation of the balanced truncation method.  Let us denote by  $\Ltwo(0,t)$ the space of square integrable functions on $[0,t]$  with the induced norm $\|f\|_{\Ltwo(0,t)}=\big(\int_0^t \|f(s)\|_2^2\,ds\big)^{1/2}$. We write $\Ltwo(0,\infty)$ for functions on $[0,\infty)$.
	
	\begin{definition}[Controllability]
		\label{def:contr}
		\begin{enumerate}
			\item 
			Let the linear system \eqref{sys_org} be given. A state  $y\in\R^n$ is said to be \textem{controllable}  from zero initial state $y_0=0$ if there exist a finite time $t^*$ and an input  $g\in  \mathcal{L}_2(0,t^*)$ such that the solution given in \eqref{solution} satisfies $Y(t^*)=\solfun(t^*,0,g)=y$.
			\item
			The \textem{controllable} or reachable \textem{subspace} $\spacecontr$ is the set of states that can be obtained from zero initial state and a given input $g\in  \mathcal{L}_2(0,t^*)$.
			\item
			The linear system  \eqref{sys_org} is said to be (completely) controllable if $\spacecontr=\R^n$.		
		\end{enumerate}
	\end{definition}
	We note that for LTI systems the controllable subspace $\spacecontr$ is invariant w.r.t.~the choice of the initial state $y_0$. This allows the above  restriction  to  controllability from  zero initial state.
	
	\begin{definition}[Observability]
		\label{def:obser}
		\begin{enumerate}
			\item 
			Let the linear system \eqref{sys_org} be given and let $g=0$. A state  $y\in\R^n$ is said to be \textem{unobservable}  if the output $Z(t)=\mat{C}\solfun(t,y,0)=0 $ for all $t\ge0$, i.e., for all $t\ge 0$ the output $Z(t)$ of the system to initial state $Y(0)=y$ is indistinguishable from the output to zero initial state. Otherwise the state $y$ is called \textem{observable}.
			\item
			The \textem{observable  subspace} $\spaceobs$ is the set all observable states.
			\item
			System \eqref{sys_org} is said to be (completely) observable if $y=0$ is the only unobservable state.
		\end{enumerate}
	\end{definition}
	\begin{remark}	
		Let $Y_1,Y_2$ and  $Z_1,Z_2$ denote state and output of system \eqref{sys_org} to  initial states $y_1,y_2$ and the same input $g$.  If system  \eqref{sys_org} is observable then the equality of outputs $Z_1(t)=Z_2(t)$ for all $t\ge 0$ implies that $y_1=y_2$. Otherwise it holds  $y_1\neq y_2$. This can easily be seen from the consideration for $Y=Y_1-Y_2$ and $Z=Z_1-Z_2$ which are the state and the output to initial state $y=0$ and input $g=0$.
	\end{remark}	
	The input-output behavior of system \eqref{sys_org} can be 
	quantified by the following measures of the ``degree of controllability and observability''. They are based on the (squared) $\mathcal{L}_2(0,\infty)$-norms of the input and output functions which in the literature are often called ``input and output energy''.
	\begin{definition}[Controllability and Observability Function]
		\label{def:contr_obs_function}
		\begin{enumerate}
			\item The controllability function $\cfun:\spacecontr\to[0,\infty)$ is given for $y\in \spacecontr$ by
			\begin{align}
				\label{controllab_function}
				\cfun(y)=\min_{\substack{g \in \mathcal{L}_2(0,\infty)\\Y(0)=0,Y(\infty)=y}} { \|g\|^2_{\Ltwo(0,\infty)}}
				=\min_{\substack{g \in \mathcal{L}_2(0,\infty)\\Y(0)=0,Y(\infty)=y}}  \int_{0}^{\infty} \|g(t)\|_2^2dt.
			\end{align}	
			\item The observability function $\R^n:\to[0,\infty)$ is given for $y\in \R^n$ by
			\begin{align}
				\label{observab_function}
				\ofun(y)= { \|Z\|^2_{\Ltwo(0,\infty)}}=\int_{0}^{\infty} \|Z(t)\|_2^2dt,\quad Y(0)=y, g(t)=0 \quad \text{ for all } t\ge 0.
			\end{align}	
		\end{enumerate}
	\end{definition}	
	The controllability function $\cfun(y)$ is the smallest input energy required to reach the   state $y\in\spacecontr$  from zero initial state in infinite time. In view of the definition of a controllable state $y$ given  in Def.~\ref{def:contr} the minimization is w.r.t.~both the input function and the terminal time $t^*$. For more details see Antouslas \cite[Lemma 4.29]{antoulas2005approximation}.  
	The observability function $\ofun(y)$ is obtained as the limit of the output energy on a finite time interval $[0,T]$ for $T\to \infty$. Since the output energy increases in $T$ we can consider $\ofun(y)$  as the maximum output energy produced by the system when it is released from initial state $y$ for zero input. 
	
	The above quantities allow the following interpretation. States which are difficult to reach are characterized by large values of the controllability function $\cfun(y)$. They require large input energy  to reach them. On the other hand,  for large values of the observability function $\ofun(y)$ the state $y$ is easy to observe whereas small values of $\ofun(y)$ indicate that  state $y$ is difficult to observe since it produces only small output energy. Note, that  unobservable states  don't produce output energy at all and it holds  $\ofun(y)=0$ for $y\not\in \spaceobs$.
	
	\subsubsection{Gramians}	
	Below we will see that the controllability and observability function can be expressed in terms of the following matrices, see Antouslas \cite[Sec.~4.3]{antoulas2005approximation} 
	
	\begin{definition}[Controllability and Observability Gramian]\ \\
		Consider a stable linear system \eqref{sys_org}. The matrices defined by 
		\begin{equation}
			\label{def:Gramians}
			\begin{aligned}
				\cgram & =\int_{0}^{\infty} e^{\mat{A}t}\mat{B}\mat{B}^{\top} e^{\mat{A}^{\top}t}dt, \\[1ex]
				\ogram & =\int_{0}^{\infty} e^{\mat{A}^{\top}t}\mat{\omatrix}^{\top}\mat{\omatrix} e^{\mat{A}t}dt. 
			\end{aligned}
		\end{equation}
		are called (infinite) \textbf{controllability and observability Gramians}, respectively.	
	\end{definition}
	The Gramians enjoy the following properties.
	\begin{lemma}
		\label{lem:Gramian_prop}
		The Gramians $\cgram$ and $\ogram$ are symmetric, positive semi-definite matrices. \\
		If the linear system \eqref{sys_org} is stable, controllable and observable  then the Gramians are strictly positive definite.	
	\end{lemma}
	\begin{proof}
		The first two properties follow directly from the definition of the Gramians. 
		The third property is proven in Theorems 2.2 and 3.2 of Davis et al.~\cite{davis2009controllability}.
	\end{proof}
	
	\begin{theorem}
		\label{theo:gram_fun}	
		Let the linear system \eqref{sys_org} be a stable, {controllable and observable}.  Then the controllability and observability function defined in \eqref{controllab_function} and \eqref{observab_function} are given by 
		\begin{equation}
			\label{gram_fun}
			\begin{aligned}
				\cfun(y) &= y^{\top} \cgram^{-1} y  && \text{for } y\in\spacecontr,\\[1ex]
				\ofun(y) &= y^{\top} \ogram \,y && \text{for } y\in\R^n.
			\end{aligned}
		\end{equation}
	\end{theorem} 
	The proof of the above theorem is given in Appendix \ref{proof:theo:gram_fun}. The above relations show that states $y$ in the span of the eigenvectors corresponding to small (large) eigenvalues of $\cgram$ lead to large (small) values of $\cfun(y)$. Thus such states require a high (small) input energy and are difficult (easy) to reach. On the other hand,  states $y$ in the span of the eigenvectors corresponding to small eigenvalues of $\ogram$ produce only small output energy $\ofun(y)$ and are difficult to observe. 			
	
	The Gramians are therefore efficient tools to quantify the degree of controllabilty and observabilty of a given state.  
	Further the eigenvectors of $\cgram$ and $\ogram$ span the controllable and observable subspaces, respectively. 
	
	\subsubsection{Computation of Gramians}
	The computation of the Gramians according to the Def.~\ref{def:Gramians} is quite cumbersome. A computationally feasible method  is based on the following theorem stating that the Gramians satisfy some linear matrix equations. 
	\begin{theorem}
		\label{lyapu_theo}	
		Let the linear system \eqref{sys_org} be a stable.  Then
		the controllability Gramian $\cgram$ and observability Gramian $\ogram$ satisfy   the algebraic Lyapunov equations	 
		\begin{equation}
			\label{theo:Gram_Lyapunov}
			\begin{aligned}
				\mat{A}\cgram+\cgram\mat{A}^{\top}& =-\mat{B}\mat{B}^{\top}, \\[0.5ex] 
				\ogram\mat{A}+\mat{A}^{\top}\ogram& =-\mat{\omatrix}^{\top}\mat{\omatrix}. 
			\end{aligned}
		\end{equation}
	\end{theorem}
	The proof can be found in  Antouslas \cite[Proposition 4.27]{antoulas2005approximation} 
	and is also provided in Appendix~\ref{append_c}.

	\begin{remark}
		For solving  the above  Lyapunov matrix equation usually numerical methods have to b applied.
		Such methods  hav been addressed in	a large body of literature. For a low-dimensional and dense
		matrix $\mat{A}$, the Lyapunov equations can be solved using Hammarling method \cite{hammarling1982numerical,hodel1992parallel}, for medium- to
		large-scale Lyapunov equations, a sign function method \cite{benner1999solving,byers1997matrix} can be used, in the case of large and sparse matrix $\mat{A}$, projection-type methods such as $\mathcal{H}$-matrices based methods \cite{hackbusch1999sparse}, Krylov subspace methods \cite{jaimoukha1994krylov,simoncini2007new} and alternating direction implicit (ADI) method \cite{benner2008numerical,benner2014self,penzl1999cyclic} are more appropriate for solving Lyapunov equations.\\
	\end{remark}

	\subsubsection{Balancing}
	We now come back to the construction of a ``suitable'' transformation matrix $\trans$ which should be such that 
	the first  entries of the transformed state $\overline{Y}=\trans Y$ provide  the largest contribution to the input-output behavoir of the system. The idea of Lyapunov balanced truncation  is to use a transformation $\trans$ that ``balances'' the state such that  the first  entries of $\overline{Y}$ are simultaneously easy to reach and  easy to observe.   This allows to truncate the remaining states  which are difficult to reach and difficult  to observe.
	
	We recall the interpretation of the Gramians after Theorem \ref{theo:gram_fun}.	States that are easy to reach, i.e., those that require a small amount of input energy to reach, are found in the span of the eigenvectors of the controllability Gramian $\cgram$ corresponding to large eigenvalues. Further, states that are easy to observe, i.e., those that produce large amounts of output energy, lie in the span of eigenvectors of the observability Gramian $\ogram$ corresponding to large eigenvalues as well. However, in an arbitrary coordinate system a state $y$ that is easy to reach might be difficult to reach. On the other hand, there might exist a different state $y^\prime$ that is difficult to reach but easy to observe. Then it is hard to decide which of the two states $y$ and $y^\prime$ is more important for the input-output behavior of the system.
	
	This observation suggests to transform the coordinate system of the state space using a transformation matrix  $\trans$ in which states that are easy to reach are simultaneously easy to observe and vice versa. This is the case if the transformed Gramians coincide. Below we give a transformation for which the two Gramians are even diagonal.
	
	We recall Lemma \ref{transform_lemma}	which describes  the transformation of the realization $(\mat{A,B,C})$ of system \eqref{sys_org} to the realization  $(\overline{\mat{A}},\overline{\mat{B}},\overline{\mat{\omatrix}})$ of the transformed system \eqref{sys_trans}. The following lemma describes the  transformation of the corresponding Gramians.
	\begin{lemma}~
		\label{transform_Gram}			
		Let $\trans$ be a $n \times n$  non-singular transformation matrix defining the transformed system \eqref{sys_trans} for the   transformed state $\overline{Y}=\trans Y$. \\
		For the transformed Gramians $\cgramtr$ and $\ogramtr$ of that system it holds
		$$\overline{\cgram}=\trans \cgram \trans ^{{\top}}, \quad  \overline{\ogram}=\trans^{-{\top}} \ogram \trans^{-1} \quad \text{and} \quad \cgramtr\ogramtr=\trans \cgram\ogram \trans^{-1}.$$
	\end{lemma}			
	The proof of Lemma~\ref{transform_Gram} can be found in Appendix~\ref{append_Gram}.
	The last relation of the above lemma shows that the product $\cgramtr\ogramtr$ results from a similarity transformation of the product $\cgram\ogram$. Thus the eigenvalues of the product of the two Gramians are preserved under transformations of the coordinate system and can be considered as system invariants. They can be expressed as squares of the Hankel singular values of the system.
	\begin{definition}[Hankel Singular Values]\label{def:HankelSVD}
		The Hankel singular values are the square roots of the eigenvalues of the product of the controllability and observability Gramian
		$$\sigma_i=\sqrt{\lambda_i(\cgram\ogram)}, \quad i=1,\ldots n.$$
		Here $\lambda_i(G)$ denotes the $i$-th eigenvalue of the matrix $G$, ordered as $\lambda_1\ge \lambda_2\ge \ldots \lambda_n\ge 0$.	
	\end{definition}
	For linear systems \eqref{sys_org} which are controllable, observable and stable it is known that the Hankel singular values are strictly positive, see Antouslas \cite[Lemma 5.8]{antoulas2005approximation}.

	The next theorem presents the main result of Lyapunov balanced truncation and gives the transformation matrix $\trans$ that balances the system such that the two Gramians are equal and diagonal.
	\begin{theorem}[Balancing Transformation]\ \\	\label{theo:balancing_trans}
		Let the linear system \eqref{sys_org} be stable, controllable and observable.
		Further, let \\
		\begin{tabular}{cl}
			$\cgram=\mat{U}\mat{U}^{{\top}}$ & be the  Cholesky decomposition of the controllability Gramian \\
			&  where $\mat{U}$ is an upper triangular matrix,\\
			$\mat{U}^{{\top}} \ogram \mat{U}=\mat{K} \Sigma^2 \mat{K}^{{\top}}$ & be the eigenvalue decomposition of  ~$\mat{U}^{{\top}} \ogram \mat{U}$ such that\\ $\Sigma=\diag(\sigma_1,\ldots,\sigma_n)$&  is the diagonal matrix of Hankel singular values from Def.~\ref{def:HankelSVD}.
		\end{tabular}\\
		Then for the transformation matrix  
		\begin{align}
			\label{trans_mat1}
			\trans=\Sigma^{\frac{1}{2}}\mat{K}^{{\top}}\mat{U}^{-1}
		\end{align}the transformed system (\ref{sys_trans}) is balanced and   the controllability $\cgramtr$ and observability $\ogramtr$ Gramians  are given by 
		$$\cgramtr=\ogramtr=\Sigma$$
		i.e., they are 	diagonal and equal to a diagonal matrix containing the  Hankel singular values as diagonal entries 
		
	\end{theorem}
	The proof of Theorem~\ref{theo:balancing_trans}  can be found in Appendix~\ref{append_e}.

	The above approach may be numerically inefficient and ill-conditioned. The reason is that for large-scale systems  the Gramians $\cgram,\ogram$ often have a numerically low rank compared to the dimension $n$. This is due to the fast decay of the eigenvalues of the Gramians and also of the Hankel singular values. Then the computation of inverse matrices such as $U^{-1}$ should be avoided from an numerical point of view. In the literature  there are several suggestions of alternative approaches which are identical in theory but yield algorithms with different numerical properties, see  Antouslas \cite[Sec.~7.4]{antoulas2005approximation} and the references therein. One of them is given in the next theorem.
	\begin{theorem}[Square Root Algorithm, Antouslas \cite{antoulas2005approximation}, Sec.~7.4]\ \\	\label{theo:square_root_alg}
		Let the assumptions of Theorem \ref{theo:balancing_trans} be fulfilled.	Further, let \\
		\begin{tabular}{cl}
			$\cgram=\mat{U}\mat{U}^{{\top}}$ & be the  Cholesky decomposition of the controllability Gramian, \\
			$\ogram=\mat{L}\mat{L}^{{\top}}$ & be the  Cholesky decomposition of the observability Gramian, \\
			&  where $\mat{U}$ is an upper and $\mat{L}$ a lower triangular matrix,\\
			$\mat{U}^{{\top}} \mat{L}=\mat{W} \Sigma \mat{V}^{{\top}}$ & be the singular value  decomposition of  $\mat{U}^{{\top}} \mat{L}$ with the  \\ 
			&  orthogonal matrices $\mat{W}$ and $\mat{V}$.
		\end{tabular}\\
		Then the transformation matrix  $\trans$ given in \eqref{trans_mat1} and its inverse can be represented as 
		\begin{align}
			\label{trans_mat2}
			\trans=\Sigma^{-1/2}\mat{V}^{{\top}}\mat{L}^{\top} \quad \text{and} \quad \trans^{-1}=\mat{U}\mat{W}\Sigma^{-1/2}.
		\end{align}
	\end{theorem}
	Note that the computation of $\trans$ according to \eqref{trans_mat2} does not require the  inversion of the  full matrix $\mat{U}$ as in \eqref{trans_mat1}.

	\subsubsection{Truncation}
	\label{subsub:truncation}
	The final step of the balanced truncation MOR is the truncation of the balanced  system as it is explained in  
	Subsec.~\ref{subsec:projection}. Then the truncated balanced system as in \eqref{sys_trunc} is the desired reduced-order system \eqref{sys_red}.  It is also balanced and 
	Antoulas \cite{antoulas2005approximation} shows in Theorem 7.9 that for $\sigma_\dimred>\sigma_{\dimred+1}$ the reduced-order system is again stable, controllable and observable.

	\subsubsection{Algorithm}
	Let the linear LTI system \eqref{sys_org} with the realization ($\mat{A,B,C}$) be stable, controllable and observable.
	Then, the balanced truncation procedure can be summarized in the following algorithm.  
	\begin{algo}[Lyapunov Balanced Truncation]~
		\label{alg_baltrunc}
		\begin{itemize}
			\item [1)] Compute  the Gramians $\cgram$ and $\ogram$ by solving the Lyapunov equations \\
			$\mat{A}\cgram+\cgram \mat{A}^{\top}=-\mat{B}\mat{B}^{\top}$  and $\ogram \mat{A}+\mat{A}^{\top}\ogram=-\mat{\omatrix}^{\top}\mat{\omatrix}$. 
			\item [2)] Compute the Cholesky decomposition of the Gramians
			$\cgram=\mat{U}\mat{U}^{{\top}} \text{ and }  \ogram=\mat{L}\mat{L}^{{\top}}$\\
			where $\mat{U}$ is an upper and $\mat{L}$ is a lower triangular matrix. 
			\item [3)] Compute the singular value decomposition  \\[0.5ex]  
			$\mat{U}^{{\top}}\mat{L}=\mat{\mat{W}}\Sigma V^{{\top}} \quad \text{ with } ~\Sigma = \diag(\sigma_1,\ldots,\sigma_n) ~\text{ and }~
			\sigma_1\ge \sigma_2\ge \ldots\ge \sigma_n>0.$
			
			\item [4)] Construct the transformation matrices \quad  $\trans=\Sigma^{-1/2}\mat{V}^{{\top}}\mat{L}^{\top}$		and $\trans^{-1}=\mat{U}\mat{W}\Sigma^{-1/2}$.	
			\item [5)]  Transforming  the state to $\overline Y=\trans Y$ leads to the balanced system defined by the matrices 		
			\begin{align*}
				(\overline {\mat{A}},\overline {\mat{B}},\overline {\mat{\omatrix}})=(\trans \mat{A} \trans^{-1}, \trans\mat{B}, \mat{\omatrix} \trans^{-1})=\left(\begin{pmatrix}
					\overline {\mat{A}}_{11} & \overline {\mat{A}}_{12} \\ \overline {\mat{A}}_{21} & \overline {\mat{A}}_{22} 
				\end{pmatrix},\begin{pmatrix}
					\overline {\mat{B}}_1\\ \overline {\mat{B}}_2
				\end{pmatrix}, \begin{pmatrix}
					\overline {\mat{\omatrix}}_1&\overline {\mat{\omatrix}}_2
				\end{pmatrix}\right).
			\end{align*} 
			\item [6)] Truncate the transformed state $\overline Y=\trans Y$ keeping the first $\dimred$ entries and choose 
			\begin{align*}
				(\widetilde{\mat{A}}, \widetilde{\mat{B}},\widetilde{\mat{\omatrix}})=(\overline {\mat{A}}_{11},\overline {\mat{A}}_1,\overline {\mat{\omatrix}}_1), \quad \widetilde{\mat{A}} \in \R^{\dimred \times \dimred}, ~~\widetilde{\mat{B}} \in \R^{\dimred \times m}, ~~\widetilde{\mat{\omatrix}} \in \R^{n_o  \times \dimred}.
			\end{align*}	
		\end{itemize}
	\end{algo}

	\begin{remark}~
		\label{rem:algo}
		\begin{enumerate}
			\item From an implementation point of view in the above algorithm the truncation in step 6  can be moved to step 4 which allows the construction of the reduced-order system without forming  the high-dimensional balanced system in step 5. 		
			This leads to  the following modification of Algorithm \ref{alg_baltrunc}.
			
			\begin{itemize}
				\item [$4^\ast)$] 
				Define  the diagonal matrix  $\mat{\Sigma}_\dimred =  \diag(\sigma_1,\ldots,\sigma_\dimred)$ formed by the $\dimred$ largest Hankel singular values and  the $  n\times \dimred$ matrices $\mat{W}_\dimred, \mat{V}_\dimred$ formed by the first $\dimred$ columns of $\mat{W}, \mat{V}$, respectively.\\
				Construct  new transformation matrices with 
				\begin{tabular}[t]{l}
					the $  \dimred\times n$   matrix  $\trans_\dimred^+= \mat{\Sigma}_\dimred^{-1/2}\mat{V}_\dimred^{{\top}}\mat{L}^{\top}$ and \\
					the $ n\times \dimred$ matrix $\trans_\dimred^-=\mat{U}\mat{W}_\dimred\mat{\Sigma}_\dimred^{-1/2}$. 
				\end{tabular}\\
				Note that $\trans^+\trans^-=\mat{I}_\dimred$ is the identity matrix of dimension $\dimred$.			
				\item [$5^\ast)$] omitted 
				\item [$6^\ast)$] 	The reduced-order system is obtained  directly by 
				\begin{align*}
					(\widetilde{\mat{A}}, \widetilde{\mat{B}},\widetilde{\mat{\omatrix}})=(\trans^+ \mat{A} \trans^{-}, \trans^+\mat{B}, \mat{\omatrix} \trans^{-}) .
				\end{align*}
			\end{itemize}
			That modification is also interesting from an computational point of view 
			since it requires not the full but only a partial singular value decomposition of $\mat{U}^{{\top}}\mat{L}$ in step 3.

			\item
			The above modification is also useful for linear systems which are not (fully) observable or controllable as it is often the case if system \eqref{sys_org} results from the semi-discretization of PDEs. We also observe this for the system derived in above in \eqref{Matrix_form1}. Then the Gramians $\cgram,\ogram$ and also the product $\cgram\ogram$ might be singular and there are zero Hankel singular values such that we have $\sigma_1\ge \ldots \ge \sigma_{n_0}>0=\sigma_{n_0+1}=\ldots = \sigma_n$ for some $n_0<n$.
			In that case  the transformation matrix $\trans$ and its inverse $\trans^{-1}$ are not defined and the above algorithm breaks down. However  $\trans^+$ and $\trans^-$ are well-defined for any $\dimred \le n_0$ and  formally  the above described modification works well.
			
			A theoretical justification of that approach is  given in Tombs and Postlethwaite \cite[Thm.~2.2]{tombs1987truncated}. The authors show that for $\dimred=n_0$ the reduced-order system is stable, fully controllable and observable and has the same transfer function as the original system, i.e.~there is no change of the input-output behavior. This allows to fit  a stable but not necessarily fully observable and controllable system into the above framework from the beginning of this section. In  a preprocessing step  balanced truncation is formally applied to obtain a reduced-order system of dimension $\dimred=n_0$. The latter system is stable, controllable and observable and can be further reduced applying the standard methods to obtain an approximation of the input-output behavior.
		\end{enumerate}

	\end{remark}
	
	\subsubsection{Error Bounds}
	
	One of the advantages of Lyapunov  balanced truncation is that there exist error estimates which can be given in terms of the discarded Hankel singular values. They  allow to select the dimension $\ell$ of the reduced-order system such that a prescribed accuracy of the output approximation is guaranteed. In the literature these error estimates are given for the transfer functions of the original and the reduced-order system from which one can derive the  estimates given below for the error measured in the $\Ltwo(0,t)$-norm. 	The following theorem is proven in  Enns \cite{enns1984model} and Glover \cite{glover1984all}.
	\begin{theorem}
		\label{theo_errorbound}
		Let the linear system \eqref{sys_org} be a stable, controllable and observable with zero initial value, i.e.~$Y(0)=0$.
		Further, let Hankel singular values be pairwise distinct with $\sigma_1> \ldots > \sigma_{n}>0$. Then it holds for all $t\ge 0$
		\begin{align}
			\label{error_bound}
			\|Z-\widetilde{Z}\|_{\Ltwo(0,t)} \leq  2\sum_{i=\ell+1}^{n}\sigma_i \; \|g\|_{\Ltwo(0,t)}.
		\end{align}
		
	\end{theorem}	
	
	The error bound given in \eqref{error_bound} depends on the reduced order $\dimred$ only via the sum of discarded Hankel singular values for which we have
	\begin{align*}
		\sum_{i=\ell+1}^{n}\sigma_i = \trace(\mat{\Sigma}_2) = \trace(\mat{\Sigma}) -\trace(\mat{\Sigma}_1), 
	\end{align*}
	$\text{where } \mat{\Sigma}_1=\diag(\sigma_1,\ldots,\sigma_\dimred) \text{ and } \mat{\Sigma}_2=\diag(\sigma_{\dimred+1},\ldots,\sigma_n)$.  According  to Theorem \ref{theo:balancing_trans} we have that $\cgramtr=\ogramtr=\mat{\Sigma}$, i.e., the controllability and observability Gramians of the balanced system are equal to the diagonal matrix $\mat{\Sigma}$. Further, the two Gramians of the reduced-order system are also equal and given by the diagonal matrix  $\mat{\Sigma}_1$. We recall Theorem \ref{theo:gram_fun} which states that the output energy contained in the output $Z$ when the system is released from initial state $y$ for zero input is given by
	$\ofun(y)= \|Z\|^2_{\Ltwo(0,\infty)} =  y^{\top} \ogram \,y$, see \eqref{observab_function} and \eqref{gram_fun}. Given the balanced realization of the original system \eqref{sys_org} the output  energy related to  the initial state $Y(0)=\one_n$, i.e., the vector with all entries equal one,   is $\ofun(\one_n)=\one_n^\top \mat{\Sigma} \one_n = \trace(\mat{\Sigma})$. Analogously, for the reduced-order system the output  energy related to  the initial state $\widetilde Y(0)=\one_\dimred$ obtained by truncation of $\one_n$  to the first $\dimred$ entries is given by  $\ofun(\one_\dimred)= \|\widetilde Z\|^2_{\Ltwo(0,\infty)}= \trace(\mat{\Sigma_1})$. Thus the sum of discarded Hankel singular values which is $\trace(\mat{\Sigma}) -\trace(\mat{\Sigma}_1) $ can be interpreted as the loss of output  energy due to truncation if the balanced system  starts with an initial ``excitation'' where all (balanced) states are  equal one and no forcing is applied, i.e., the input $g$ is zero.
	
	Obviously, the above sum is decreasing in  $\dimred$ and reaches zero for $\dimred=n$.
	This motivates the introduction of the following relative selection criterion
	\begin{align}
		\selcrit(\ell) = \frac{\trace(\mat{\Sigma_1})}{\trace(\mat{\Sigma})}, \quad \text{for } \dimred=1,\ldots,n,
	\end{align}
	with values in $(0,1]$. It is increasing in $\dimred$ reaching $\selcrit(\dimred)=1$ for $\dimred=n$. $\selcrit(\dimred)$ can be used as a measure of the proportion of the output energy which can be captured by a reduced-order system of dimension $\dimred$. The selection of an appropriate dimension  $\dimred $ can be based on a prescribed  threshold level  $\alpha\in(0,1]$ for that proportion for which the minimal reduced order  reaching that level is defined  by
	\begin{align}
		\label{def_min_reduced_order}
		\dimred_\alpha = \min \{\dimred: \selcrit(\dimred)\ge \alpha\}.
	\end{align}

	\begin{remark}~\label{rem:error:bounds}
		\begin{enumerate}
			\item 
			Since the Hankel singular values only depend on the original model \eqref{sys_org} the error bound in \eqref{error_bound} can be computed a priori. Given the input $g$ this allows to control the approximation error by the selection of the reduced order $\dimred$.
			\item The reverse triangular inequality yields $\|Z-\widetilde{Z}\|_{\Ltwo(0,t)}\ge \big|\|Z\|_{\Ltwo(0,t)}- \|\widetilde{Z}\|_{\Ltwo(0,t)}\big|$. Hence, the error bound in   \eqref{error_bound} also applies to the difference of the $\Ltwo$-norms of the output $Z$ and its approximation $\widetilde Z$.
			
			\item The error bound can be generalized to systems with Hankel singular values with multiplicity larger than one. In that case they only need to be counted once, leading to tighter bounds, see Glover \cite{glover1984all}.
			\item	The assumption  of a zero initial state is quite restrictive and usually not fulfilled in applications.  We refer to  Beattie et al.~\cite{beattie2017model}, Daraghmeh at al.~\cite{daraghmeh2019balanced},  Heinkenschloss et al.~\cite{heinkenschloss2011balanced} and Schr\"oder and Voigt \cite{schroder2022balanced} where the authors study the general case of linear systems with non-zero initial conditions and derive error bounds with extra terms accounting for the initial condition.
			\item  \label{temp_shift}
			The linear systems considered in the present paper are obtained by semi-discretization of the heat equation \eqref{heat_eq2} with the associated boundary and initial conditions. The initial value $Y(0)=y_0  \in \R^n$ represents the initial temperatures $Q(0,\cdot,\cdot)$ at the corresponding grid points. In general we have $y_0\neq 0_n$. However, for the case of an homogeneous  initial temperature distribution with $Q(0,x,y)=Q_0$ for all $(x,y)\in \Domainspace$ and some constant $Q_0$ one can derive an equivalent linear system with zero initial value. That case is considered in our numerical experiments in Sec.~\ref{Numerical_exp}. 
			
			The idea is to shift the temperature scale by $Q_0$ and describe the temperature distribution by $\widehat Q(t,x,y)=Q(t,x,y)-Q_0$. Then the initial condition reads as $\widehat Q(0,x,y)=0$. Thanks to linearity of the heat equation  \eqref{heat_eq2} and the boundary conditions  $\widehat Q$ also satisfies the heat equation  together with the boundary conditions if the inlet and ground temperature appearing in the Dirichlet and Robin condition  are shifted accordingly, i.e., $\Qin$ and $\Qg$ are replaced by $\Qin-Q_0$ and $\Qg-Q_0$.  Semi-discretization of the modified  heat equation generates a linear system \eqref{sys_org} with zero initial condition  $Y(0)=0_n$ for which we can apply the error estimate given in \eqref{error_bound}. The aggregated characteristics of the temperature distribution corresponding to the  model with constant but non-zero initial temperature $Q_0$ and forming the system output  $Z$ can easily be derived from the output of the modified system.  In a post-processing step the inverse shift of the temperature scale is applied where  $Q_0$ is added to all temperatures.
		\end{enumerate}
	\end{remark}
	
	\section{Numerical Results}
	\label{Numerical_exp}		
	In this section we present results of numerical experiments on model order reduction for the system of ODEs \eqref{Matrix_form1} resulting from semi-discretization of the heat equation \eqref{heat_eq2} which  models    the spatio-temporal temperature distribution of a geothermal storage. For describing the input-output behavior of that storage we use the aggregated characteristics of the spatial temperature distribution introduced in  Sec.~\ref{sec:Aggregate}.  Further we work with the approximation of  \eqref{Matrix_form1} by an analogous system as explained in Sec.~\ref{sec:analog:system}.  		 	 
	Recall, that here it is assumed that the pump is  always on and during the waiting periods the inlet temperature $\Qin$ is set to be the average temperature $\Qf$ in the \phx fluid.
	
	We present  approximations of the aggregated characteristics obtained from reduced-order models together with  error estimates   provided in Theorem \ref{theo_errorbound} for systems with zero initial conditions. Here we apply the approach sketched in Remark \ref{rem:error:bounds}, item \ref{temp_shift}, where in a preprocessing step we first shift the temperature scale by the constant initial temperature $Q_0$ to obtain a linear system \eqref{sys_org} with zero initial value for which we can compute the error bounds. Then in a postprocessing step the inverse scale shift is applied.

	The experiments are	based on Algorithm~\ref{alg_baltrunc} and  are performed for the cases of one and three  \phxsk. The model with three \phxs is more realistic and shows more structure of the spatial temperature distribution in the storage whereas for one \phx  that  distribution is less heterogeneous.

	After explaining the experimental settings in Subsec.~\ref{Num_Settings} we start in Subsec.~\ref{subsec:num_ex1} with an experiment  where the system output consists of only one variable  which is the average temperature $\Qm$ in the storage medium. For  that case we restrict ourselves to charging and discharging  the storage  without intermediate waiting periods. Note that the analogous system requires during the waiting periods  the knowledge of the average \phx temperature which in this experiment is not included in the output variables.

	In order to allow to work with waiting periods we add in Subsec.~\ref{subsec:num_ex2} a second output variable which is the average \phx temperature $\Qf$.  Then we add a third output variable which is in Subsec.~\ref{subsec:num_ex3} the average temperature at the outlet $\Qout$ whereas in in Subsec.~\ref{subsec:num_ex3b}   the average temperature at the bottom $\Qbottom$ of the storage is used as third output variable.    The outlet temperature $\Qout$ is interesting for the management of heating systems equipped with a geothermal storage while $\Qbottom$ allows to quantify the transfer of thermal energy to the environment at the bottom boundary of the geothermal storage. Finally,  Subsec.~\ref{subsec:num_ex4} presents results for a model where the output contains all of the four above mentioned quantities.

	\subsection{Experimental Settings}	
	\label{Num_Settings}
	\begin{figure}[h!]
		\centering 
		\includegraphics[width=0.49\linewidth]{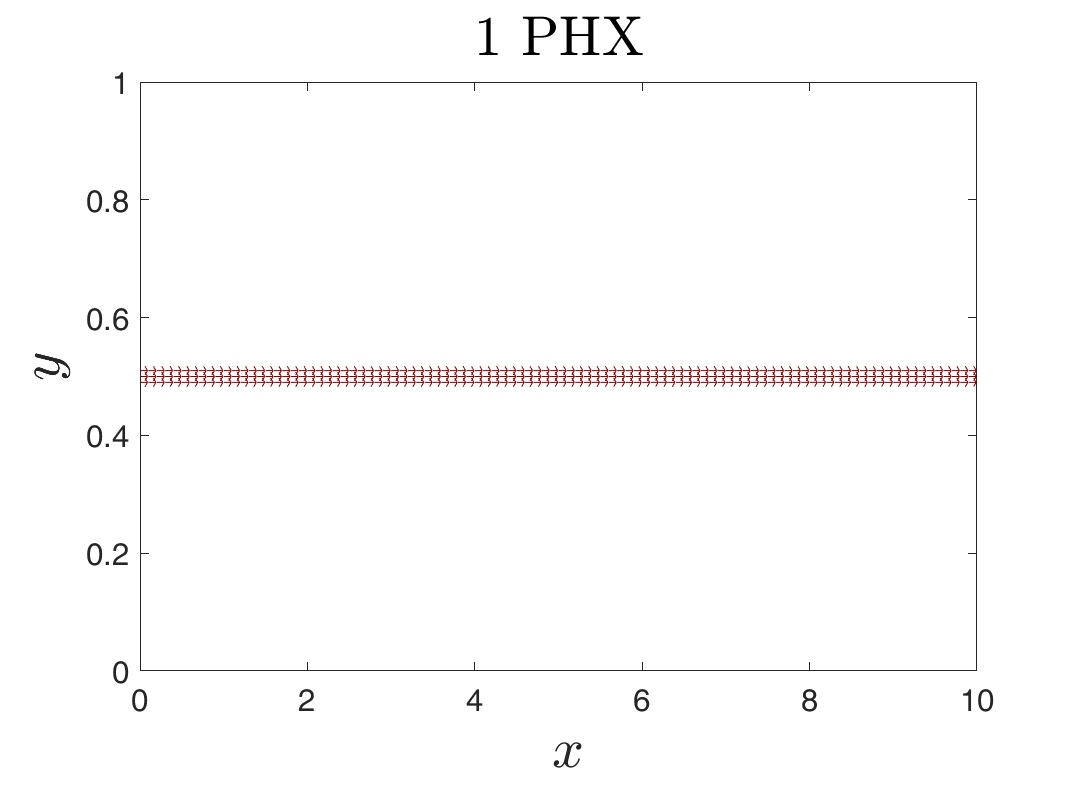}
		\includegraphics[width=0.49\linewidth]{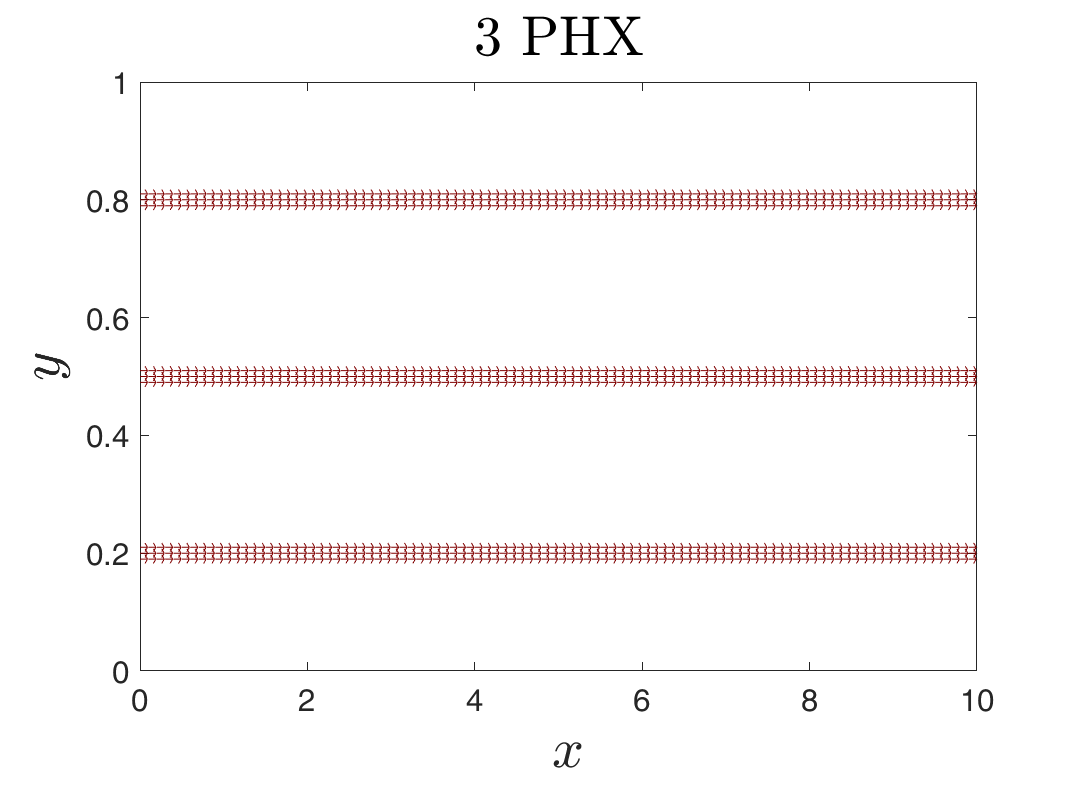}
		\caption{					Computational domain with horizontal straight \phxsk. 
			Left:  one \phx. Right: three \phxsk.					
		}
		\label{BT_pipes}
	\end{figure}
	For our numerical examples we use the model and discretization parameters taken from our companion paper  \cite{takam2021shortb} and given in Table \ref{tab:cap1}.			
	The storage is charged and discharged either by a single  \phx or by three  \phxs filled with a moving fluid, see Fig.~\ref{BT_pipes}. Thermal energy is stored by raising the temperature of the storage medium.
	We recall the open architecture of the storage which is only insulated at the top and the side but not at the bottom. This leads to an additional heat transfer to the underground for which we assume a constant temperature of  $\Qg(t)=15 \Celsius$.  In all experiments we start with a homogeneous initial temperature distribution with the constant temperature $Q_0=10 \Celsius$.
	The fluid is assumed to be water while the storage medium is dry soil. During charging  a pump moves the fluid with constant velocity $\vconst$ arriving with constant temperature $\Qin(t)=\QinC=40\Celsius$  at the inlet. This temperature   is higher than in the vicinity of the \phxsk, thus induces a heat flux into the storage medium. During discharging the inlet  temperature is $\Qin(t)=\QinD=5\Celsius$ leading to a cooling of the storage. At the outlet we impose a vanishing diffusive heat flux, i.e.~during pumping there is only a convective heat flux. 
	We also consider waiting periods where the pump is off. This helps to mitigate saturation effects in the vicinity of the \phxs   which reduce  the injection and extraction efficiency. During that waiting periods the injected heat (cold) can propagate to other regions of the storage.  Without convection of the \phx fluid   we have only diffusive propagation of heat in the storage and the transfer over the bottom boundary.

	\begin{table}[h]
		\centering	
		\begin{tabular}[T]{|lc|rrr|} \hline
			Parameters&&&Values& Units\\
			\hline		
			\textbf{Geometry} & & &&\\
			width  &$l_x$ &&$10$~&$m$ \\
			height  &$l_y$ &&$1$&$m$ \\
			depth  &$l_z$ &&$1$&$m$ \\
			diameter of \phxs &$d_P$ && $0.02$&$ ~m$ \\
			number of \phxs &$n_P$&& $1$ and $ 3$& \\\hline 
			\textbf{Material} & && &\\
			\textit{medium (dry soil)} & && &\\
			\hspace*{1em} mass density &$\rhom$ && $2000$ &$~kg/m^3$\\
			\hspace*{1em}  specific heat capacity &$\cpm$ && $800$& $~J/kg\, K$\\
			\hspace*{1em}  thermal conductivity  &$\kappam$ &&$1.59$ &~ $W/m\,  K$\\	
			\hspace*{1em}  thermal diffusivity ~~$\kappam(\rhom \cpm)^{-1}$& $\am$ && $9.9375 \times 10^{-7}$&$m^2/s$\\		
			\textit{fluid (water)} & && &\\
			\hspace*{1em} mass density &$\rhof$ &&$998$ ~&$kg/m^3$\\
			\hspace*{1em} specific heat capacity&$\cpf$&&$4182$ & $~J/kg\, K$\\
			\hspace*{1em}  thermal conductivity  &$\kappaf$ &&$0.60$ &~ $W/m\,  K$\\
			\hspace*{1em}  thermal diffusivity ~~$\kappaf(\rhof\cpf)^{-1}$ &$\af$&& $1.4376\times 10^{-7}$&$m^2/s$\\			
			velocity during pumping & $~\vconst$ && $ 10^{-2}$&$~m/s$\\
			heat transfer coeff.~ to underground & $\heattransfer$ && $10$&$ W/(m^2~ K)$\\	
			initial temperature 	&$Q_0$ &&  $10 $  &$~\Celsius$\\
			inlet temperature: charging  & $\QinC$ && $40 $&$~\Celsius$\\
			\phantom{inlet temperature:} discharging  &	$\QinD$ && $5 $&$\Celsius$\\
			underground temperature &	$\Qg$ && $15$&$ ~\Celsius$\\
			\hline 
			\textbf{Discretization} &&& & \\
			step size&$h_x$&&$0.1$&$m$\\
			step size& $h_y$&&$0.01$&$~m$\\
			time step&  $\tau$ && $1$&$ ~s$\\		
			time  horizon &  $~T$ &&  $72$ &$ ~h$ \\
			\hline
		\end{tabular}

		\medskip	
		\caption{ Model and discretization parameters.}	
		\label{tab:cap1}	
	\end{table}
	For the chosen discretization parameters the dimension of the state equation  \eqref{sys_org} resulting from the space-discretization of the heat equation is $n=10201$.  		
	The output matrix $\omatrix$ depends on the number of output variables  and changes in the various experiments.

	\subsection{One Aggregated Characteristic: ~$\Qm$}
	\label{subsec:num_ex1}
	In this example we consider a model with only a single output variable $Z_1=\Qm$, the average temperature of the medium. Then the system output does not contain the average fluid temperature $\Qf$ which serves in the analogous system as inlet temperature during waiting periods. Therefore we consider only charging and discharging periods without intermediate waiting periods. For horizon time $T=72$ hours we divide the interval $[0,T]$ into a charging period  $I_{C}=[0, T/2]$ followed by a discharging period is $I_D=(T/2,T]$.		
	The input function  $g: ~[0, T] \to \R^2$  is defined as 
	\begin{align*}
		g(t)=(\Qin(t),\Qg(t))^{\top} \quad \text{with } \quad  \Qin(t)=\begin{cases}
			\QinC =40 \Celsius& \quad \text{for}~~ t \in I_{C} ~~\text{(charging)},\\
			\QinD =5 \Celsius& \quad \text{for}~~ t \in I_{D}  ~~\text{(discharging)},\\
		\end{cases}
	\end{align*} 
	with the piece-wise constant inlet temperature $\Qin(t)$ and the constant underground temperature $\Qg(t)=15 \Celsius$.
	The output matrix in this case is given by $\omatrix=\OutputM$ which is given in our companion paper \cite[Sec.~4]{takam2021shortb}.
	\begin{figure}[h!]
		\centering 
		\includegraphics[width=0.49\linewidth]{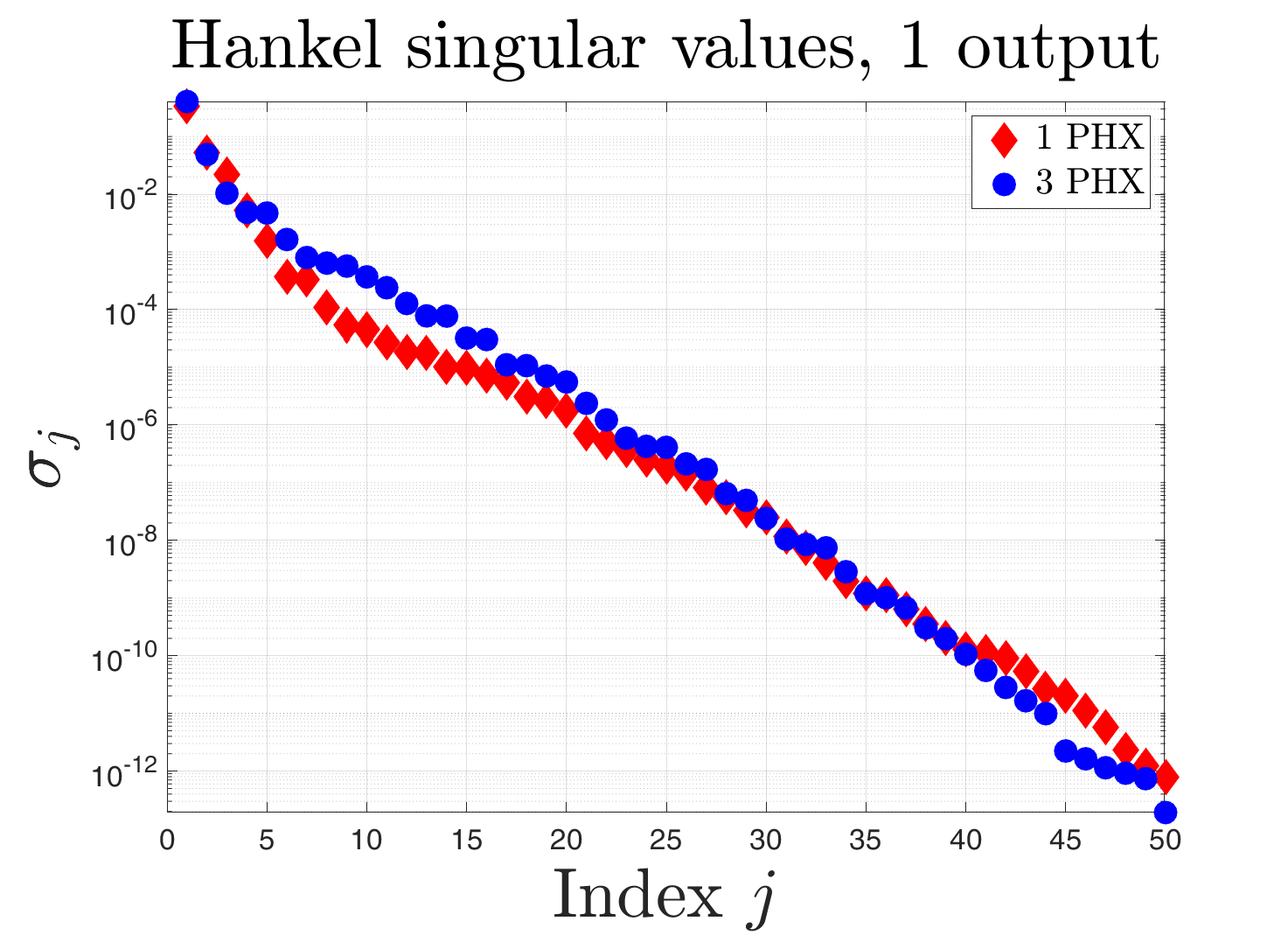}
		\includegraphics[width=0.49\linewidth]{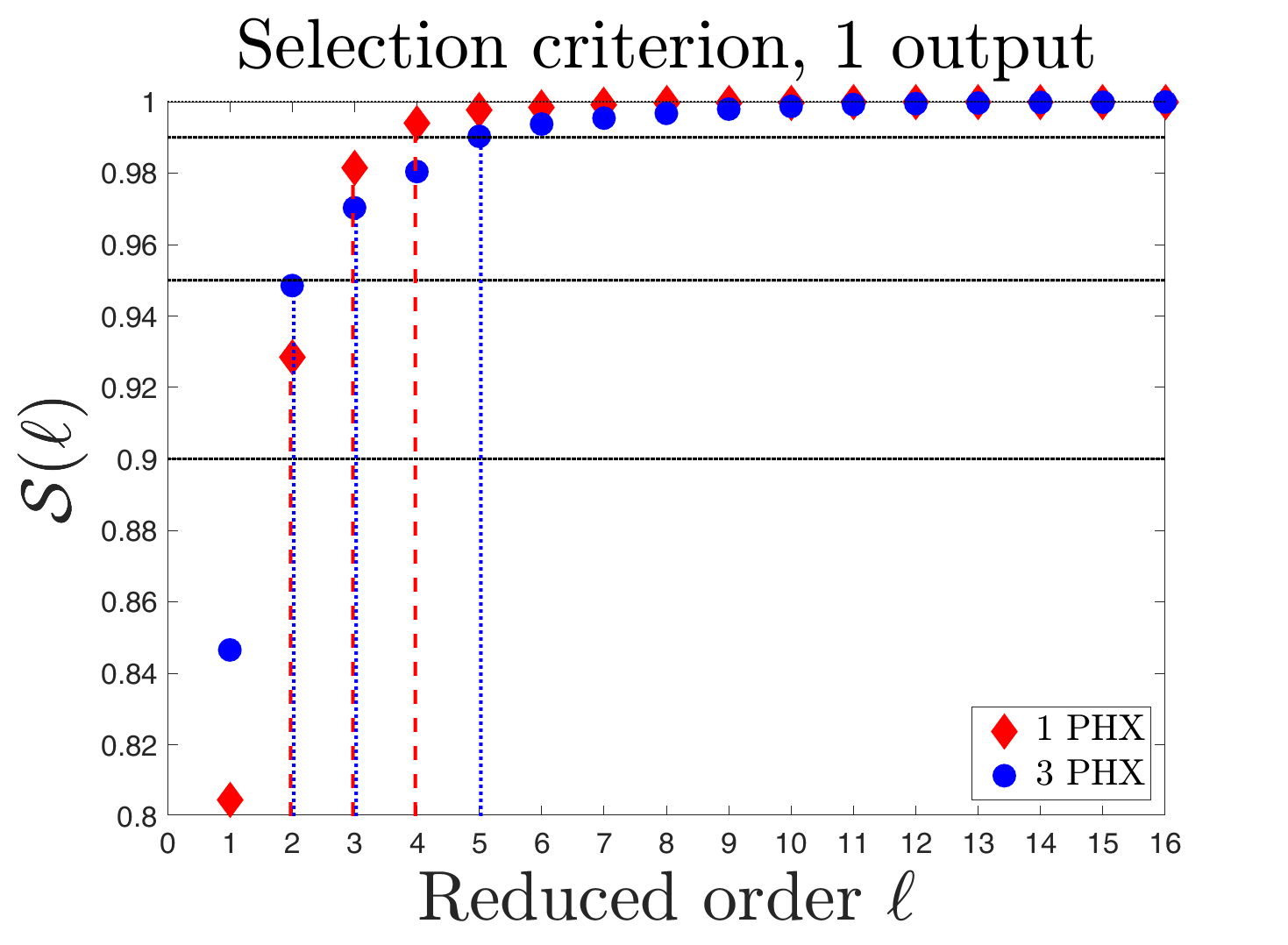}	
		\caption{Model with one output $Z=\Qm$:\newline 
			Left: first 50 largest Hankel singular values, 
			Right: selection criterion.  }
		\label{selection1CTn}
	\end{figure}
	
	In Fig.~\ref{selection1CTn}, the left panel shows first 50 largest Hankel singular values associated to the most observable and most reachable states whereas in  the right panel we plot the selection criterion against the reduced order $\dimred$ (red  for 1 \phx and blue  for 3 \phxsk). For the first 50 singular values we observe for both models  that they are all distinct and decrease fast by 12 orders of magnitude. The first 20 singular values decrease faster for the model with one \phx than for the 3 \phx model.

	\begin{table}[h]
		\[\begin{array}{|@{\hspace*{0.5em}}l|r@{\hspace*{-0.08em}/\hspace*{-0.08em}}l|r@{\hspace*{-0.08em}/\hspace*{-0.08em}}l|r@{\hspace*{-0.08em}/\hspace*{-0.08em}}l|}
			\hline
			\text{Output } \hspace*{7em}\Big|~~~~\alpha~~& \multicolumn{2}{c|}{90\%}& \multicolumn{2}{c|}{95\%}& \multicolumn{2}{c|}{99\%}\\[0.5ex]\hline
			\rule{0pt}{2.5ex}
			Z=\Qm & \hspace*{1em}2&2\hspace*{1em} & \hspace*{1em}3&3\hspace*{1em} &\hspace*{1em}4&5 \hspace*{1em}\\[0.5ex]\hline \rule{0pt}{2.5ex}
			Z=(\Qm,\Qf)^\top & 4&4 & 5&6 &11&11 \\[0.5ex]\hline \rule{0pt}{2.5ex}
			Z=(\Qm,\Qf,\Qbottom)^\top & 5&6 & 7&8 &12&13 \\[0.5ex]\hline \rule{0pt}{2.5ex}
			Z=(\Qm,\Qf,\Qout)^\top & 8&8 & 10&9 &15&14 \\[0.5ex]\hline \rule{0pt}{2.5ex}
			Z=(\Qm,\Qf,\Qout,\Qbottom)^\top & 9&9 & 11&11 &17&16 ~~\\[0.5ex]\hline
		\end{array}\]
		\caption{Minimal reduced orders  $\dimred_\alpha = \min \{\dimred: \selcrit(\dimred)\ge \alpha\}$, 1 \phx/ 3 \phxsk.}
		\label{tab:ReducedOrders}
	\end{table}
	
	We recall that the selection criterion $\selcrit(l)$ provides an  estimate of the proportion of output energy of the original system that can be captured by the reduced-order system of dimension $\dimred$. In the right panel of Fig.~\ref{selection1CTn} we draw vertical red dashed (one PHX) and blue dotted lines (three PHX) to indicate the reduced orders $\dimred$ for which the selection criterion $\selcrit(l)$ exceeds the threshold values $\alpha=90\%, 95\%,99\%$ for the first time, respectively. This allows to determine graphically the associated minimal reduced orders $\dimred_\alpha= \min \{\dimred: \selcrit(\dimred)\ge \alpha\}$ which  have been introduced in \eqref{def_min_reduced_order}. The resulting values are also given in Table \ref{tab:ReducedOrders}. It can be seen that already with $\dimred_{0.9}=2$ states the reduced-order system can capture more than $90 \%$ of the output energy  while with $\dimred_{0.99}=4$ states the level $99\%$ is exceeded for the one \phx model whereas the three \phx model is only slightly below that level and requires  $\dimred_{0.99}=5$ states.

	\begin{figure}[h!]
		\centering
		\includegraphics[width=0.49\linewidth]{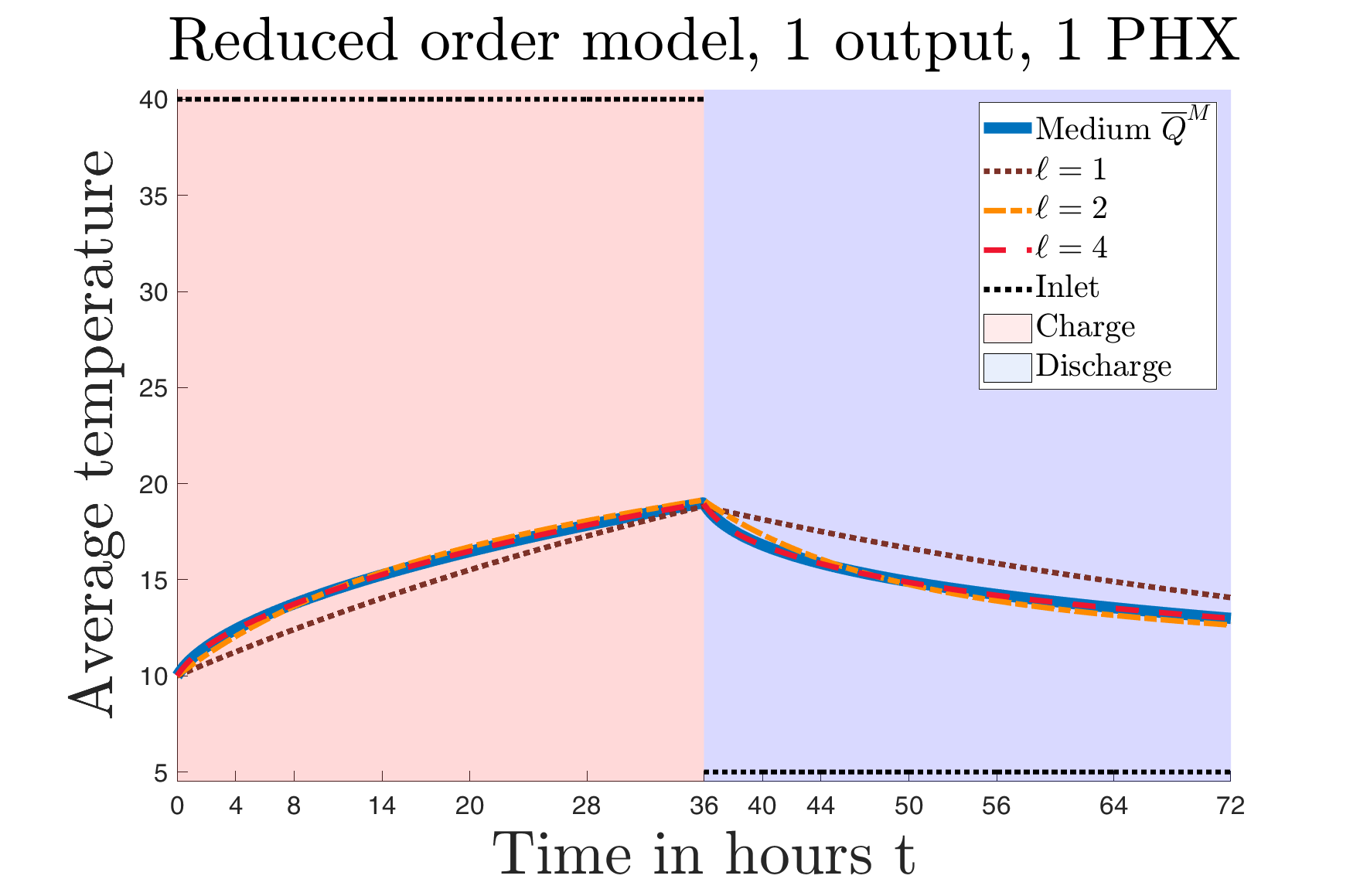}
		\includegraphics[width=0.49\linewidth]{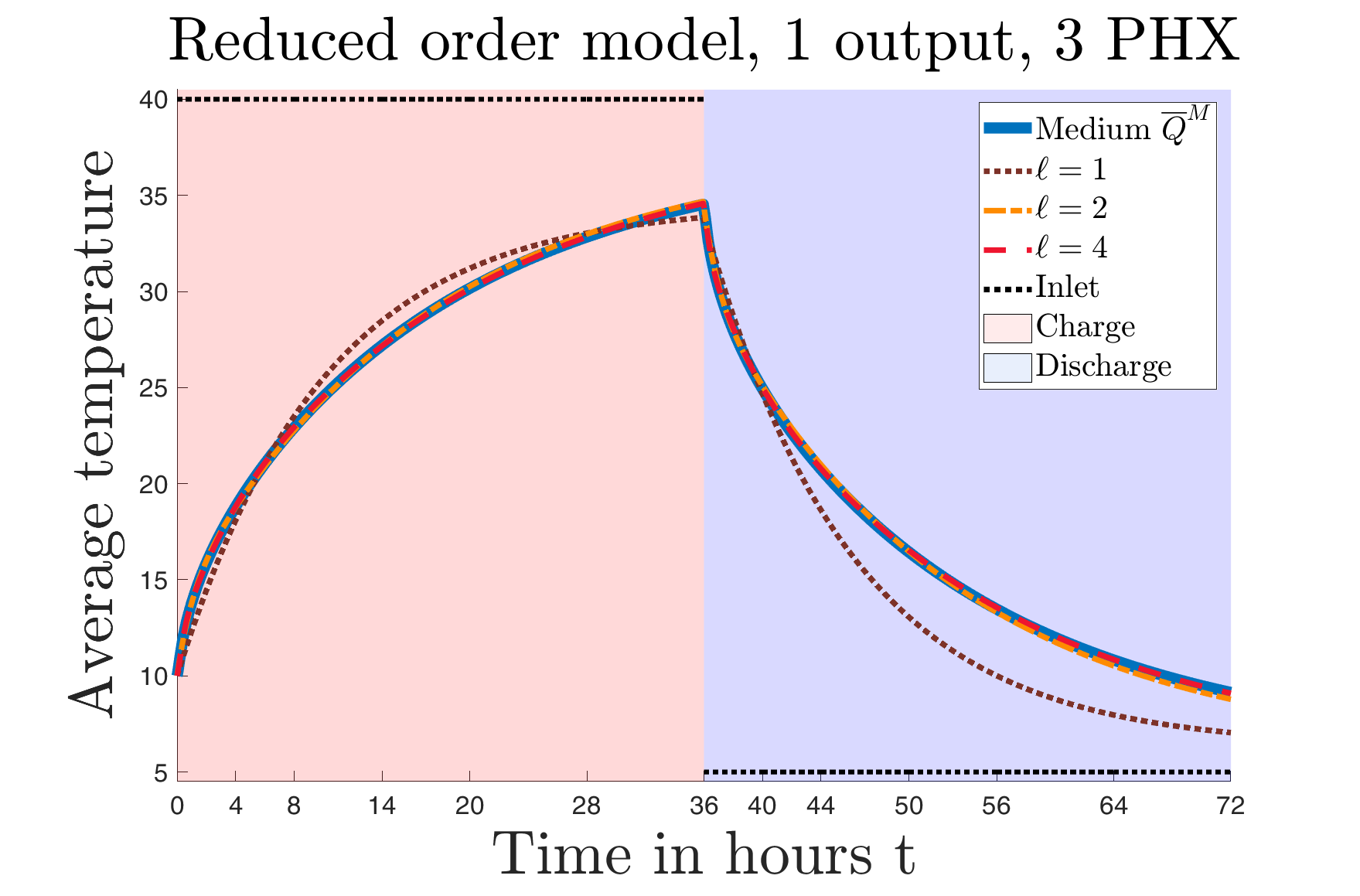}
		
		\caption{Model with one output $Z=\Qm$: \quad Approximation of the output for $\dimred=1,2,4$. \newline  Left: one \phx, Right three \phxsk. }
		\label{13CTn2D}
	\end{figure}
	\begin{figure}[!h]
		\centering
		\includegraphics[width=0.49\linewidth]{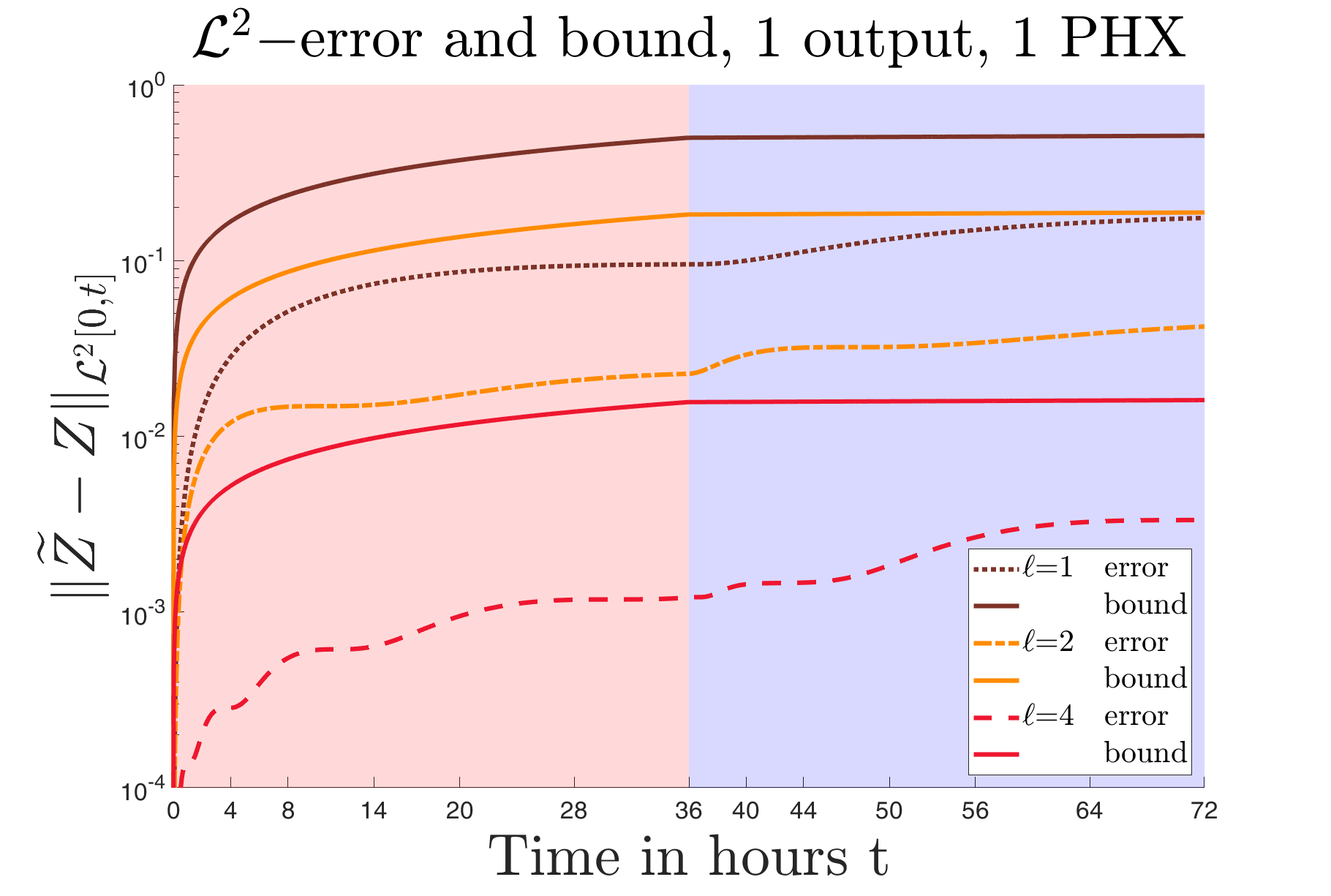}
		\includegraphics[width=0.49\linewidth]{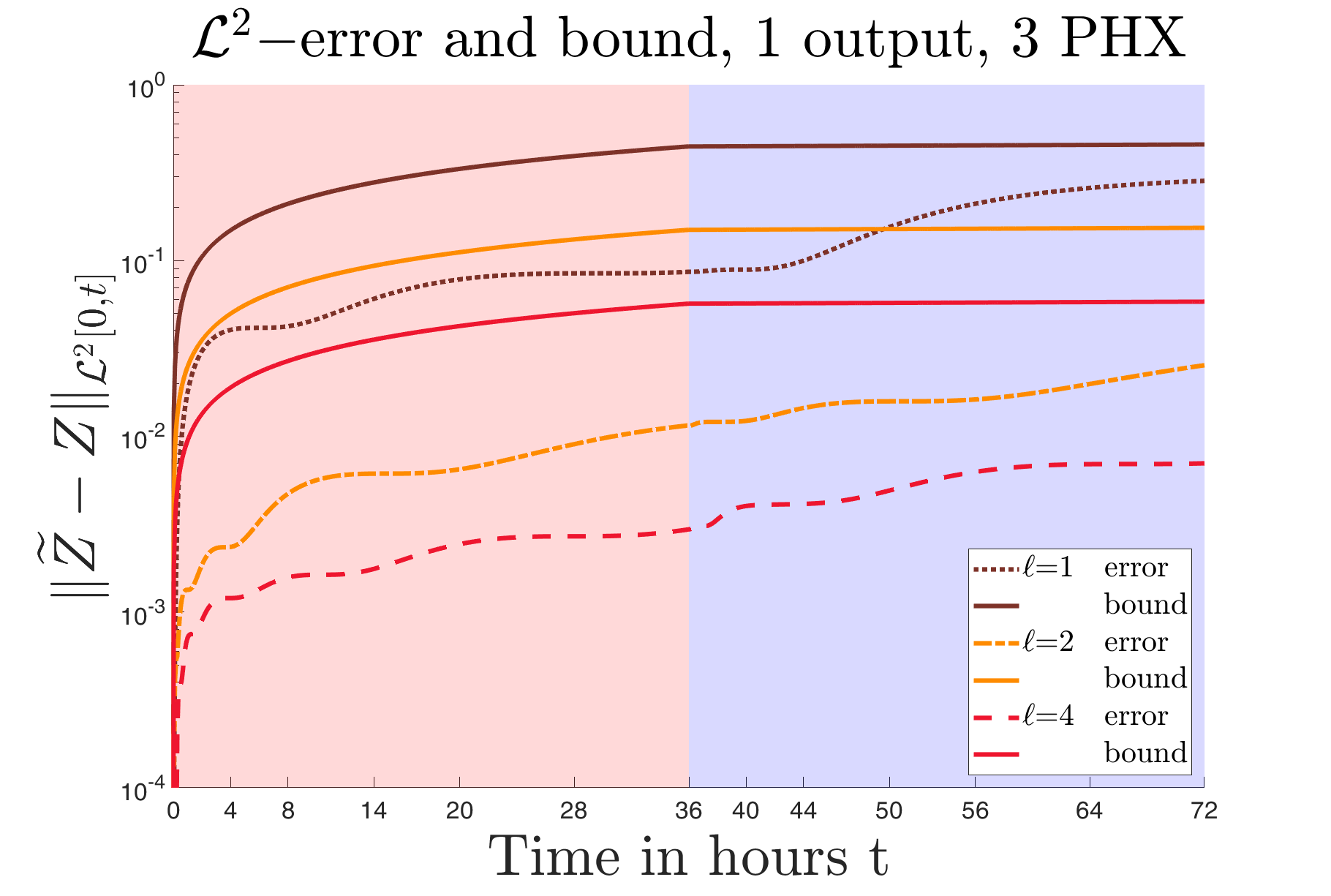}
		\caption{Model with one output $Z=\Qm$: \quad $\Ltwo$-error and error bound for $\dimred=1,2,4$.\newline  Left: one \phx, Right three \phxsk.}
		\label{error1CTn}
	\end{figure}

	Fig.~\ref{13CTn2D} allows to evaluate the actual quality of the output approximation. It plots  the average temperature in the medium $Z(t)=\Qm(t)$ against time and compares that system output of the original system (solid blue line) with the approximation $\widetilde Z(t)$ from the reduced-order model (brown, orange, red lines) for  $\dimred=1,2,4$. The figures also show the inlet temperature $\Qin(t)$ (black dotted line) which are constant and equal to $\QinC=40 \Celsius$ during charging (light red region) and $\QinD=5 \Celsius$ during discharging (light blue region), respectively. 
	The figure shows that both for one and three \phxsk, the reduced-order system captures well the input-output behaviour of the original high-dimensional system already for $\dimred \geq 2$. For $\dimred=1$ the approximation is less good  as expected in view of the low value of the selection criterion (see Fig.~\ref{selection1CTn}). 
	
	Finally, 
	Fig.~\ref{error1CTn} plots the $\Ltwo$-error  $\|Z-\widetilde{Z}\|_{\Ltwo(0,t)}$ against time $t$ and compares with the associated error bound given in Theorem \ref{theo_errorbound} for $\dimred=1,2,4$. Note that both quantities are non-decreasing in $t$. The error bounds grow less during discharging since here the  growth of the $\Ltwo$-norm of the input  $g$ is smaller due to the smaller inlet temperature $\Qin$, see Theorem \ref{theo_errorbound}. As expected from that theorem, the error bounds decrease with $\dimred$ and this is also the case for the actual error.

	The next examples will show that  the number of states needed to capture well  the input-output behavior of the system  may increase considerably if we add more aggregated characteristics to the system output.

	\subsection{Two Aggregated Characteristics: ~$\Qm,\Qf$}
	\label{subsec:num_ex2}
	\begin{figure}[h!]
		\centering 
		\includegraphics[width=0.49\linewidth]{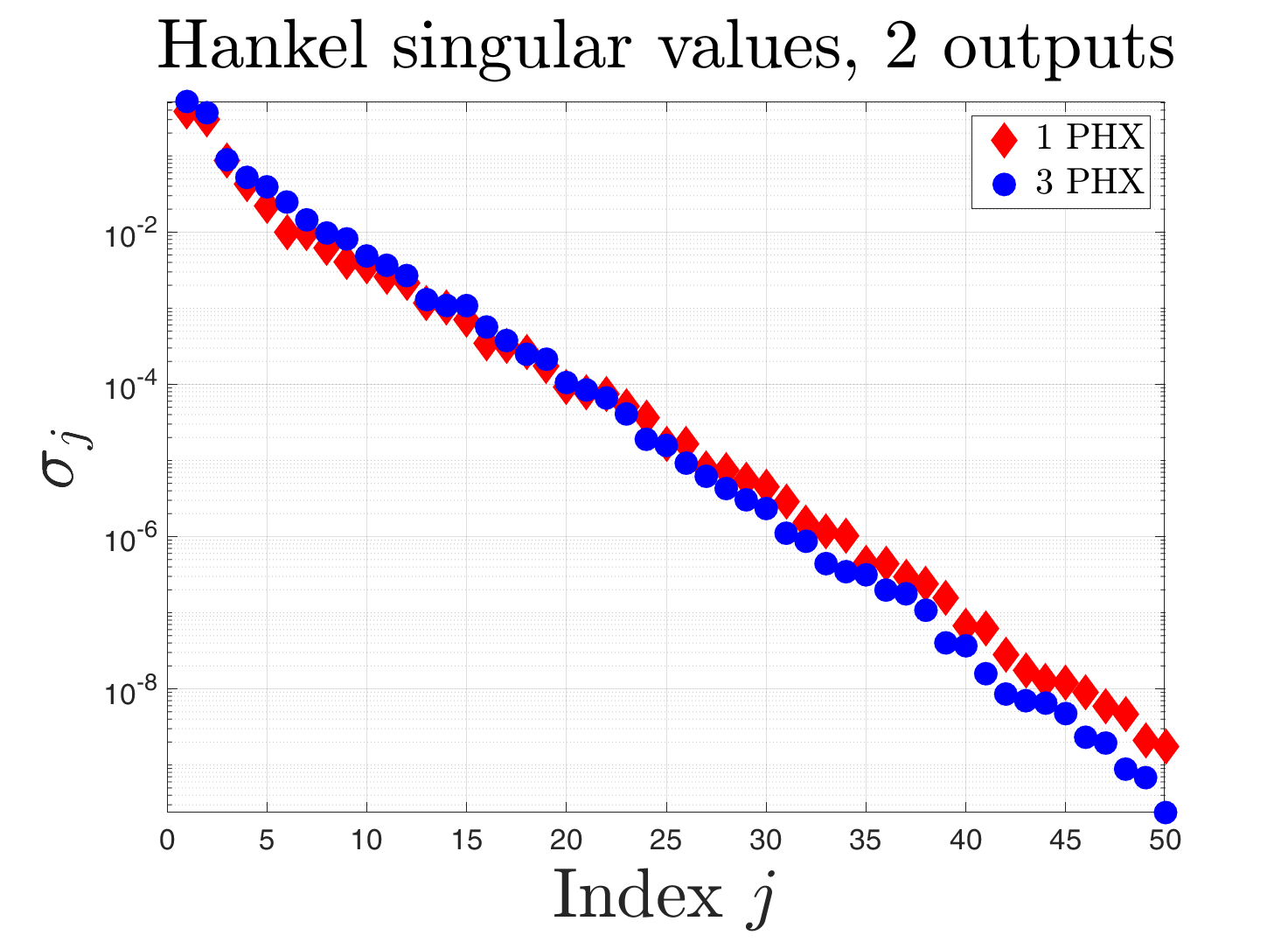}
		\includegraphics[width=0.49\linewidth]{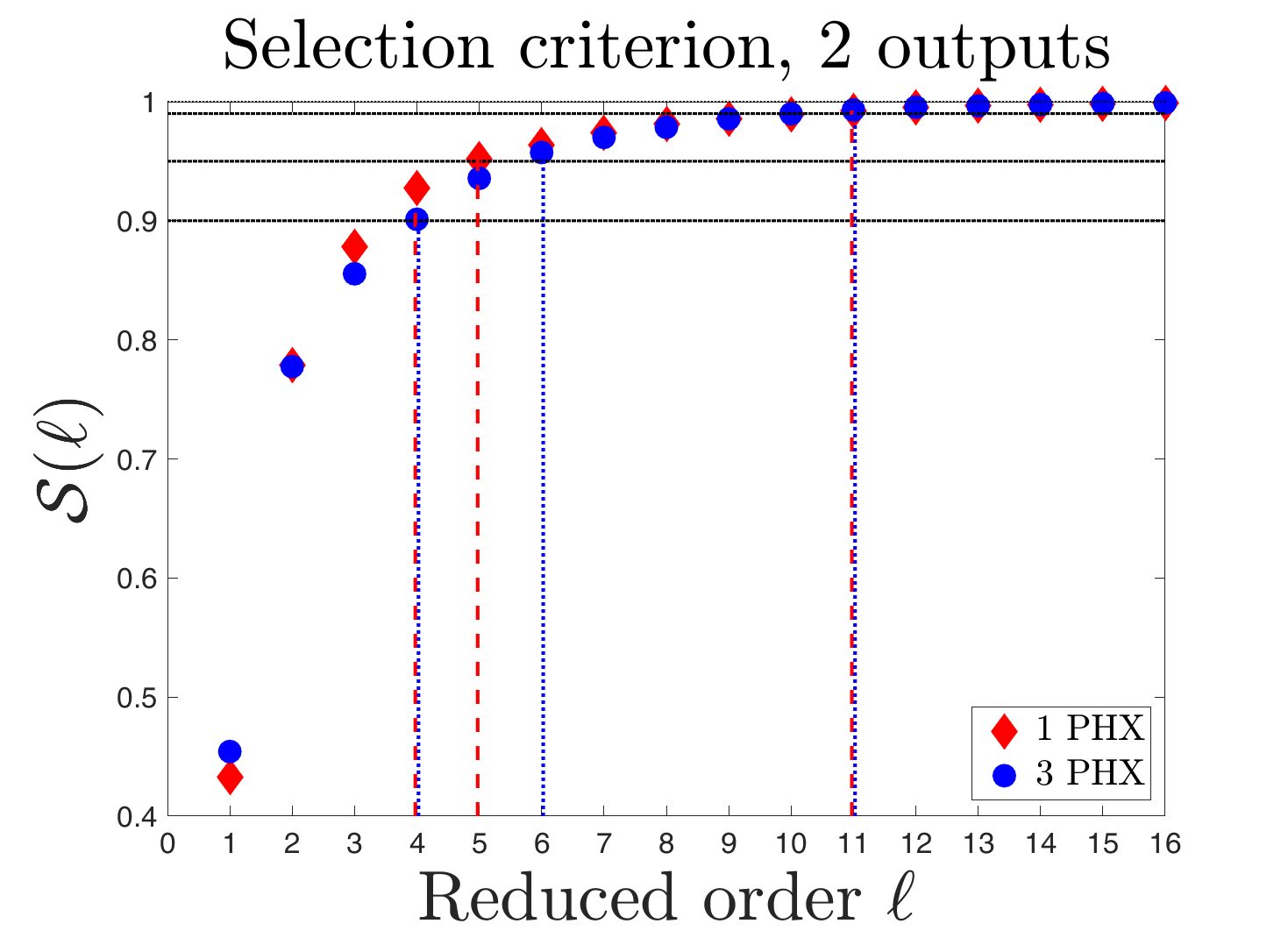}	
		\caption{Model with two outputs $Z=(\Qm,\Qf)^\top$:\newline 
			Left: first 50 largest Hankel singular values, 
			Right: selection criterion.  }
		\label{selection3Cmf}
	\end{figure}
	
	In this example we add to the system output the average temperature of the \phx fluid leading to the two-dimensional output  $Z=(\Qm,\Qf)^\top$. Since in the analogous system $\Qf$ is used as inlet temperature we now can include also waiting periods between periods of charging and discharging allowing the storage to mitigate saturation effects.  For time horizon $T=72$ hours  we divide the simulation time interval $[0,T]$ into charging, discharging and waiting periods  with 
	\begin{center}
		\begin{tabular}[t]{ll}
			$I_{C}=~~[0, 4] ~~\cup ~~[8, 14] \cup [20, 28]$, & charging,  \\ [0.5ex]
			$I_{D}=[36, 40] \cup [44,50] \cup [56, 64]$,& discharging,\\[0.5ex]
			$I_{W}=[0,72]\setminus (I_C \cup I_D$), & waiting,
		\end{tabular}
	\end{center}
	which are also depicted in Fig.~\ref{BT3Cmf2}.
	The two-dimensional  input function  $g$  is defined as 
	\begin{align}\label{ex:input}
		g(t)=(\Qin(t),\Qg(t))^{\top} ~ \text{with } ~  \Qin(t)=\begin{cases}
			\QinC =40 \Celsius& \text{for}~~ t \in I_{C} ~~\text{(charging)},\\
			\QinD =5 \Celsius& \text{for}~~ t \in I_{D}  ~~\text{(discharging)},\\
			\Qf\!(t) &  \text{for}~~ t \in I_{W}  ~~\text{(waiting)}.\\
		\end{cases}
	\end{align} 
	Here, the inlet temperature $\Qin(t)$ is the piece-wise constant during charging and discharging but time-dependent and equal to $\Qf(t)$ during waiting periods. Again the underground temperature is constant with $\Qg(t)=15 \Celsius$.
	The two rows of the $2\times n$ output matrix $\omatrix$ are $\OutputM$ and $\OutputF$ which are  given in our companion paper \cite[Sec.~4]{takam2021shortb}.
	
	\begin{figure}[h]
			\centering
			\includegraphics[width=.49\linewidth]{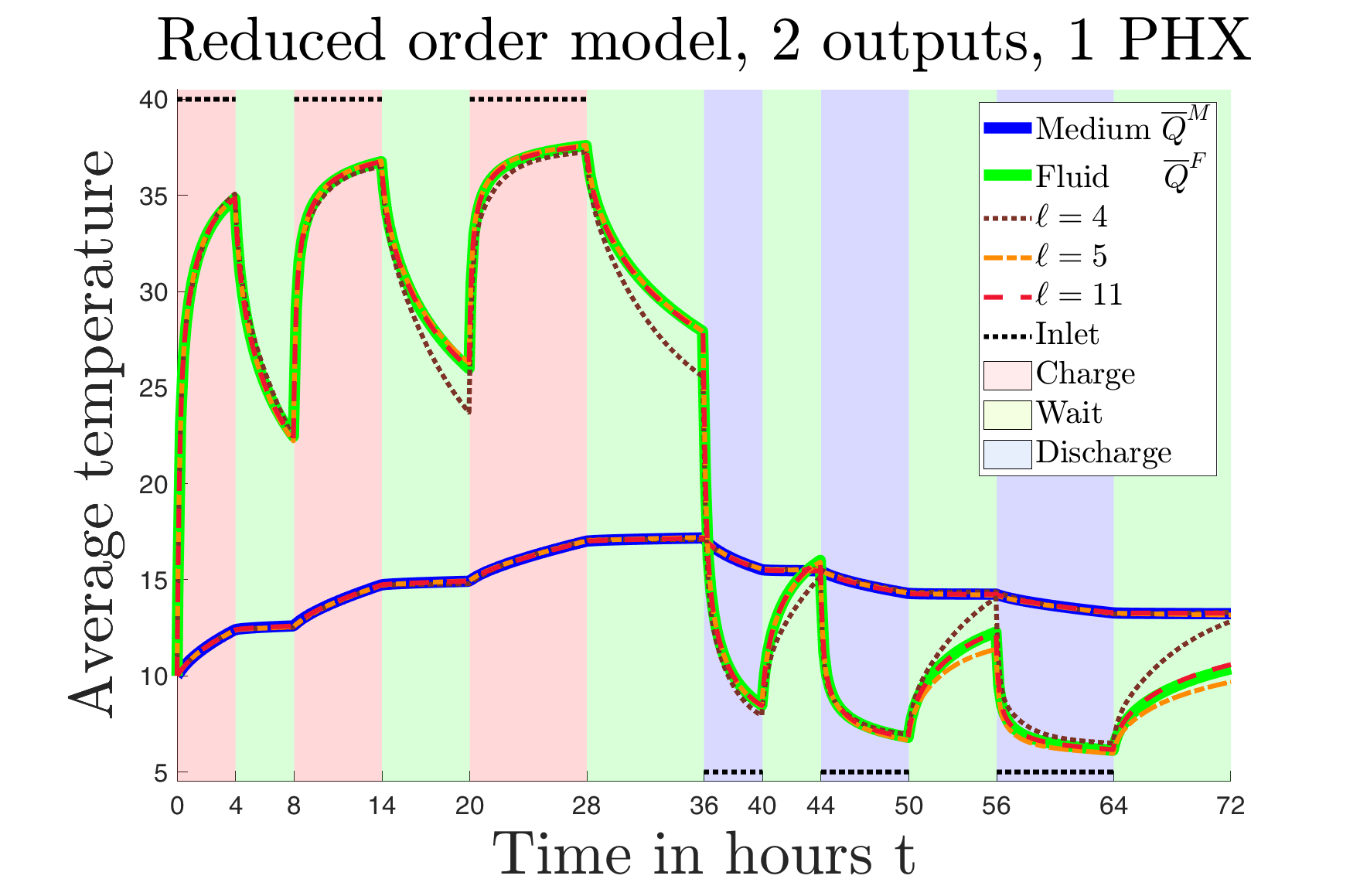}
			\includegraphics[width=.49\linewidth]{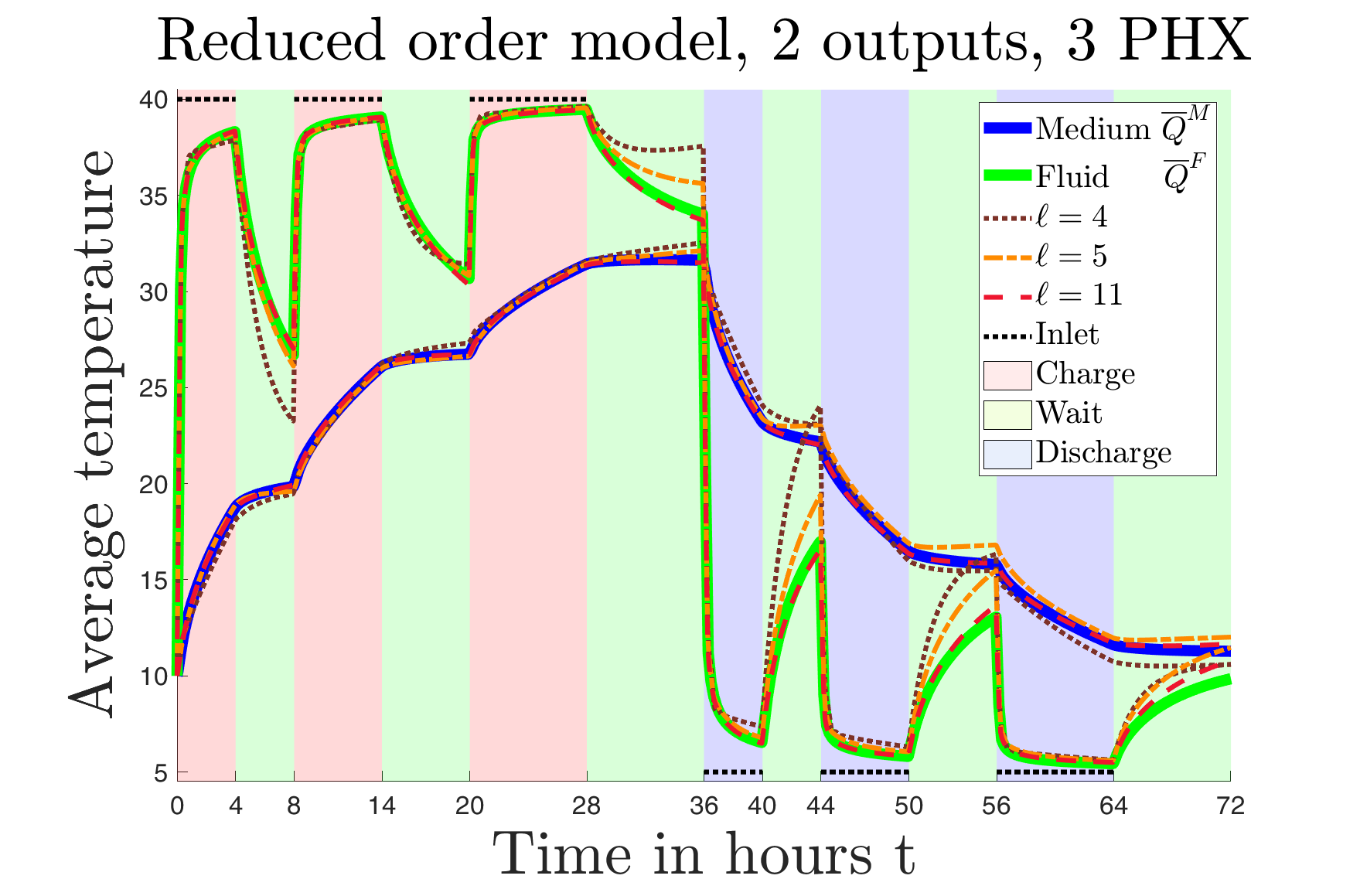}
			\mycaption{Model with two outputs $Z=(\Qm,\Qf)^\top$: \quad Approximation of the output for $\dimred=4,5,11$. \newline  Left: one \phx, Right three \phxsk. 				
			}
			\label{BT3Cmf2}
		\end{figure}

		Fig.~\ref{selection3Cmf} depicts 
		in the left panel the first 50 largest Hankel singular values whereas the right right panel shows the selection criteria (red  for 1 \phx and blue  for 3 \phxsk). For the first 50 singular values we observe for both models that they are all distinct and decrease by 9 orders of magnitude. As in the example with a single output   the first 20 singular values decrease faster for the model with one \phx than for the 3 \phx model. The selection criterion for the model with one \phx is for all $\dimred\ge 2$ larger than for 3 \phxsk. From Fig.~\ref{selection3Cmf} and also from  the minimal reduced orders reported in Table \ref{tab:ReducedOrders}  it can be seen  that a reduced-order system with $\dimred_{0.9}=4 $ states  can capture more than $90 \%$ of the output energy of the original system. For the level threshold $95 \%$ the one \phx model requires $\dimred_{0.95}=5$ states  while for the three \phx model $\dimred_{0.95}=6$ states are needed. In both cases the  level of $99 \%$ is exceeded for the first time for $\dimred_{0.99}=11$. Hence, for dimension $\ell \geq 11 $ an almost perfect approximation of the input-output behavior can be expected.

		\begin{figure}[h]
			\centering
			\includegraphics[width=0.49\linewidth]{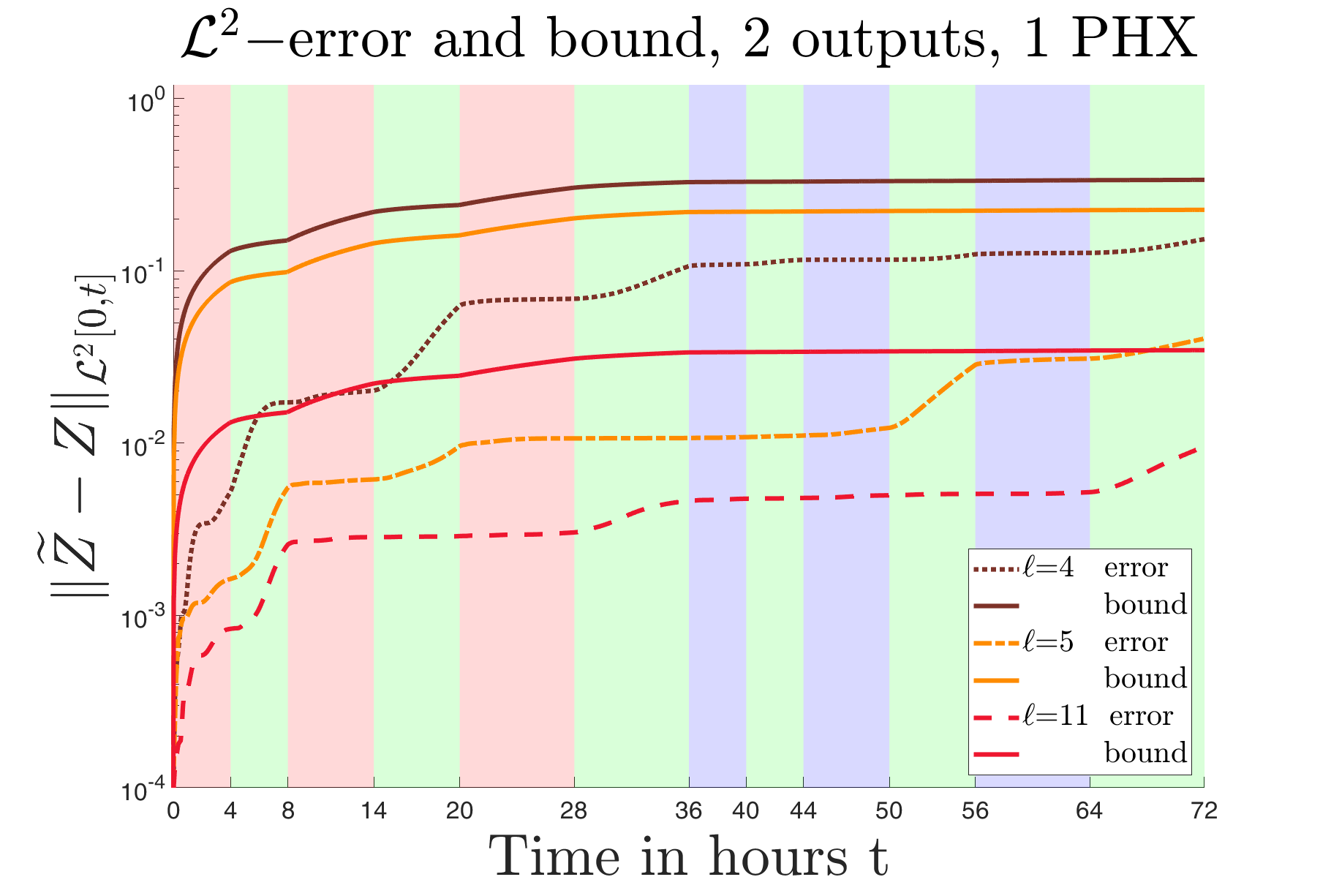}
			\includegraphics[width=0.49\linewidth]{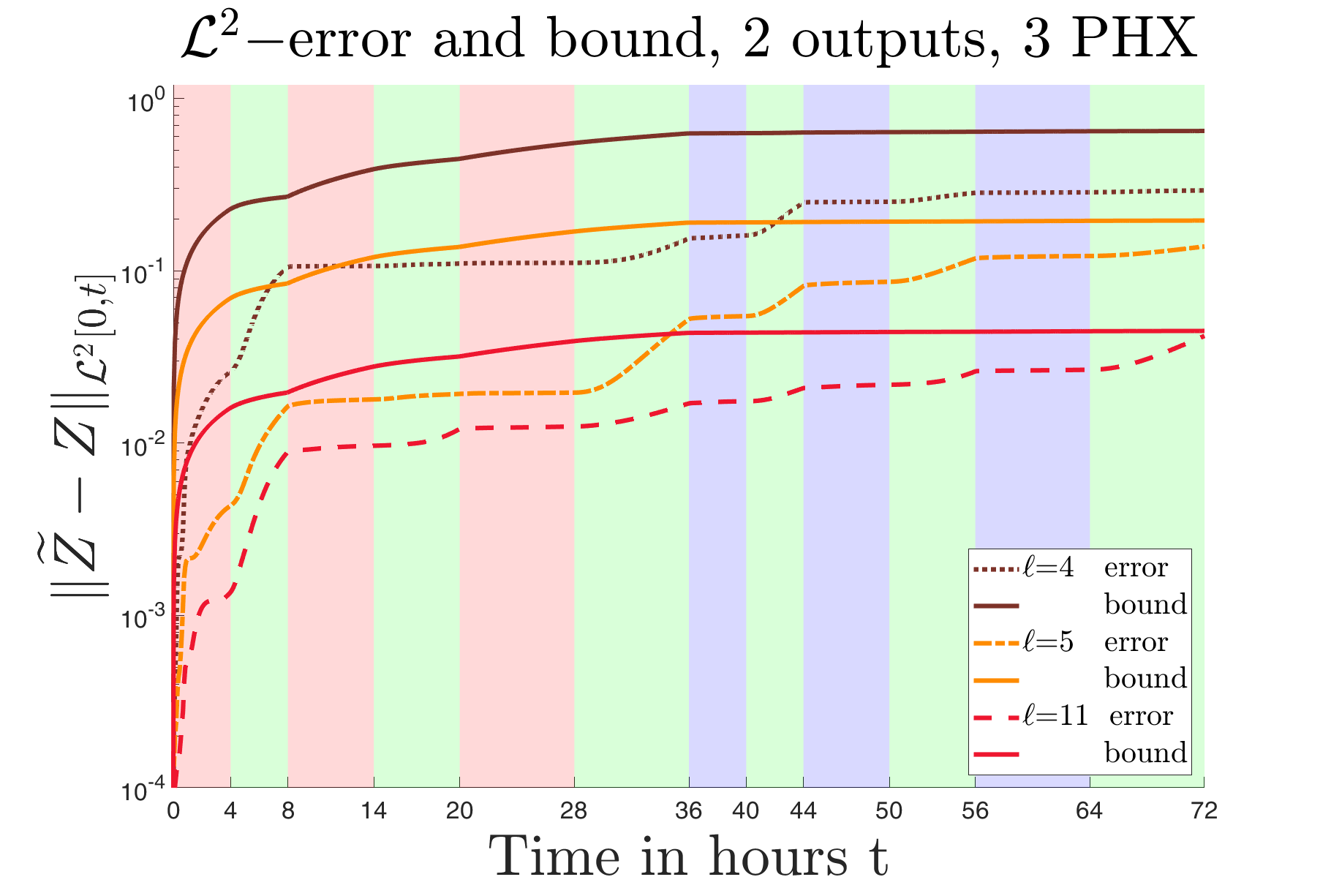}
			\caption{Model with two outputs $Z=(\Qm,\Qf)^\top$: \quad $\Ltwo$-error and error bound for $\dimred=4,5,11$.\newline  Left: one \phx, Right three \phxsk. }	
			\label{error3Cmf2}
		\end{figure}

		For the evaluation of the actual quality of the output approximation we plot in Fig.~\ref{BT3Cmf2} the output variables of the original and reduced-order system against time. The average temperatures $Z_1(t)=\Qm(t)$ and $Z_2(t)=\Qf(t)$ in the medium and fluid are drawn as solid  blue and green lines, respectively, and its approximations as brown, orange and red lines  for  $\dimred=4,5,11$. Further,  the inlet temperature $\Qin(t)$ during the charging an discharging periods is shown as black dotted line. The figures show that the approximation of $\Qm$ is better than for $\Qf$. A possible explanation is that $\Qm$ is an average of the spatial temperature distribution over the quite  large subdomain $\Dm$ (medium)  while for $\Qf$ the temperature is averaged only over the much smaller subdomain $\Df$ of the  \phx fluid.    Further, the temporal variations of $\Qf$ are much larger than those of $\Qm$ due to the impact of the changing inlet temperature during charging, discharging and waiting.   Errors are more pronounced during waiting periods than during charging and discharging. For the three \phx model the pointwise errors are slightly larger.  As noted above, for $\dimred=11$ the selection criterion exceeds $99 \%$  and now the  approximation errors are almost negligible. This was also observed for $\dimred>11$.		
		
		Fig.~\ref{error3Cmf2} plots for the reduced orders $\dimred$ considered above the $\Ltwo$-error  $\|Z-\widetilde{Z}\|_{\Ltwo (0,t)}$ against time $t$ together with the  error bounds from Theorem \ref{theo_errorbound}. This allows an alternative evaluation of the approximation quality.  As expected,  the error bounds and also the actual errors decrease with $\dimred$. While the error bounds increase more during the charging periods due to the larger norm of the input $g$ caused by the higher inlet temperature the actual error increase more during the waiting periods. This corresponds to the above observed larger errors in the output approximation in during these periods.
		
		\subsection{Three Aggregated Characteristics I: ~$\Qm,\Qf,\Qout$}			
		\label{subsec:num_ex3}
		
		This example extends the example considered in Sec.~\ref{subsec:num_ex2}  by adding  a third variable to the system output which is the average temperature at the  \phx outlet, i.e., we consider the  three-dimensional output  $Z=(\Qm,\Qf,\Qout)^\top$. 
		The outlet temperature is needed if the geothermal storage is embedded into a residential system. Then the management of  the heating system and  the interaction between the geothermal and the internal buffer storage rely on the knowledge of $\Qout$. Further, the difference $\Qin(t)-\Qout(t)$ between inlet and outlet temperature is the key quantity for the quantification of the  the amount of heat injected    to or withdrawn from the storage due to convection of the   fluid in the \phxk, we refer to Eq.~\eqref{Rout} and the explanations in  Subsec. \ref{subsec:AggCharBoundary}.
		
		\begin{figure}[h!]
			\centering 
			\includegraphics[width=0.49\linewidth]{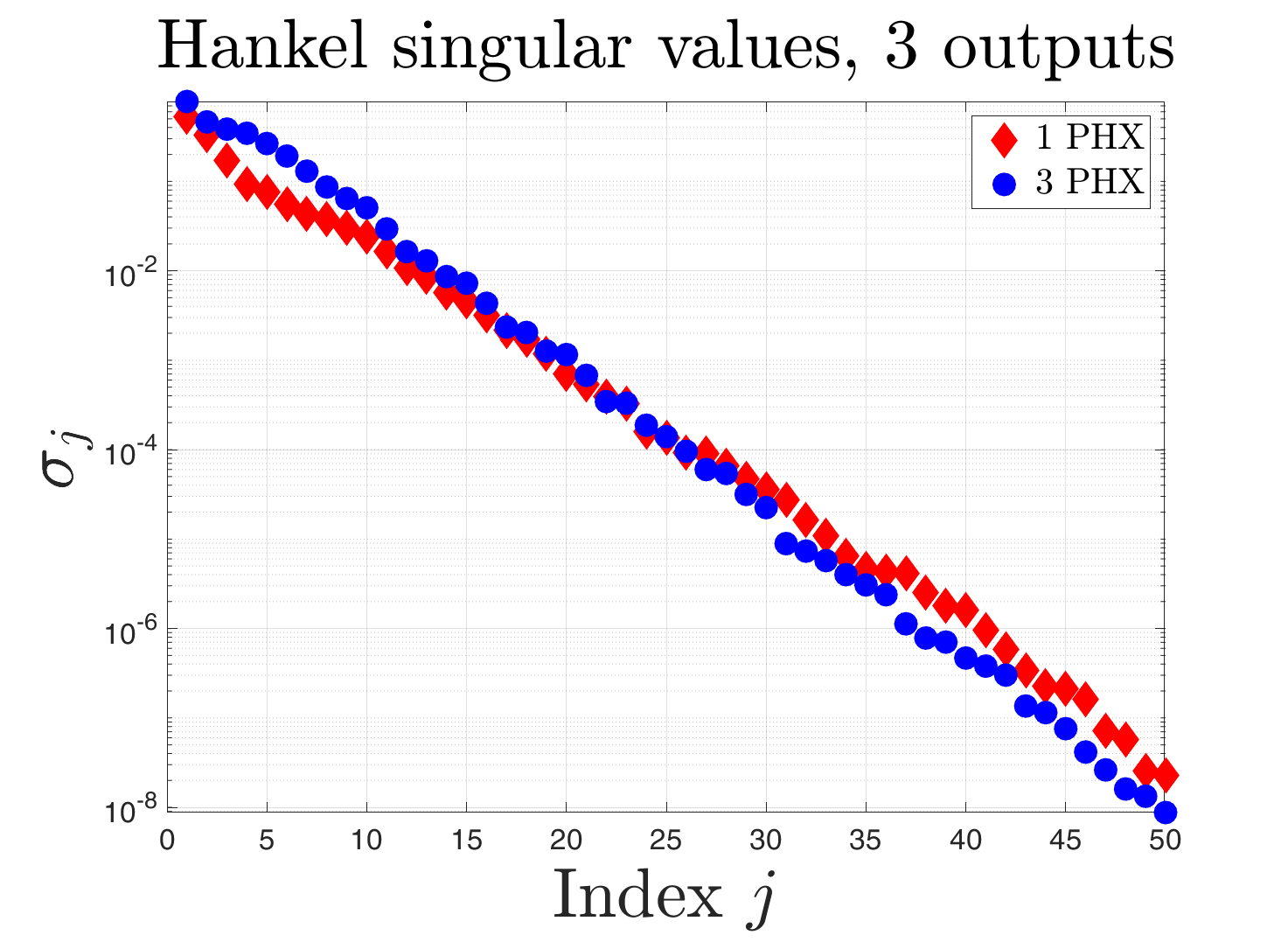}
			\includegraphics[width=0.49\linewidth]{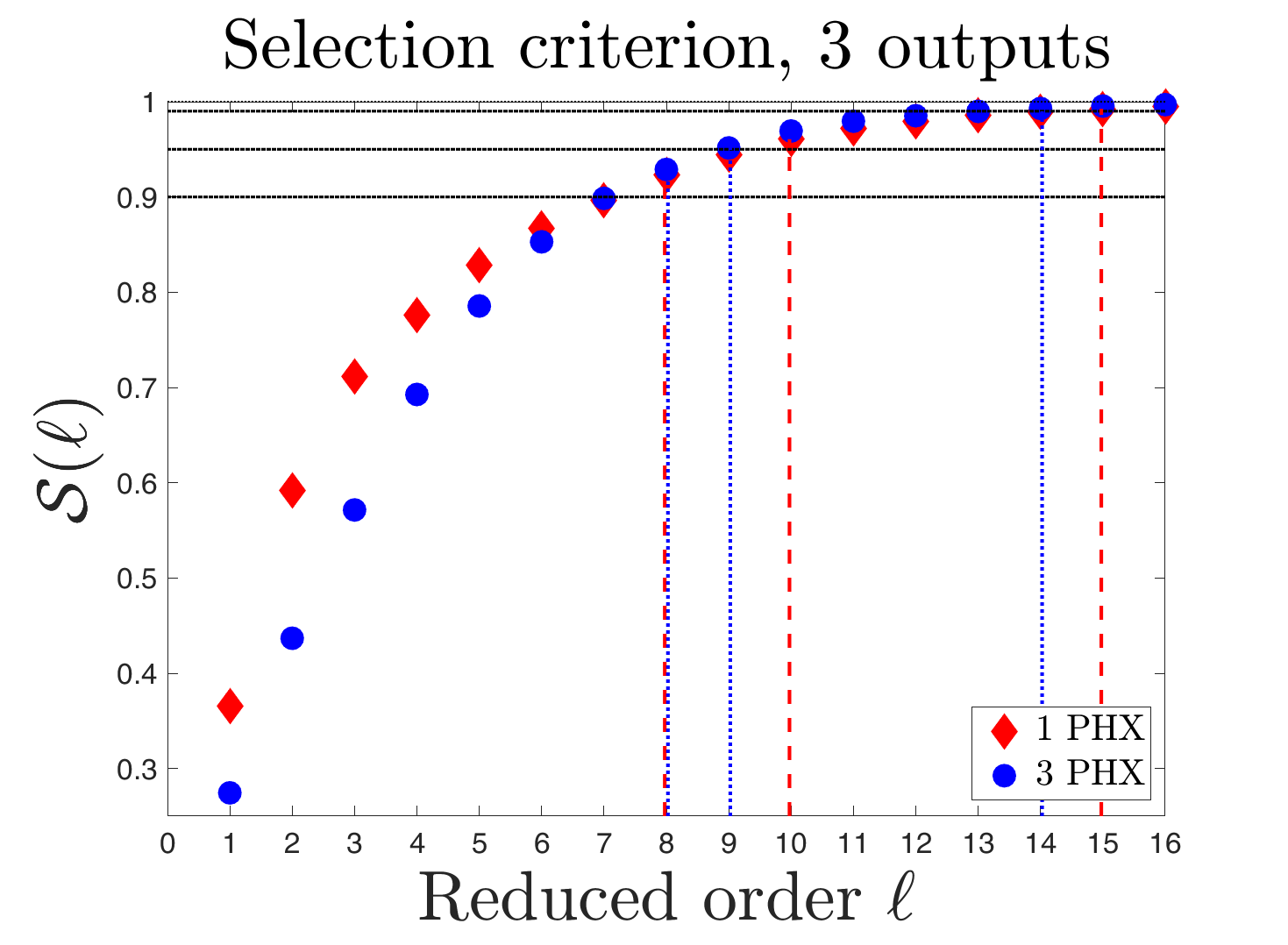}	
			\caption{Model with three outputs $Z=(\Qm,\Qf,\Qout)^\top$:\newline 
				Left: first 50 largest Hankel singular values, 
				Right: selection criterion.  }
			\label{selection3Cmo}
		\end{figure}

		The setting is analogous to Subsec.~\ref{subsec:num_ex2}. The input function $g$ is given in \eqref{ex:input} and the $3\times n$ output matrix $\omatrix$ is formed by the three rows $\OutputM,\OutputF,\OutputOut$ which are given in our companion paper \cite[Sec.~4]{takam2021shortb}.

		\begin{figure}[h!]
			\centering
			\includegraphics[width=.49\linewidth]{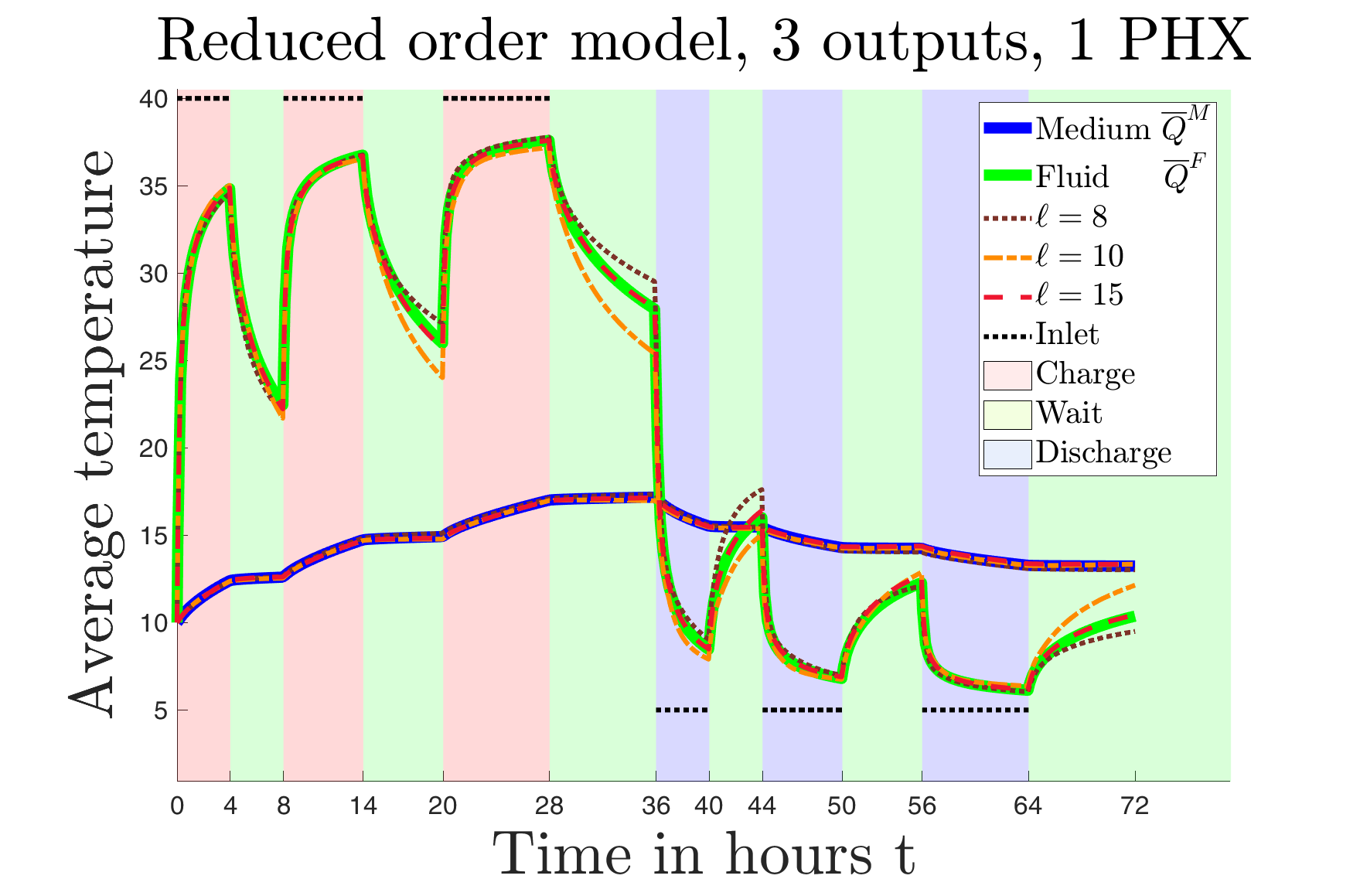}
			\includegraphics[width=.49\linewidth]{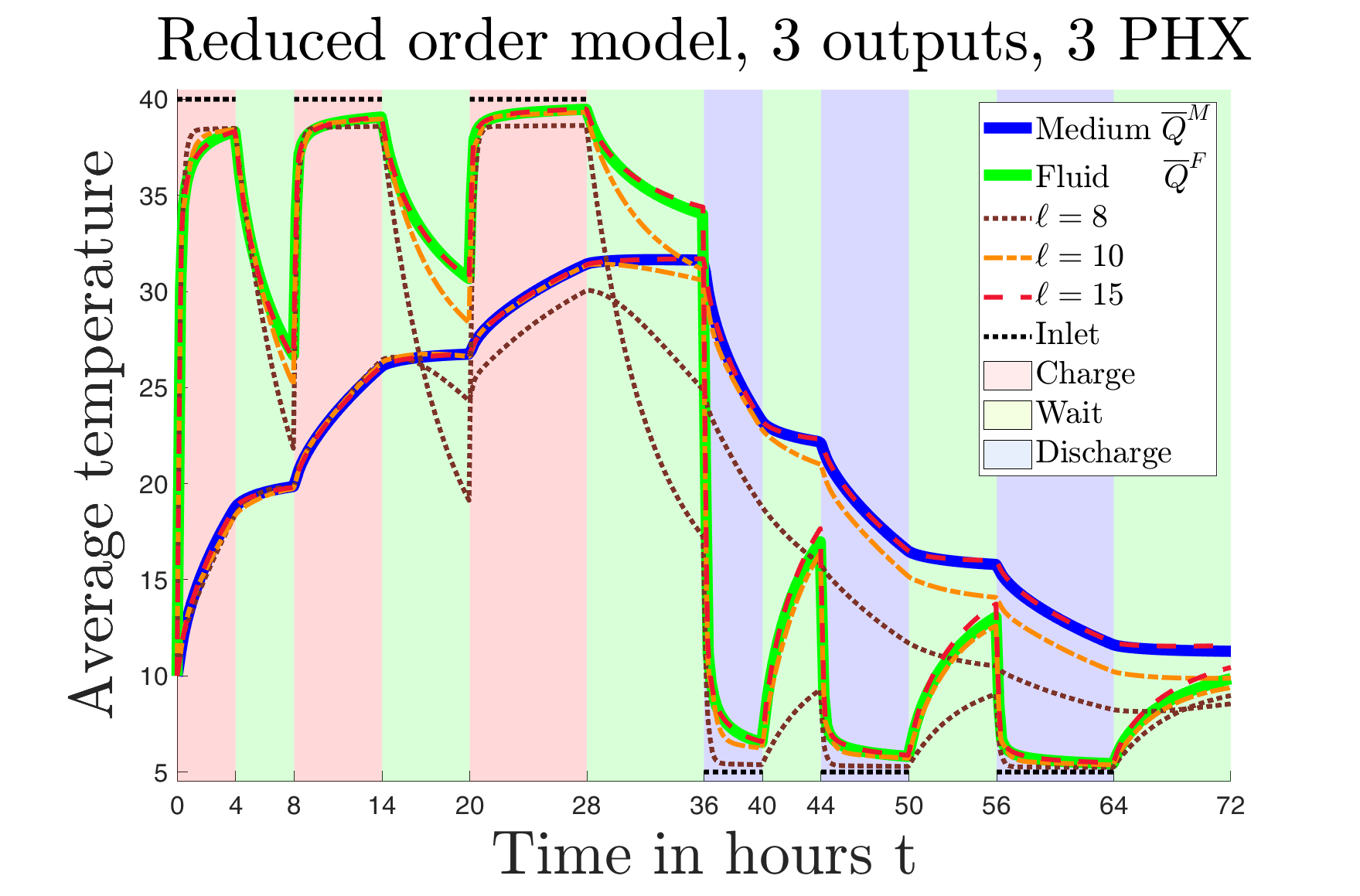}
			\includegraphics[width=.49\linewidth]{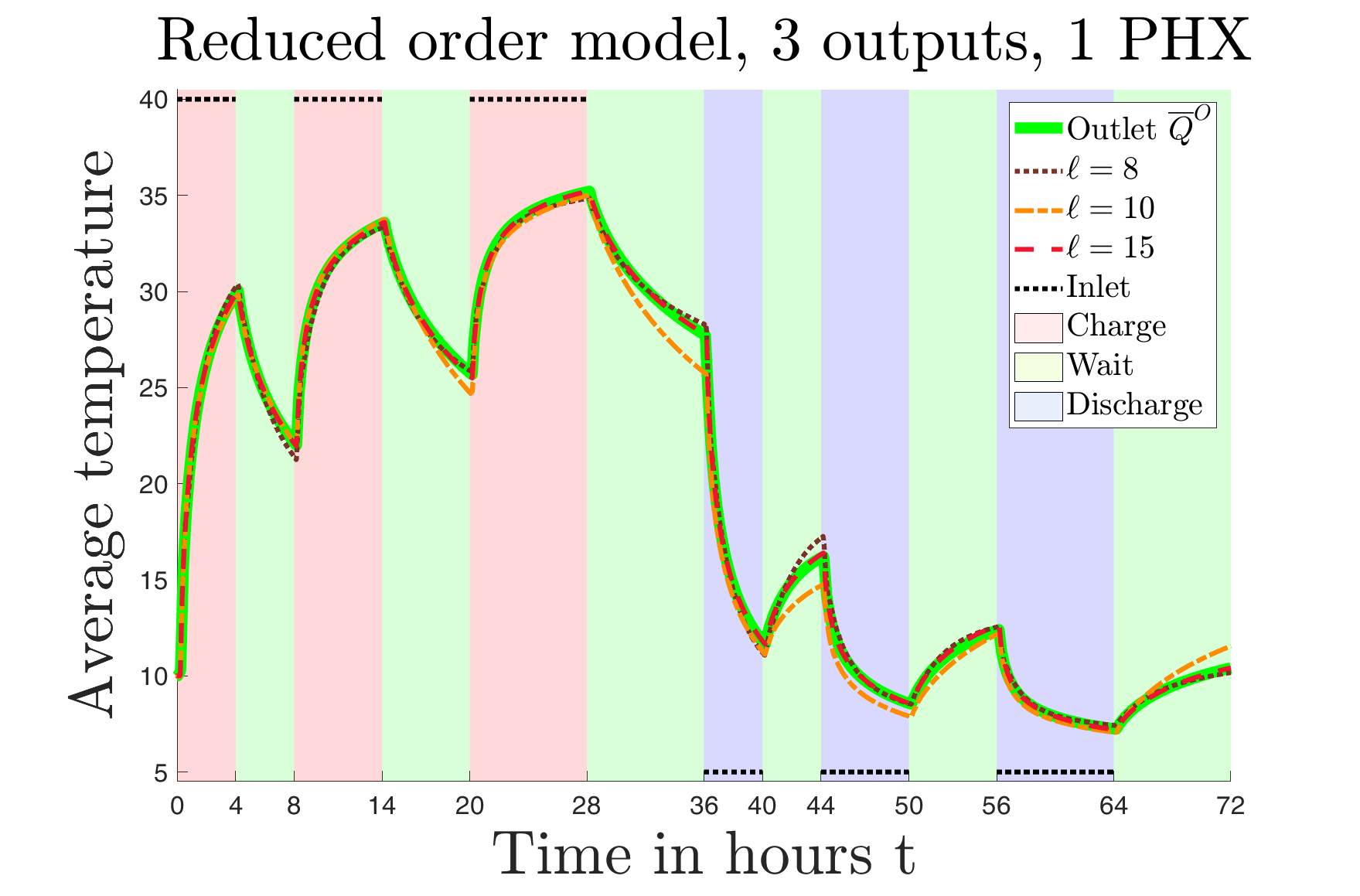}
			\includegraphics[width=.49\linewidth]{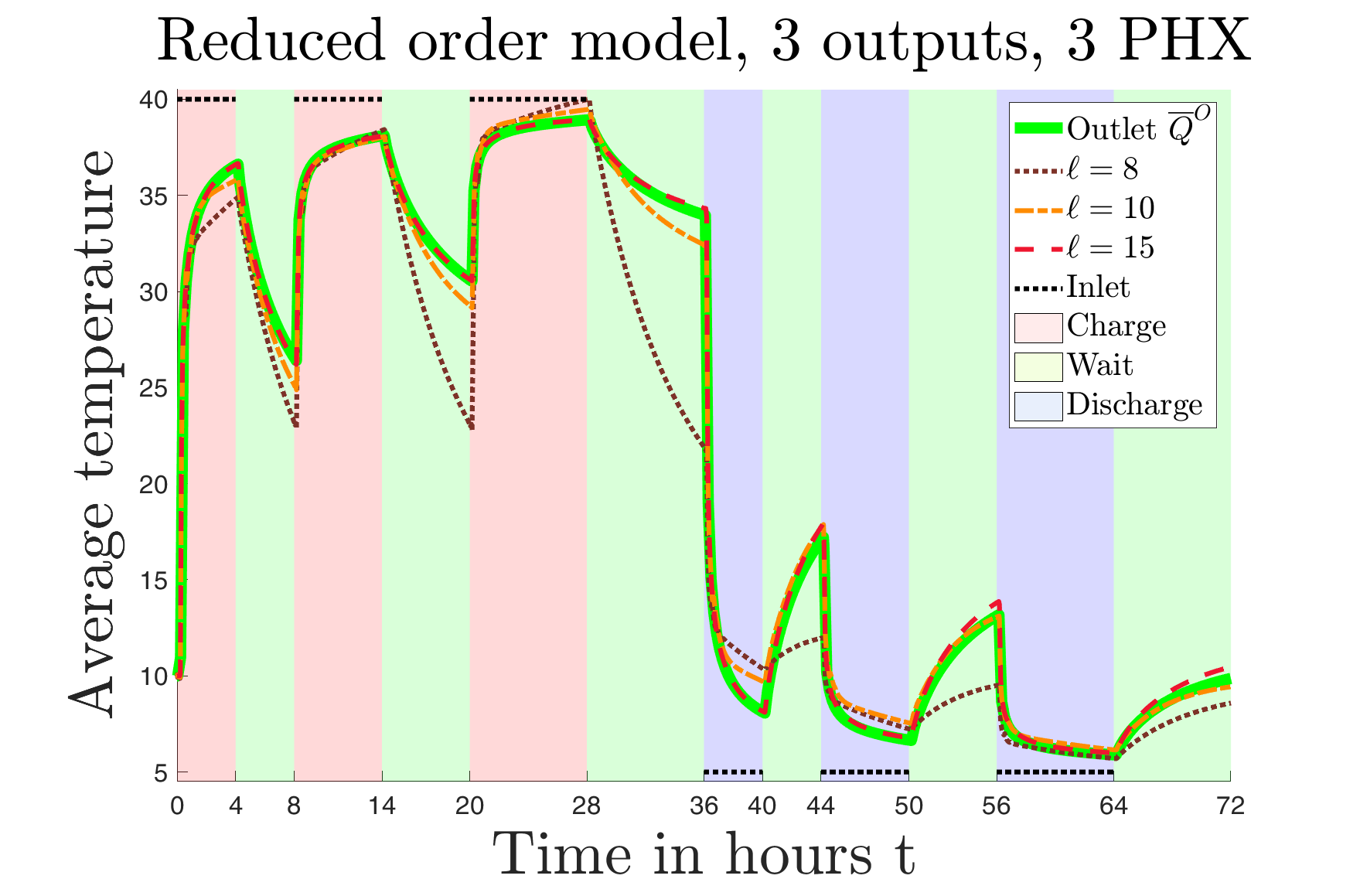}
			\mycaption{Model with three outputs $Z=(\Qm,\Qf,\Qout)^\top$: \quad Approximation of the output for $\dimred=8,10,15$. \newline  \begin{tabular}[t]{ll}
					Top: & Average temperatures in the medium $\Qm$ and the fluid $\Qf$,\\
					Bottom:&   Average temperature at the outlet $\Qout$,\\
					Left:&  one \phx, Right three \phxsk. 	
				\end{tabular}											
			}
			\label{BT3Cmo} 
		\end{figure}

		Fig.~\ref{selection3Cmo} shows as in the previous experiments 
		in the left panel the first 50 largest Hankel singular values whereas the right right panel shows the selection criteria. For the first 50 singular values we observe for both models that they are all distinct and decrease by 8 orders of magnitude which is slightly less than for the case of only two outputs. As in the examples with  one and two outputs  the first 20 singular values decrease faster for the model with one \phx than for the 3 \phx model. The selection criterion for the model with one \phx is for  $\dimred\le 6$ larger than for 3 \phxs and for $\dimred\ge 7$ slightly smaller.  Table \ref{tab:ReducedOrders} shows that for reaching threshold levels of $\alpha=90\%, 95\%, 99\%$ in the one \phx case $\dimred_\alpha=8,10,15$ states are required while for three \phxs one needs $\dimred_\alpha=8,9,14$ states, respectively.
		Thus, for dimension $\ell \geq 15 $ an almost perfect approximation of the input-output behavior can be expected.		
		Note that in the previous experiment with two outputs (without outlet temperature) reaching the above thresholds requires about 4 to 5 states less.
		
		In Fig.~\ref{BT3Cmo}  we plot  the output variables of the original and reduced-order system against time. The top panels show the average temperatures $Z_1(t)=\Qm(t)$ and $Z_2(t)=\Qf(t)$ in the medium and fluid which  are drawn as solid blue and green lines, respectively. The bottom panels depict the average temperature at the outlet $Z_3(t)=\Qout(t)$ by a solid  green line. The reduced-order approximations  are drawn  for  $\dimred=8,10,15$. As in the previous experiments with two outputs it can be observed that the approximation of $\Qm$ is better than for $\Qf$. The approximation errors for the outlet temperature $\Qout$ are quite similar to the errors for the fluid temperature. Note  that the outlet temperature represents  an average  of the spatial temperature distribution over the  quite small subdomain $\Dout$ on the boundary which is still smaller than the subdomain $\Df$ over which the average is taken for the fluid temperature $\Qf$. Both, the fluid and the outlet temperature show much larger temporal variations than  the temperature in medium $\Qm$. Again, errors are more pronounced during waiting periods than during charging and discharging and for the three \phx model the pointwise errors are larger than for the one \phx model.  For $\dimred\ge 15$ states the selection criterion is above  $99 \%$  and the  approximation errors are almost negligible.
		
		Fig.~\ref{BT3Cmo} also shows that the average temperatures of the fluid and at the outlet pipe, $\Qm$ and $\Qout$, exhibit almost the  same pattern during the charging, discharging and waiting periods. Hence, knowing the average fluid temperature one can simply predict the outlet temperature and remove   $\Qout$ from the output variables. Then we are back in the setting of the two output experiment in Subsec.~\ref{subsec:num_ex2} and need 4 to 5 states  less to capture the input-output behaviour with the same approximation quality.  Below in Subsec.~\ref{subsec:num_ex3b} we consider a model where instead of removing $\Qout$ from the output this quantity is replaced by the average bottom temperature $\Qbottom$ leading again to a model with three outputs.
		
		An alternative evaluation of the approximation quality can be derived from  Fig.~\ref{error3Cmo} which plots for the reduced orders $\dimred$ considered above the $\Ltwo$-error  $\|Z-\widetilde{Z}\|_{\Ltwo(0,t)}$ against time $t$ together with the  error bounds from Theorem \ref{theo_errorbound}. The results are similar to Fig.~\ref{error3Cmf2} and we refer for the interpretation to the end of the previous subsection.

		\begin{figure}[!h]
			\centering
			\includegraphics[width=.49\linewidth]{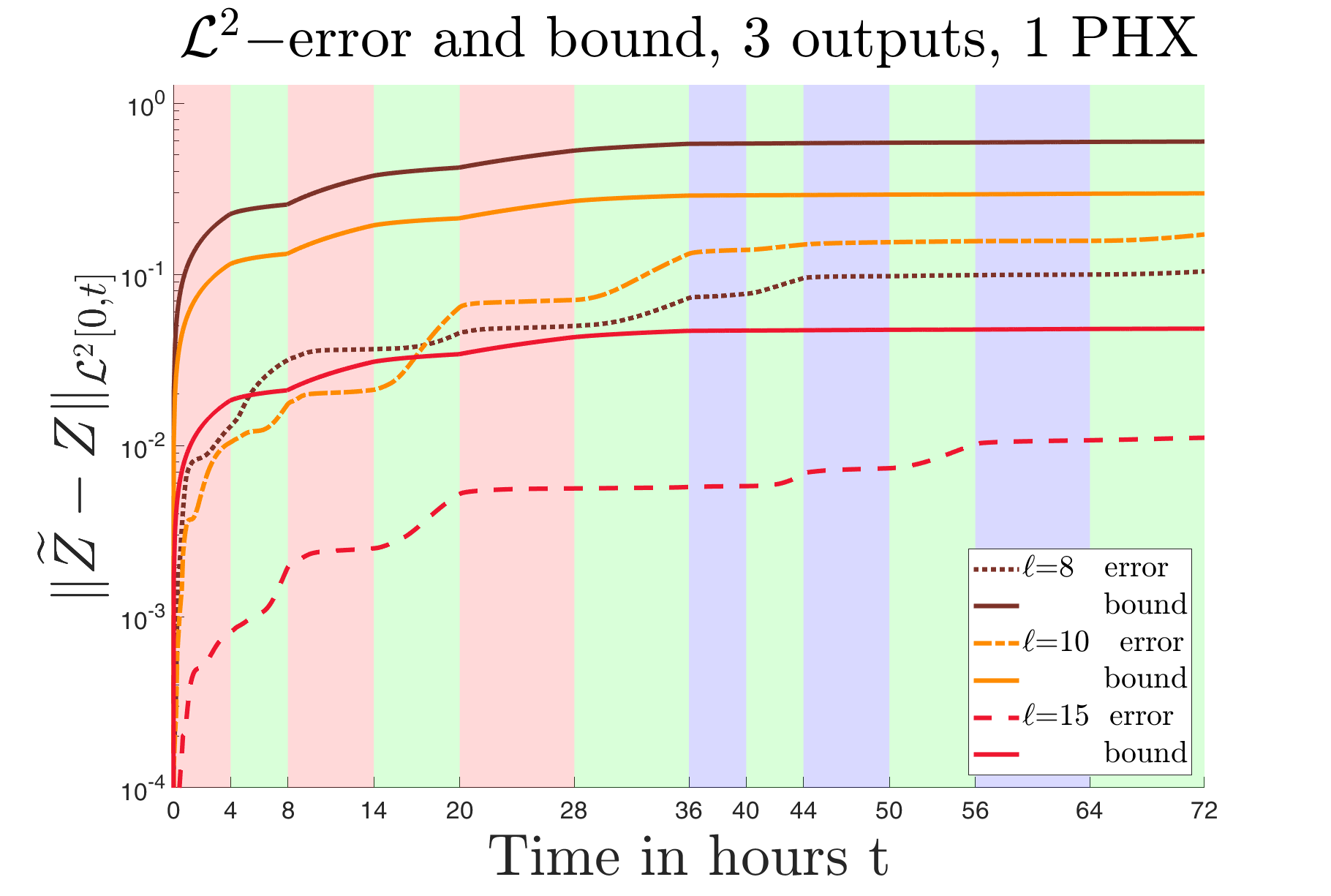}
			\includegraphics[width=.49\linewidth]{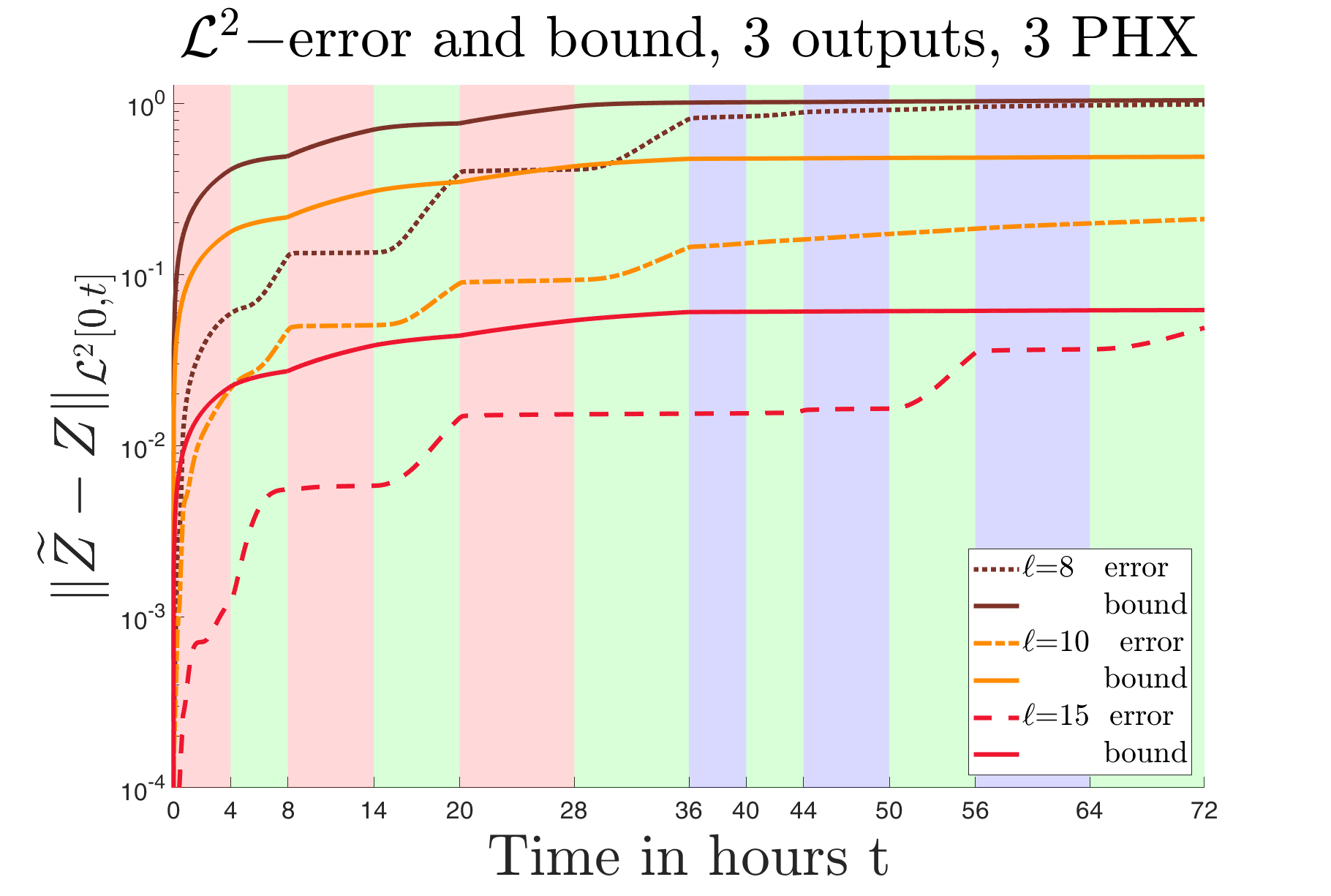}
			
			\caption{Model with three outputs $Z=(\Qm,\Qf,\Qout)^\top$:  $\Ltwo$-error and error bound for~$\dimred=8,10,15$.\newline  Left: one \phx, Right three \phxsk. 	}	
			\label{error3Cmo}
		\end{figure}	
		
		\subsection{Three Aggregated Characteristics II: ~$\Qm,\Qf,\Qbottom$}			
		\label{subsec:num_ex3b}
		As already announced above in this experiment we again consider a model with three outputs but instead of the outlet temperature $\Qout$ now the third output is  the average temperature at the bottom boundary $\Qbottom$. Hence, the output is  $Z=(\Qm,\Qf,\Qbottom)^\top$. We recall that the bottom boundary is open and not insulated and the temperature  $\Qbottom$  is of crucial importance for the quantification of gains and losses of thermal energy resulting from the heat transfer to the underground of the geothermal storage. We refer to Eq.~\eqref{RB} and the explanations in  Subsec. \ref{subsec:AggCharBoundary} and the Robin boundary condition \eqref{Robin} modeling that heat transfer from the storage to the underground.

		\begin{figure}[h!]
			\centering 
			\includegraphics[width=0.49\linewidth]{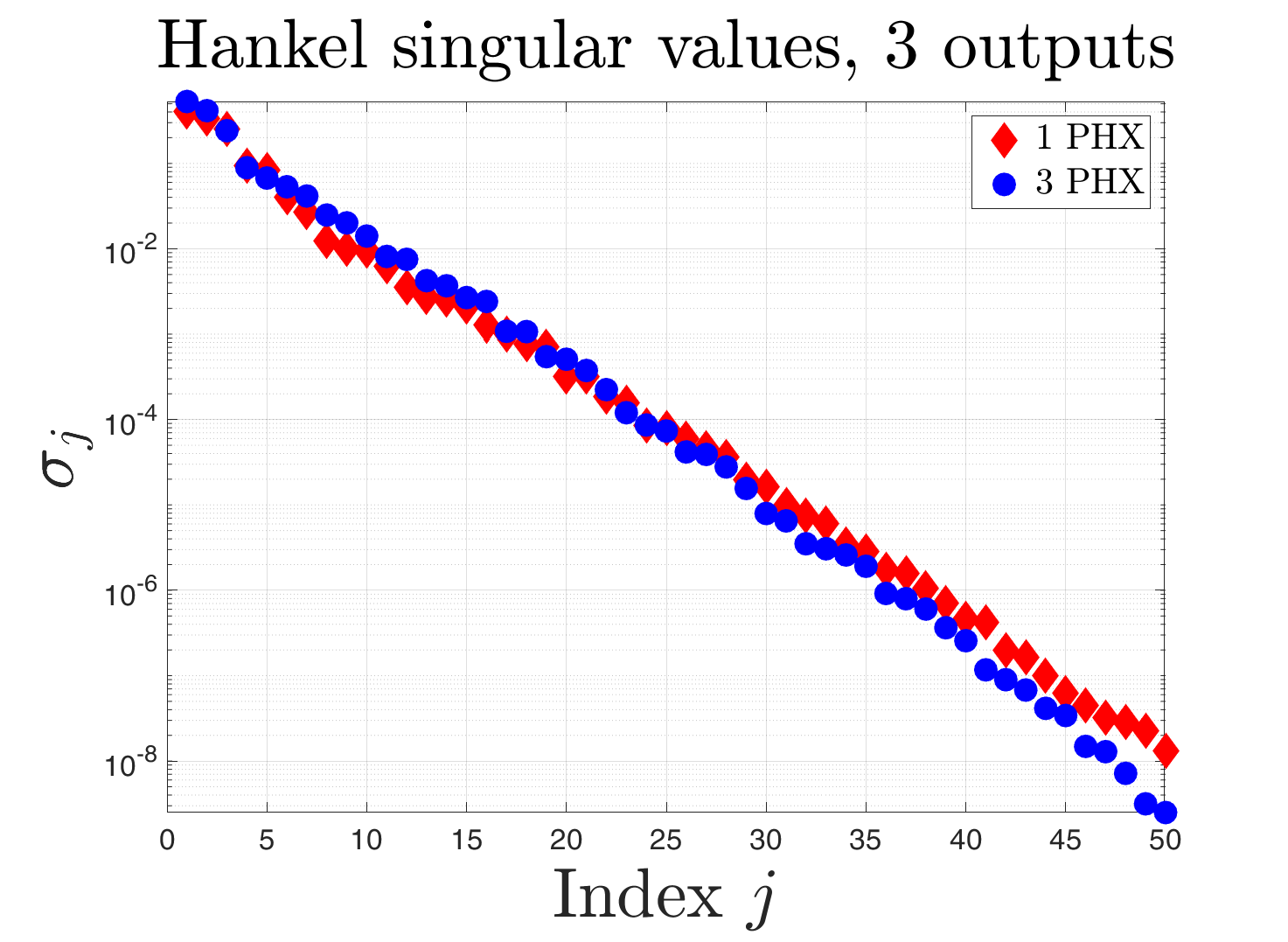}
			\includegraphics[width=0.49\linewidth]{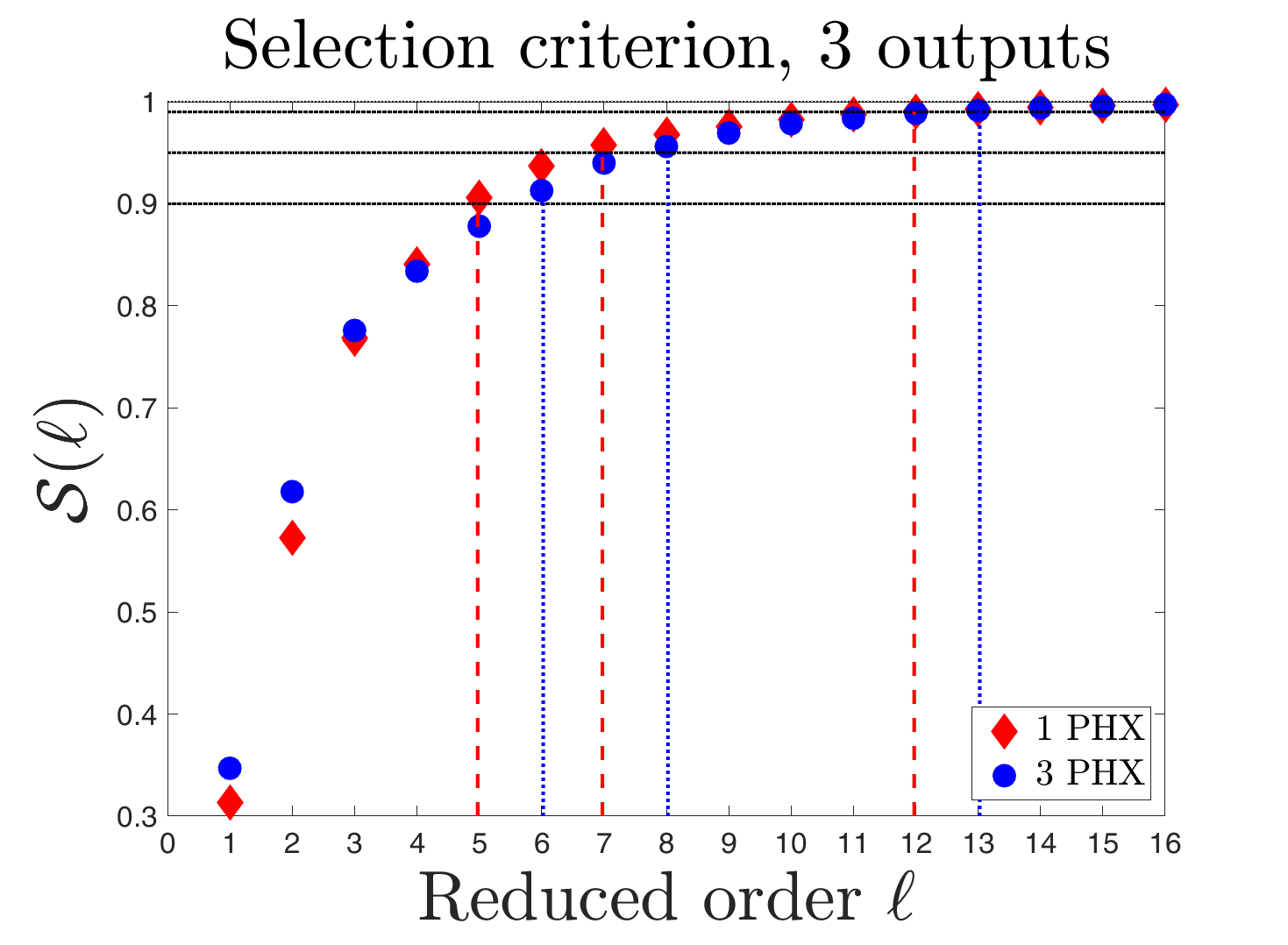}	
			\caption{Model with three outputs $Z=(\Qm,\Qf,\Qbottom)^\top$:\newline 
				Left: first 50 largest Hankel singular values, 
				Right: selection criterion.  }
			\label{selection3Cmb}
		\end{figure}
		
		The setting is analogous to Subsec.~\ref{subsec:num_ex2}. The input function $g$ is given in \eqref{ex:input} and the $3\times n$ output matrix $\omatrix$ is formed by the four rows $\OutputM,\OutputF,\OutputBottom$ which are given in our companion paper \cite[Sec.~4]{takam2021shortb}.
		
		Fig.~\ref{selection3Cmb} shows
		in the left panel the first 50 largest Hankel singular values whereas the right right panel shows the selection criteria. For  both \phx models  the first 50 singular values  are distinct and decrease by  more than 8 orders of magnitude which is only slightly less than for the case of two outputs. As in the previous experiments the first 20 singular values decrease faster for the model with one \phx than for the 3 \phx model.  The selection criterion for the model with one \phx is for  $\dimred\le 3$ smaller than for 3 \phxs and for $\dimred\ge 4$ slightly larger. 
		From the figure and also from Table \ref{tab:ReducedOrders} it can be seen that for reaching threshold levels of $\alpha=90\%, 95\%, 99\%$ in the one \phx case $\dimred_\alpha=5,7,12$ states are required while for three \phxs one needs $\dimred_\alpha=6,8,13$ states, respectively.
		Thus, for dimension $\ell \geq 13 $ an almost perfect approximation of the input-output behavior can be expected.	
		
		A comparison with the two-output model in Subsec.~\ref{subsec:num_ex2} with output  $Z=(\Qm,\Qf)^\top$ shows that the additional third output variable $\Qbottom$ requires only one or two more state variables to ensure the same approximation quality.	However, the three-output model considered above in Subsec.~\ref{subsec:num_ex3} where the third output is the average  outlet temperature $\Qout$ requires two  or three states more in the reduced-order system to ensure the same approximation quality. This shows that  $\Qbottom$ is much easier to reconstruct by a reduced-order model than $\Qout$. 
		An explanation is that 
		for $\Qout$ the spatial temperature distribution is averaged over the subdomain $\Dout$ which is much smaller than the corresponding domain $\Dbottom$ over which the average is taken for $\Qbottom$. Further, due to charging and discharging  via the \phxs the outlet temperature shows much larger temporal variations than  the temperature at the bottom $\Qm$.
		
		\begin{figure}[h!]
			\centering
			\includegraphics[width=.49\linewidth]{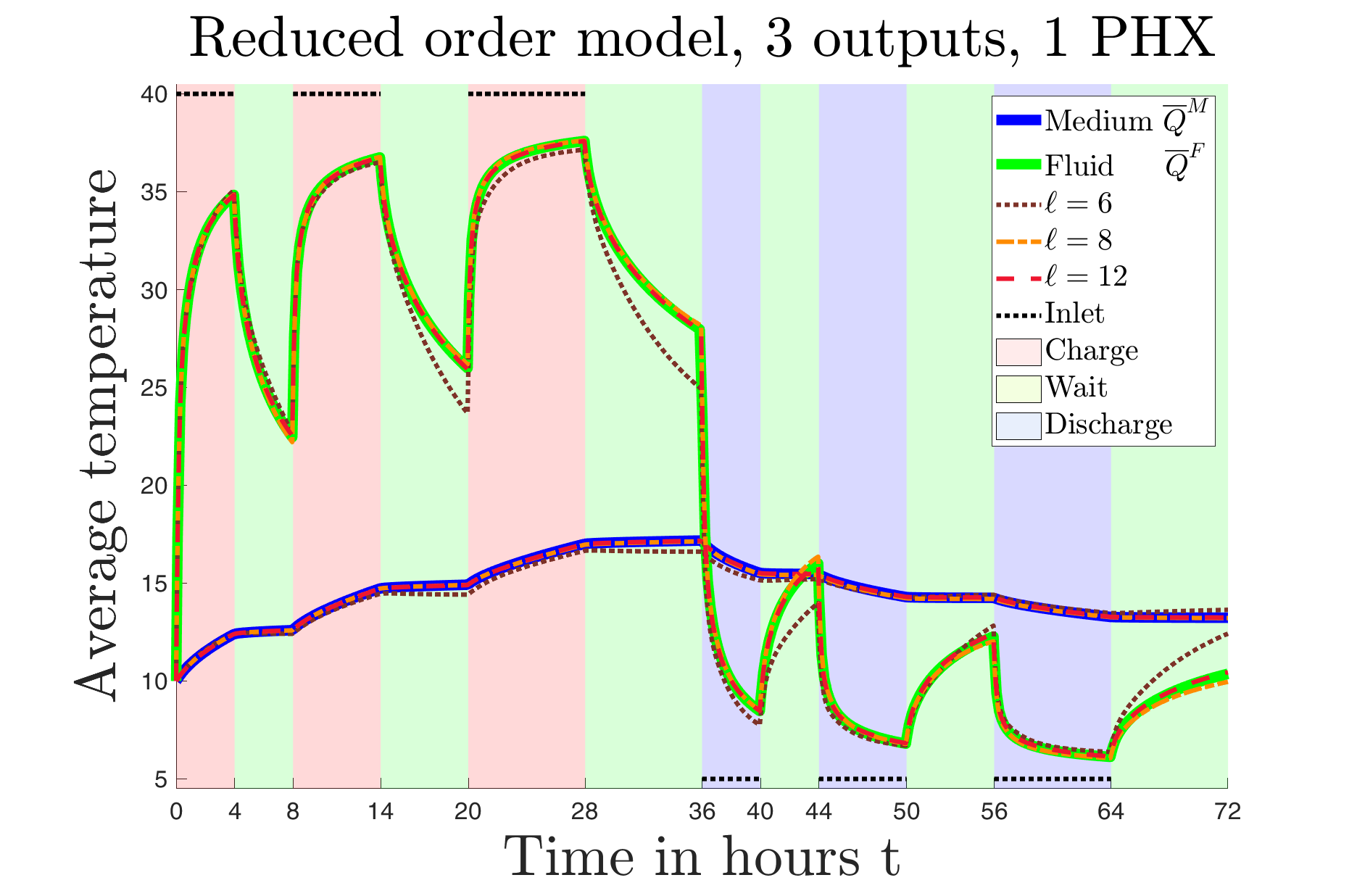}
			\includegraphics[width=.49\linewidth]{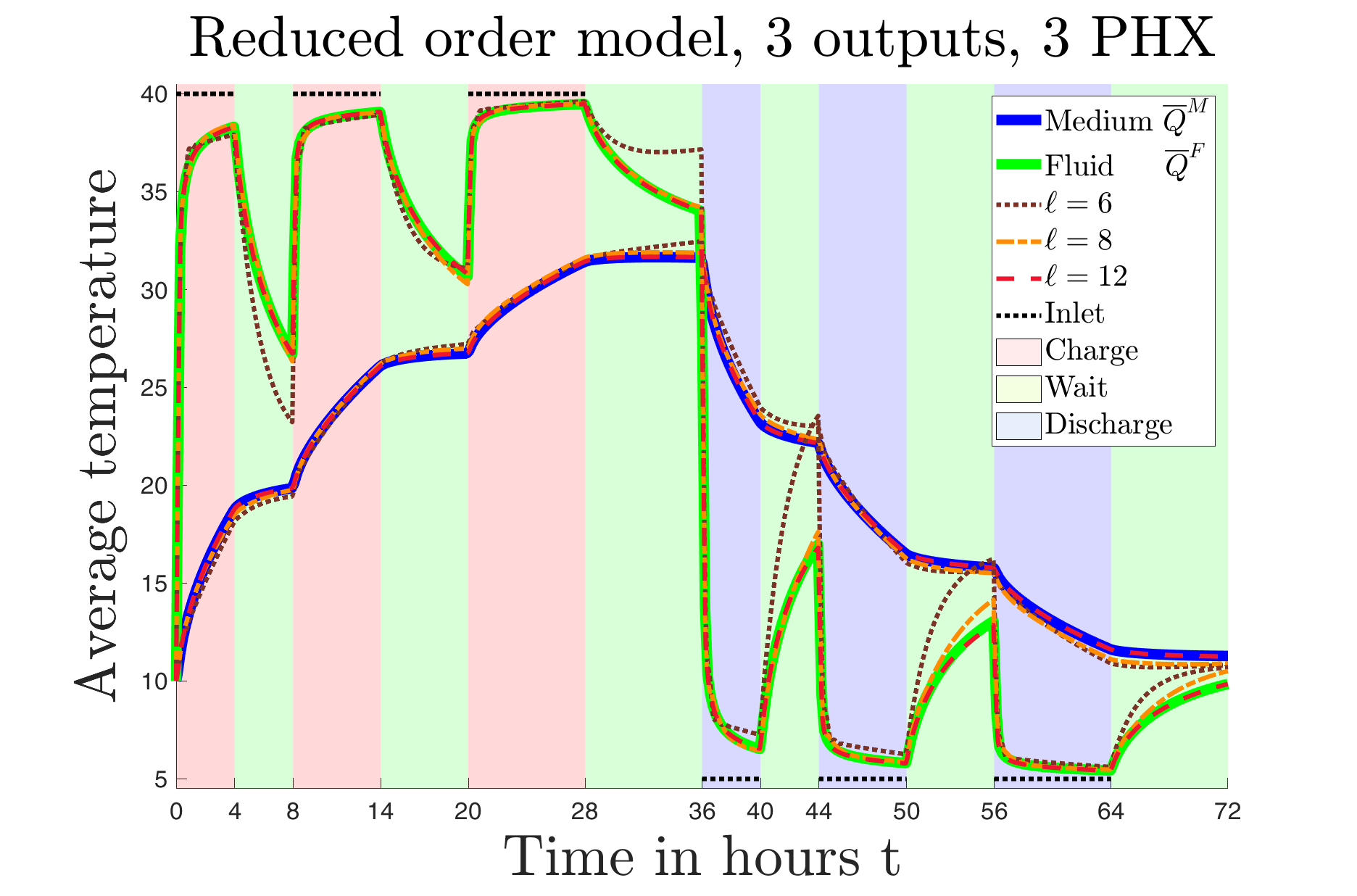}
			\includegraphics[width=.49\linewidth]{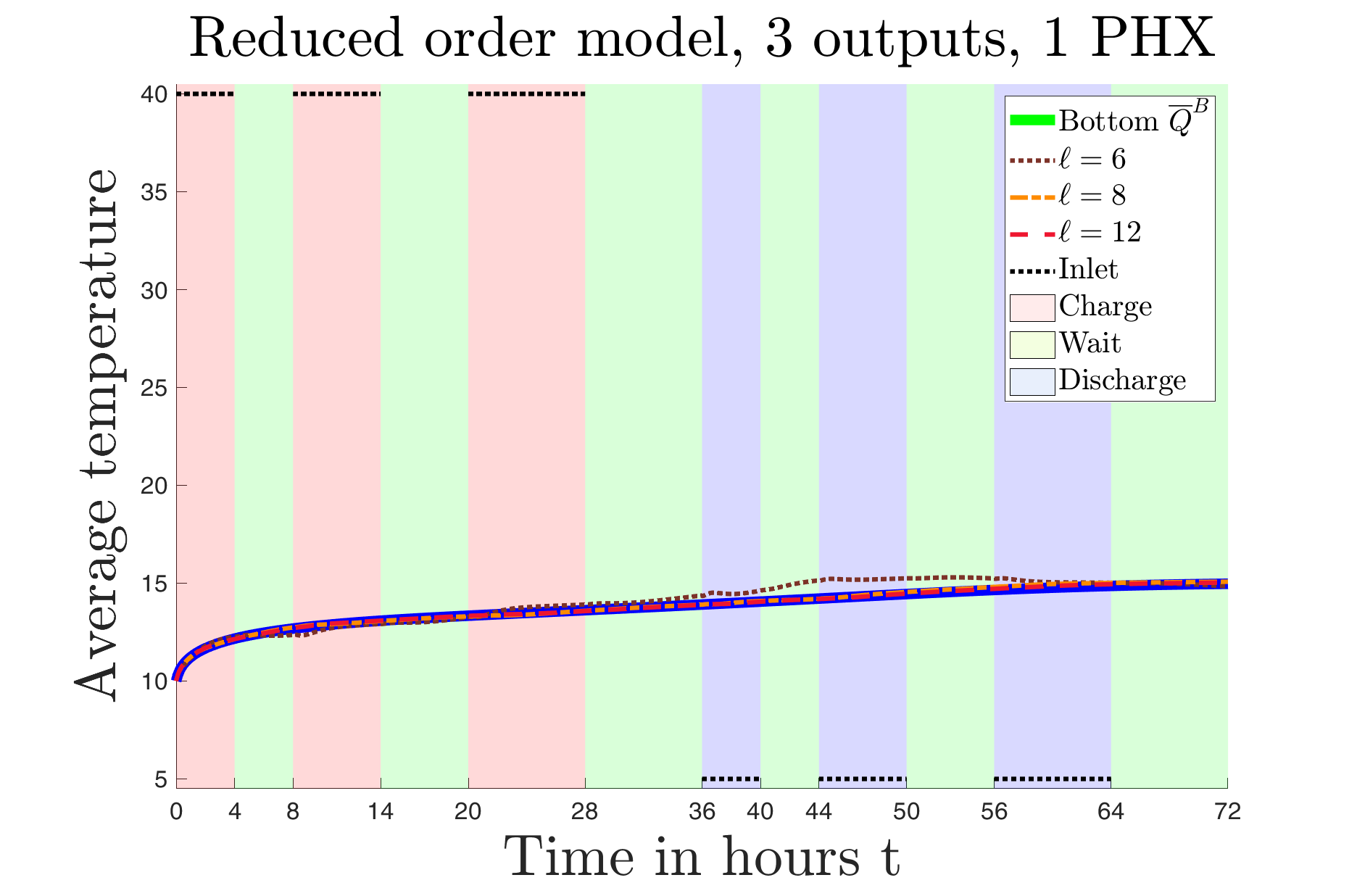}
			\includegraphics[width=.49\linewidth]{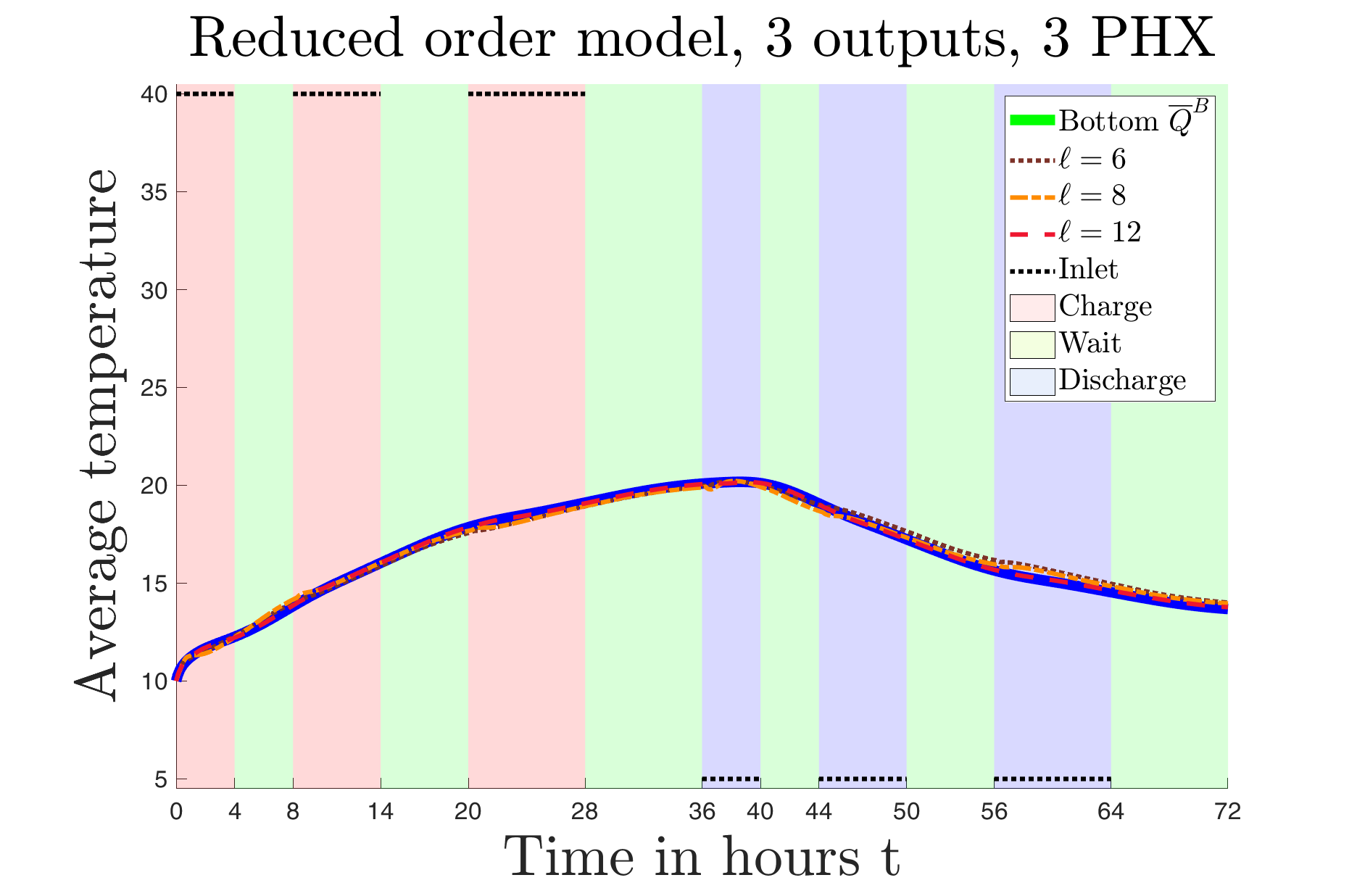}
			\mycaption{Model with three outputs $Z=(\Qm,\Qf,\Qbottom)^\top$: \quad Approximation of the output for $\dimred=6,8,12$. \newline  \begin{tabular}[t]{ll}
					Top: & Average temperatures in the medium $\Qm$ and the fluid $\Qf$,\\
					Bottom:&   Average bottom temperature  $\Qbottom$,\\
					Left:&  one \phx, Right three \phxsk. 	
				\end{tabular}											
			}
			\label{BT3Cmb}
		\end{figure}
		
		Fig.~\ref{BT3Cmb}  shows the output variables of the original and reduced-order system which are plotted against  time. In the top panels the average temperatures $Z_1(t)=\Qm(t)$ and $Z_2(t)=\Qf(t)$ in the medium and fluid are drawn as solid blue and green lines, respectively. The bottom panel depicts the average temperature at the bottom boundary $\Qbottom$ by a blue solid line. The reduced-order approximations  as drawn   for  $\dimred=6,8,12$.

		As in the previous experiments the approximation of $\Qm$ is much better than for $\Qf$. The approximation of the third output variable $\Qbottom$ is quite good although it represents an average of the spatial temperature distribution over the rather small subdomain $\Dbottom$ at the bottom boundary. Possible explanations are the relatively small temporal fluctuations of that quantity and the large distance of the bottom boundary to the \phxs where the charging and discharging generates large temporal and spatial fluctuations.
		
		In   Fig.~\ref{error3Cmb} we show  for the reduced orders $\dimred$ considered above the $\Ltwo$-error  $\|Z-\widetilde{Z}\|_{\Ltwo(0,t)}$ which plotted against time $t$ together with the  error bounds from Theorem \ref{theo_errorbound}. The results are similar to Fig.~\ref{error3Cmf2} and we refer for the interpretation to the end of  Subsec.~\ref{subsec:num_ex2}.Fig.~\ref{error3Cmb}: 
		\begin{figure}[!h]
			\centering
			\includegraphics[width=.49\linewidth]{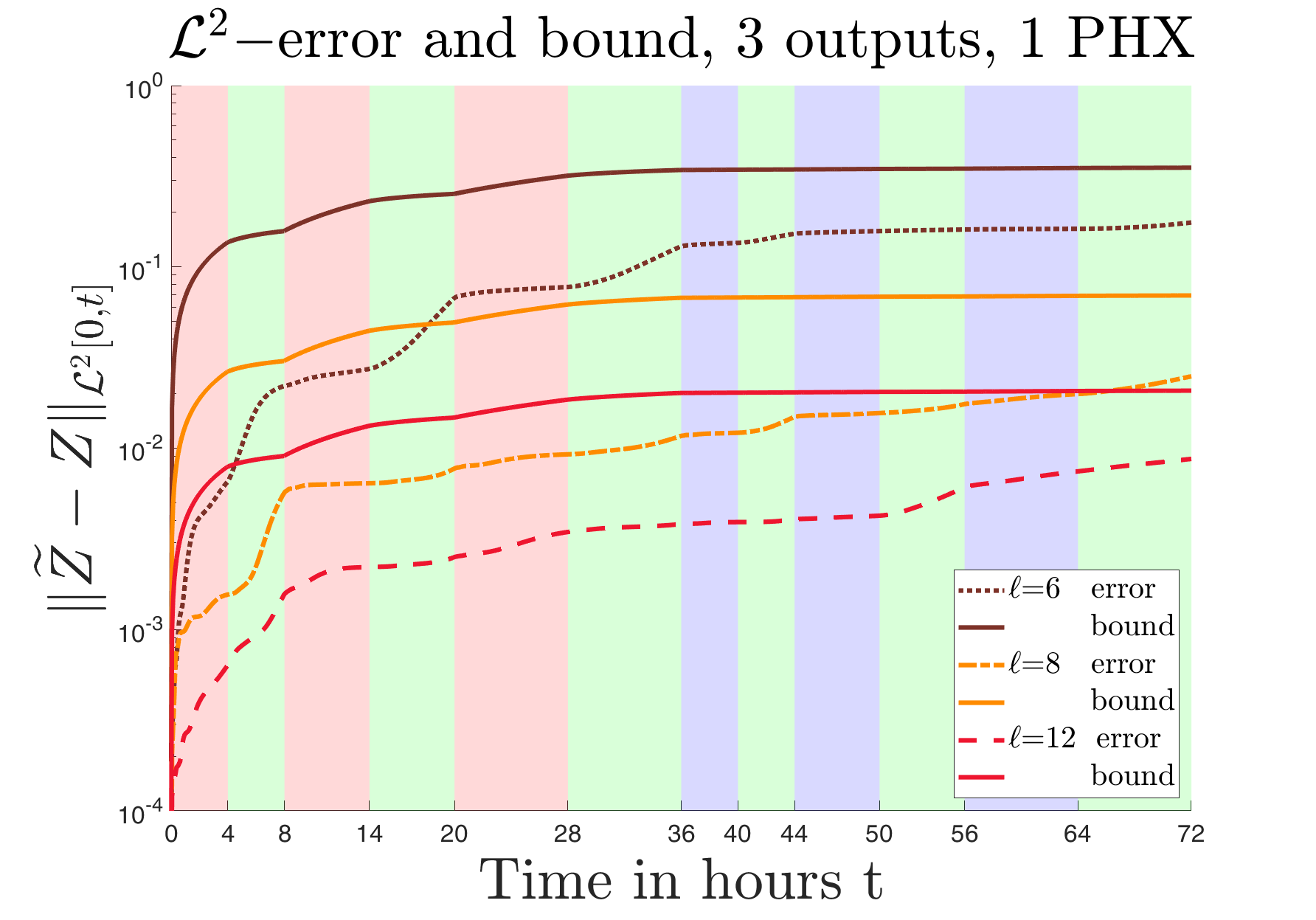}
			\includegraphics[width=.49\linewidth]{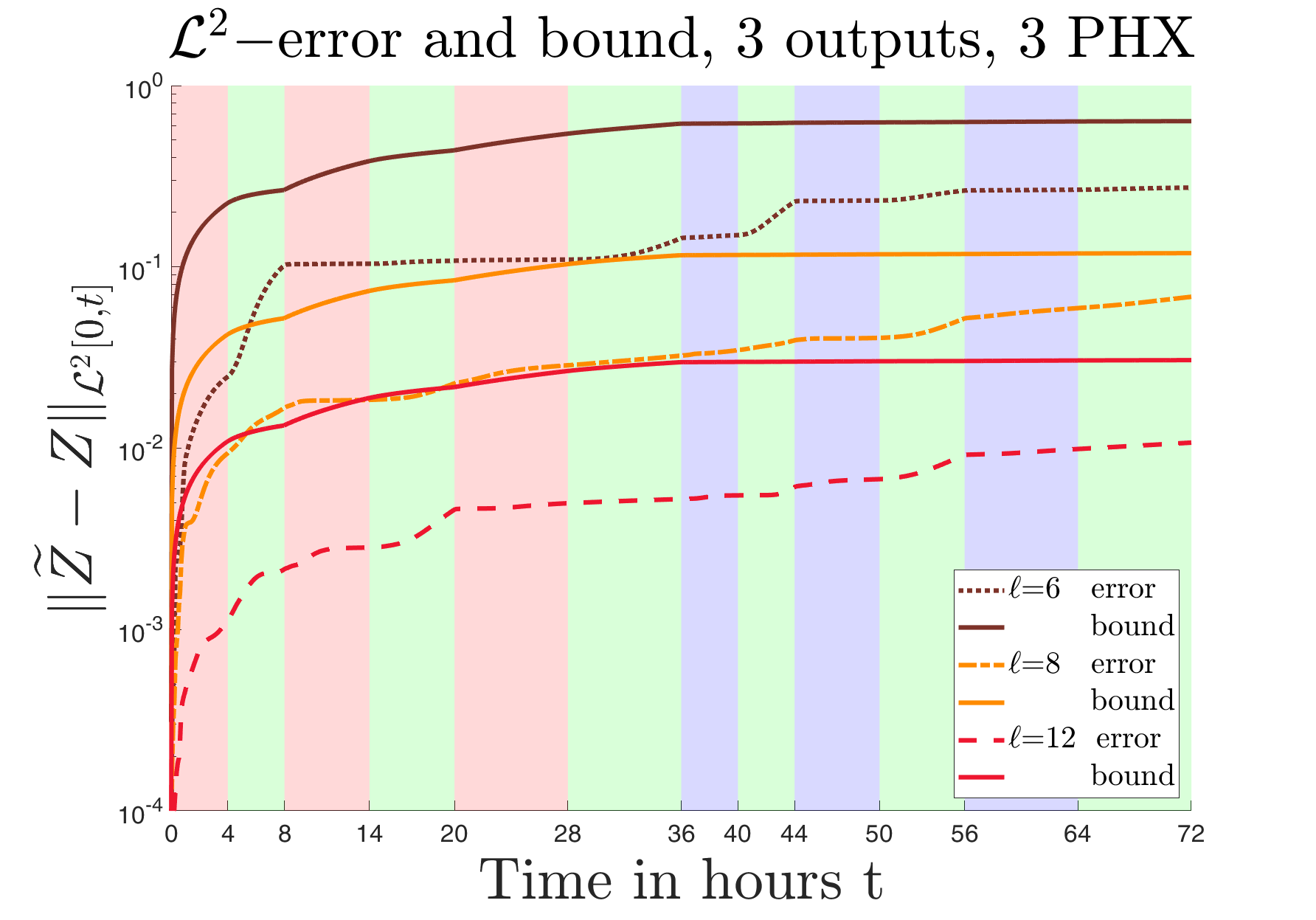}
			
			\caption{Model with three outputs $Z=(\Qm,\Qf,\Qbottom)^\top$:  $\Ltwo$-error and error bound for~$\dimred=6,8,12$.\newline  Left: one \phx, Right three \phxsk. 	}	
			\label{error3Cmb}
		\end{figure}								
		
		\subsection{Four Aggregated Characteristics: ~$\Qm,\Qf,\Qout,\Qbottom$}
		\label{subsec:num_ex4}								
		
		In this last experiment the output contains all of four  aggregated characteristics appearing in the above experiments and is given by   $Z=(\Qm,\Qf,\Qout,\Qbottom)^\top$.

		The setting is analogous to Subsec.~\ref{subsec:num_ex2}. The input function $g$ is given in \eqref{ex:input} and the $4\times n$ output matrix $\omatrix$ is formed by the four rows $\OutputM,\OutputF,\OutputOut,\OutputBottom$ which are given in our companion paper \cite[Sec.~4]{takam2021shortb}.
		
		\begin{figure}[h!]
			\centering 
			\includegraphics[width=0.49\linewidth]{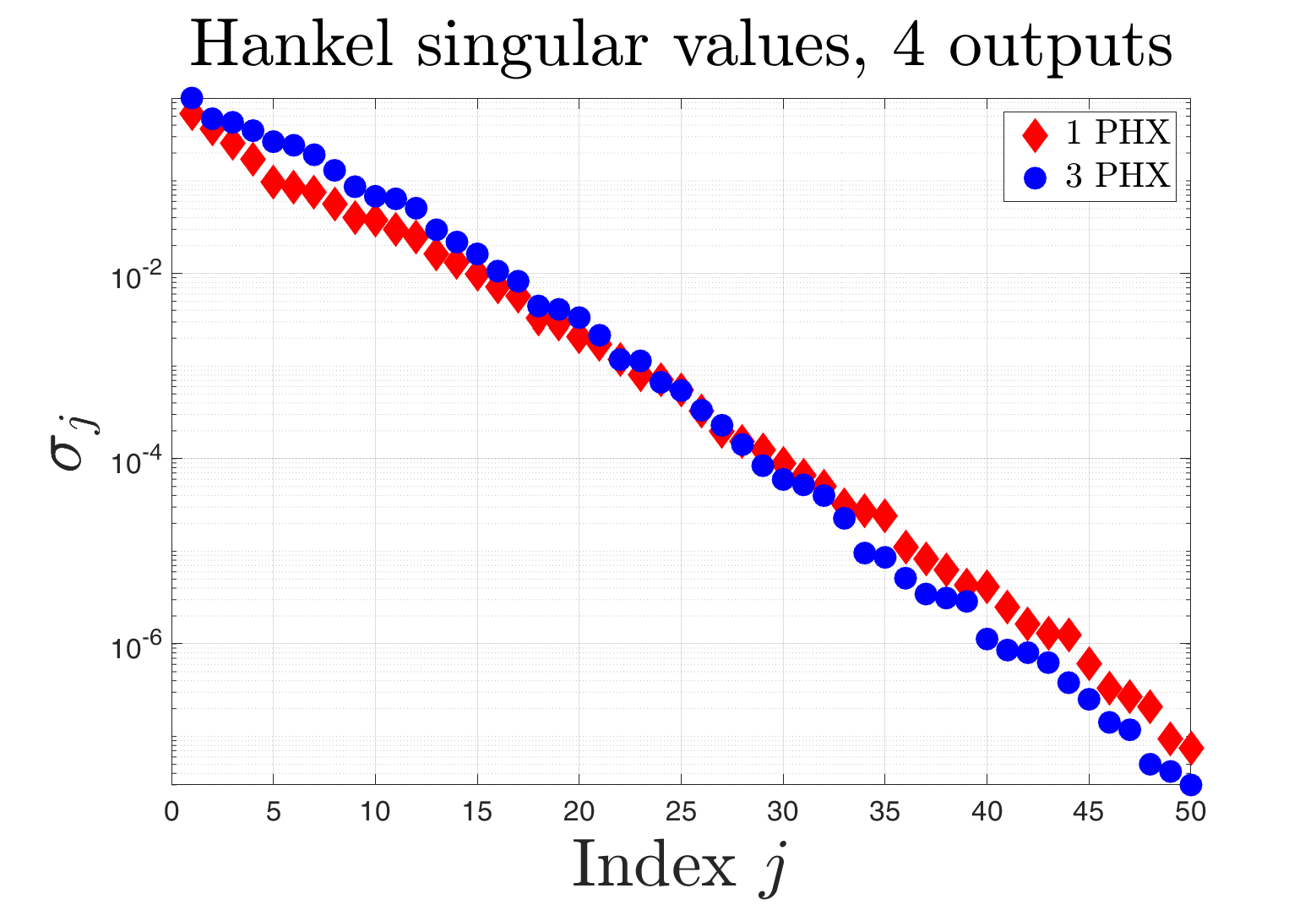}
			\includegraphics[width=0.49\linewidth]{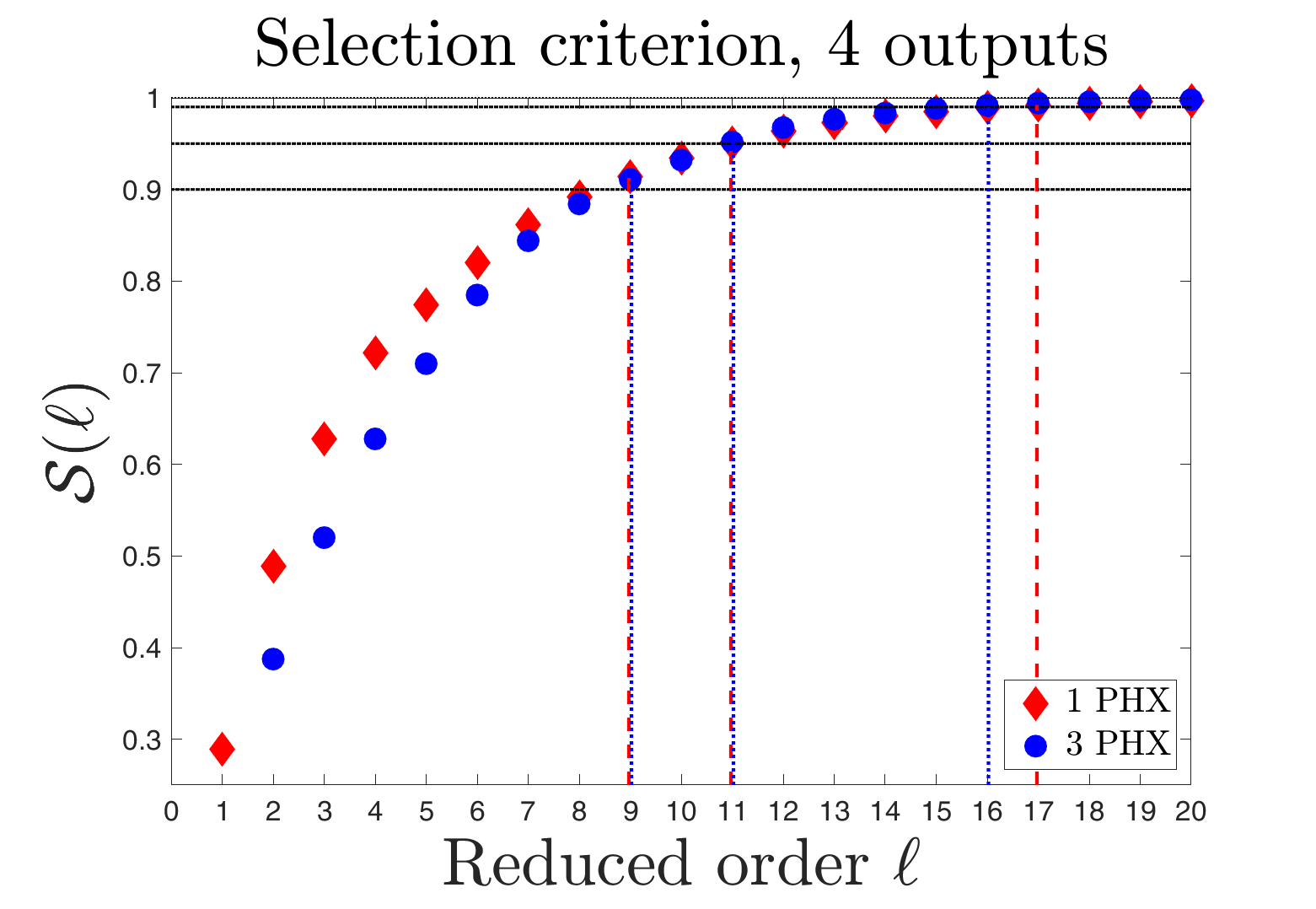}	
			\caption{Model with four outputs $Z=(\Qm,\Qf,\Qout,\Qbottom)^\top$:\newline 
				Left: first 50 largest Hankel singular values, 
				Right: selection criterion.  }
			\label{selection3Cmfob}
		\end{figure}
		
		Fig.~\ref{selection3Cmfob} shows 
		in the left panel the first 50 largest Hankel singular values whereas the right right panel shows the selection criteria. For  both \phx models  the first 50 singular values   are distinct and decrease by almost 8 orders of magnitude which is only slightly less than for the case of three outputs. As in the previous experiments the first 20 singular values decrease faster for the model with one \phx than for the 3 \phx model.  The selection criterion for the model with one \phx is for  $\dimred\le 10$ larger than for 3 \phxs and for $\dimred\ge 11$ slightly smaller. 						 
		From the figure and also from Table \ref{tab:ReducedOrders} it can be seen that for reaching threshold levels of $\alpha=90\%, 95\%, 99\%$ in the one \phx case $\dimred_\alpha=9,11,17$ states are required while for three \phxs one needs $\dimred_\alpha=9,11,16$ states, respectively.
		Thus, for dimension $\ell \geq 17 $ an almost perfect approximation of the input-output behavior can be expected.	A comparison with the three output model in Subsec.~\ref{subsec:num_ex3} with output  $Z=(\Qm,\Qf,\Qout)^\top$ shows that the additional fourth output variable $\Qbottom$ requires only one or two more state variables to ensure the same approximation quality.	This corresponds to our previous observations for the augmentation of the  output $Z=(\Qm,\Qf)^\top$ of the model considered in Subsec.~\ref{subsec:num_ex2} by adding as third output the average bottom  temperature  $\Qbottom$, see Subsec.~\ref{subsec:num_ex3b}. There the minimal reduced orders $\dimred_\alpha$ also increase  only by one ore two.

		\begin{figure}[h!]
				\centering							 		
				\includegraphics[width=.49\linewidth]{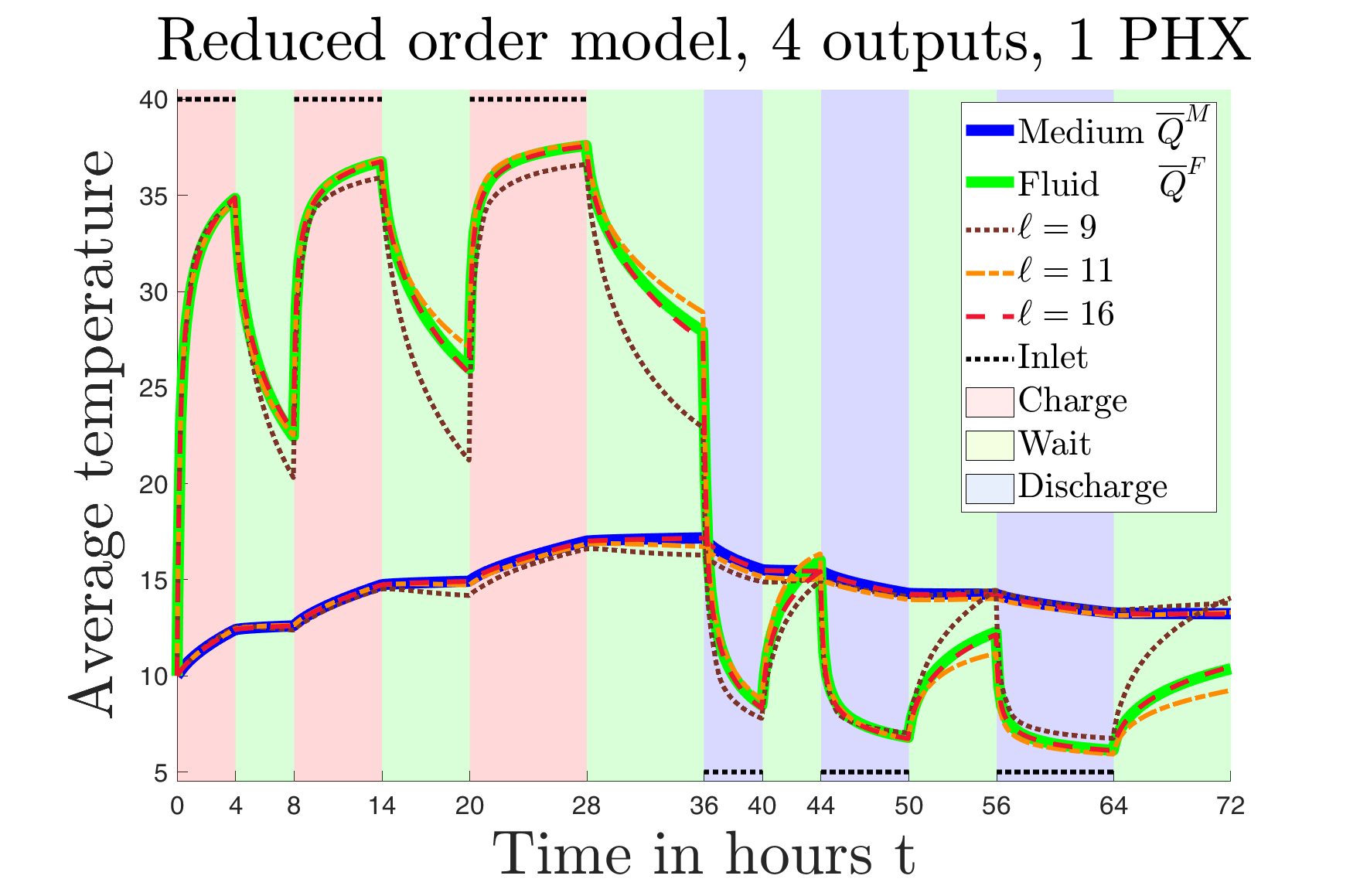}
				\includegraphics[width=.49\linewidth]{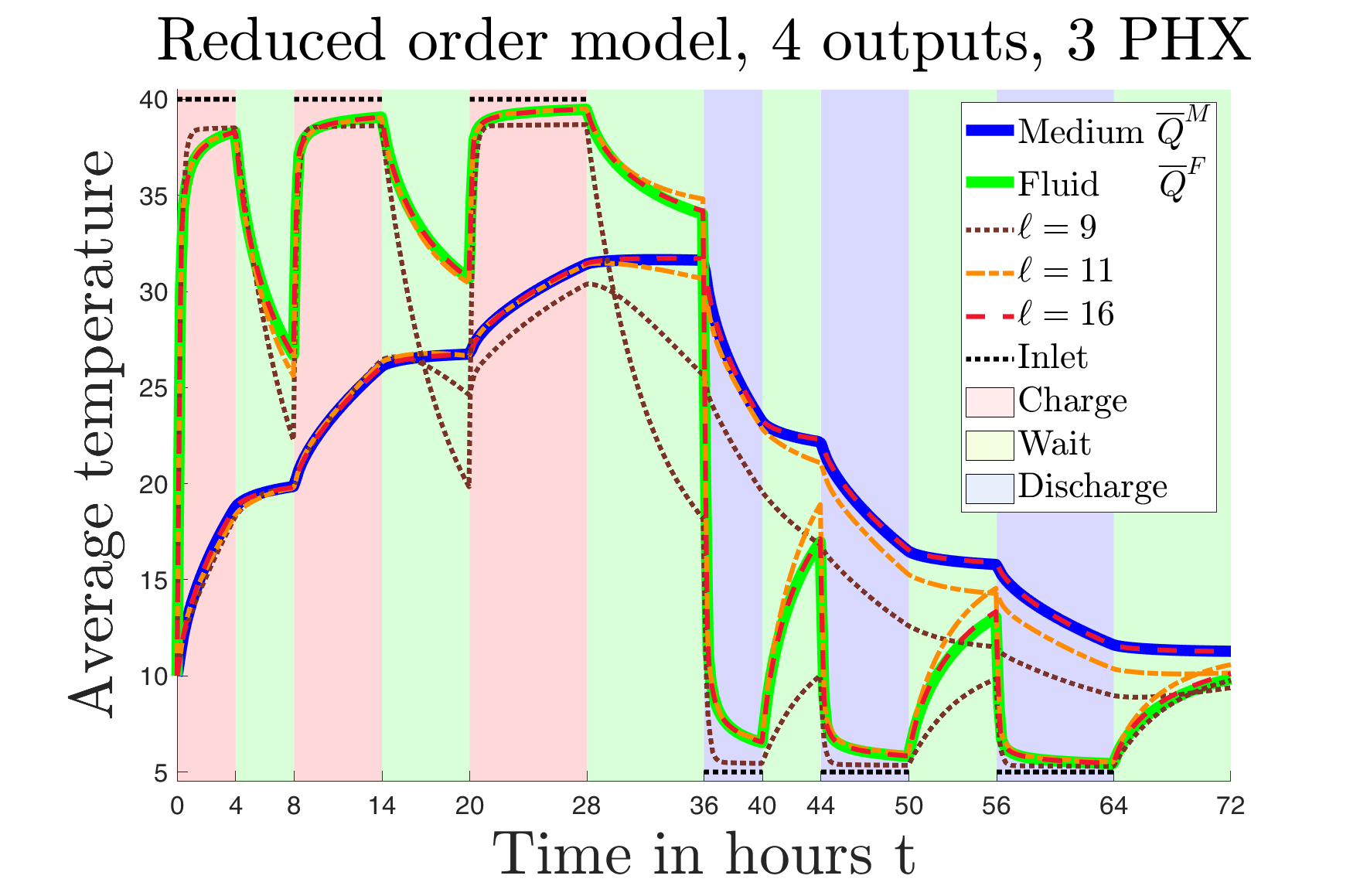}
				\includegraphics[width=.49\linewidth]{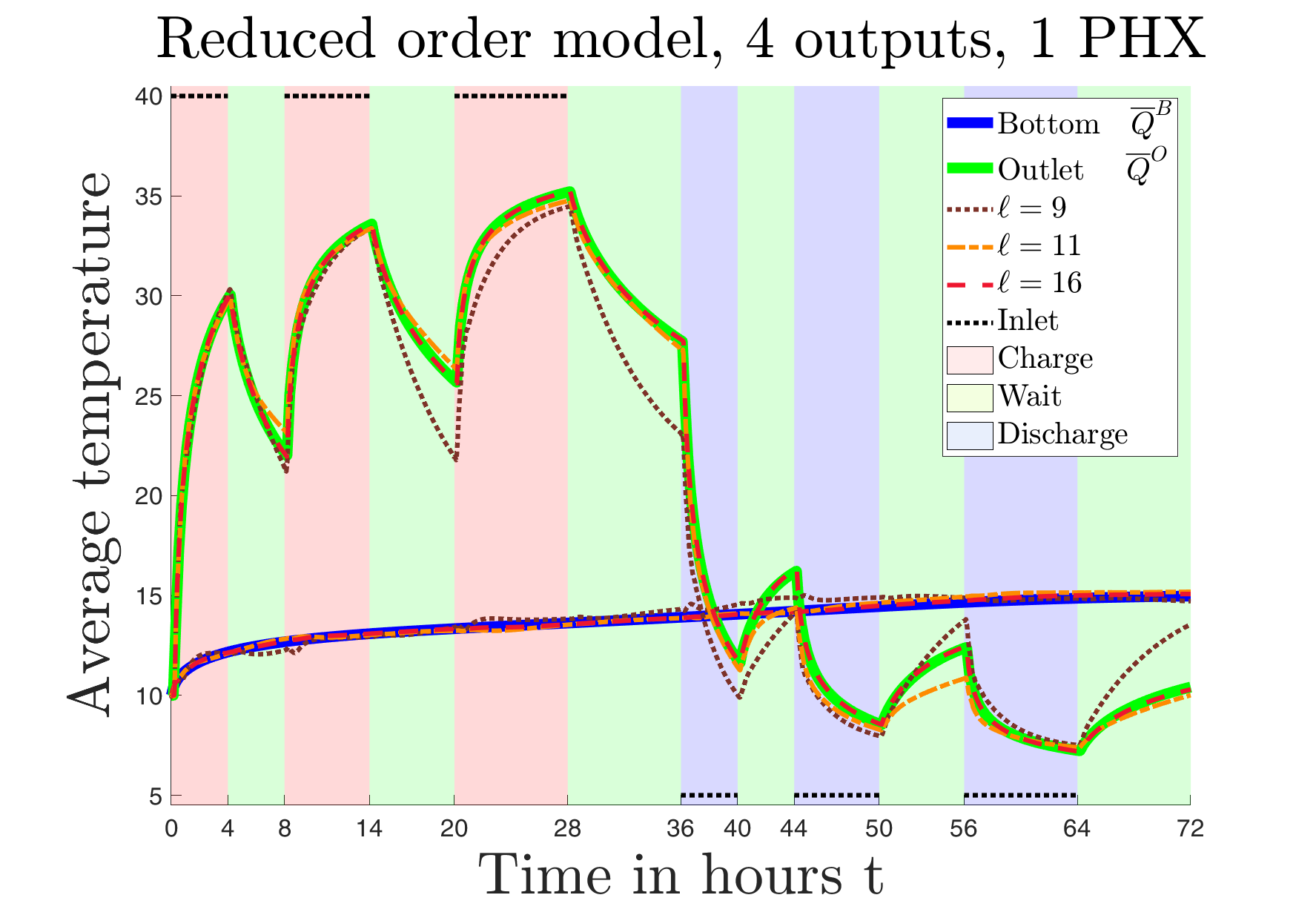}
				\includegraphics[width=.49\linewidth]{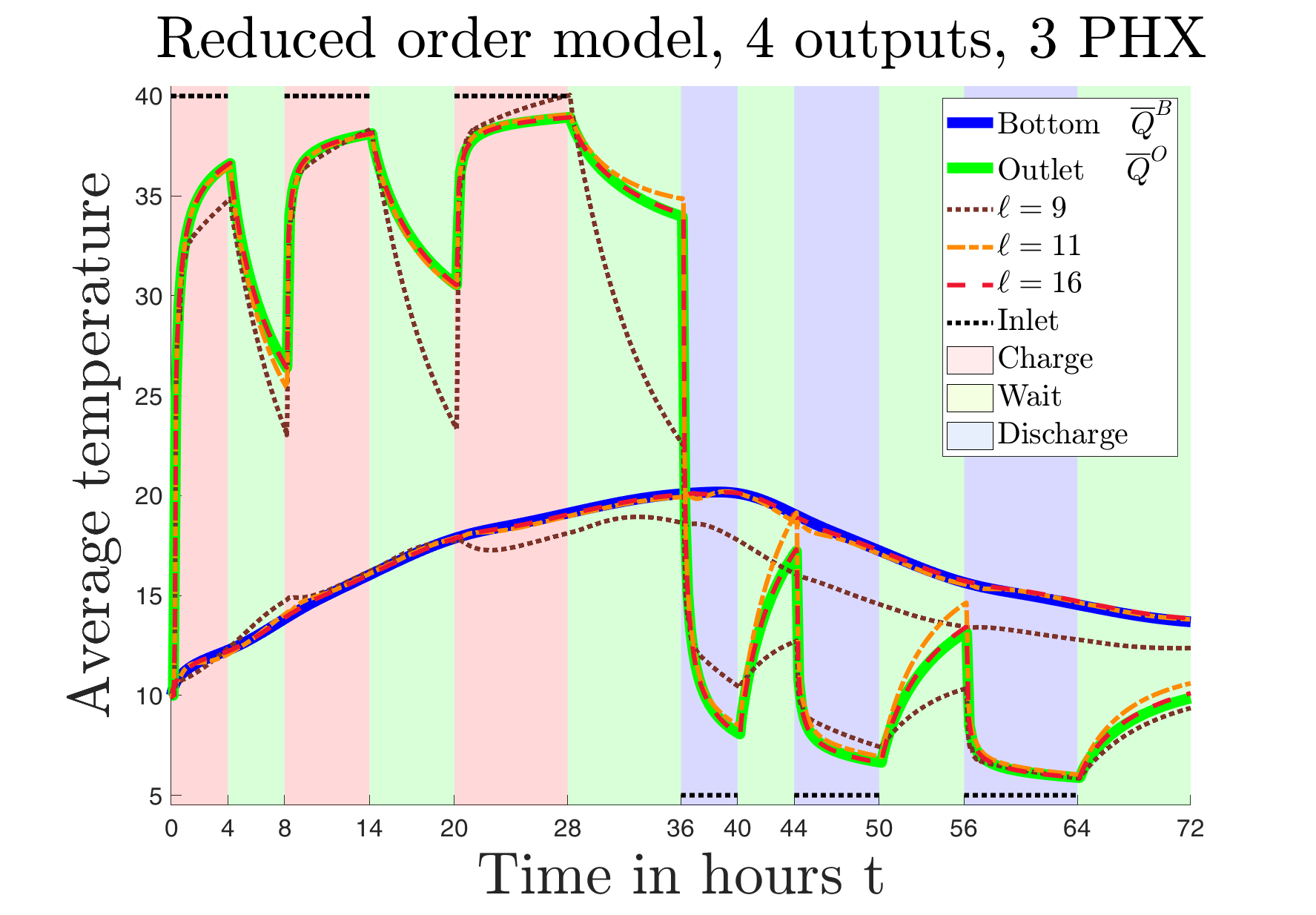}
				
				\mycaption{Model with four outputs $Z=(\Qm,\Qf,\Qout\Qbottom)^\top$: ~ Approximation of the output for $\dimred=9,11,16$. \newline  
					\begin{tabular}[t]{ll}
						Top: & Average temperatures in the medium $\Qm$ and the fluid $\Qf$,\\
						Bottom:&   Average temperatures at the outlet $\Qout$ and the bottom boundary $\Qbottom$,\\
						Left:&  one \phx, Right three \phxsk. 	
					\end{tabular}											
				}
				\label{BT3Cmfob2}
			\end{figure}
			Fig.~\ref{BT3Cmfob2}  depicts the output variables of the original and reduced-order system which are plotted against  time. In the top panels the average temperatures $Z_1(t)=\Qm(t)$ and $Z_2(t)=\Qf(t)$ in the medium and fluid are drawn as solid blue and green lines, respectively. The bottom panel shows the average temperatures at the outlet $Z_3(t)=\Qout(t)$ and at the bottom boundary $\Qbottom$ by a blue and green line, respectively. The reduced-order approximations  are drawn   for  $\dimred=9,11,16$.

			The results for the first three outputs $\Qm,\Qf,\Qout$ are similar to the experiment with those three outputs considered in  Subsec.~\ref{subsec:num_ex3}.  The approximation of the fourth output variable $\Qbottom$ is quite good and comparable to the results in Subsec.~\ref{subsec:num_ex3b}. For the model with 3 \phxs we notice some visible errors for the smallest order $\dimred=9$.
			
			Finally,   Fig.~\ref{error3Cmo} shows  for the reduced orders $\dimred$ considered above the $\Ltwo$-error  $\|Z-\widetilde{Z}\|_{\Ltwo(0,t)}$ which plotted against time $t$ together with the  error bounds from Theorem \ref{theo_errorbound}. The results are similar to Fig.~\ref{error3Cmf2} and we refer for the interpretation to the end of  Subsec.~\ref{subsec:num_ex2}.

			\begin{figure}[!h]
				\centering
				\includegraphics[width=.49\linewidth]{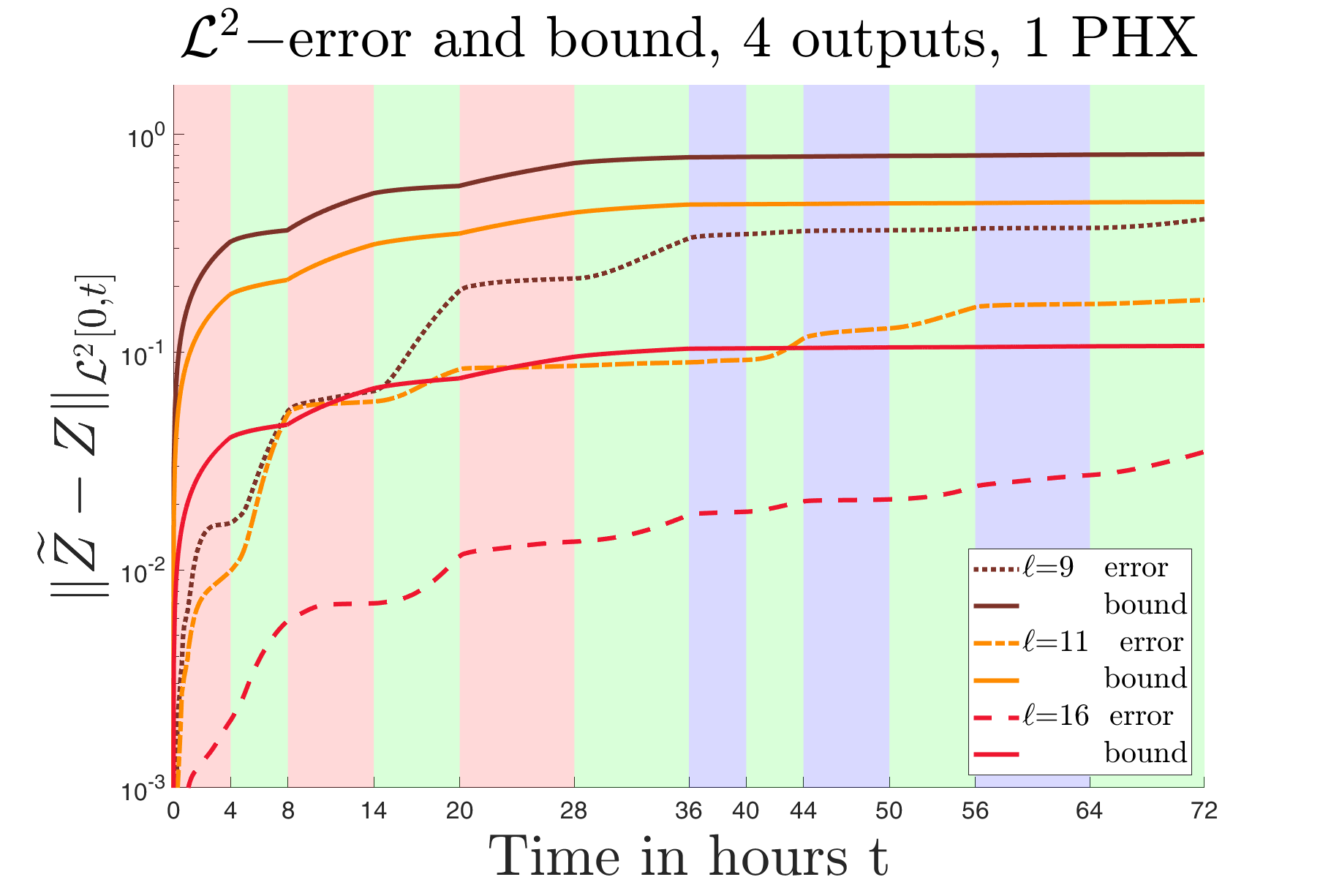}
				\includegraphics[width=.49\linewidth]{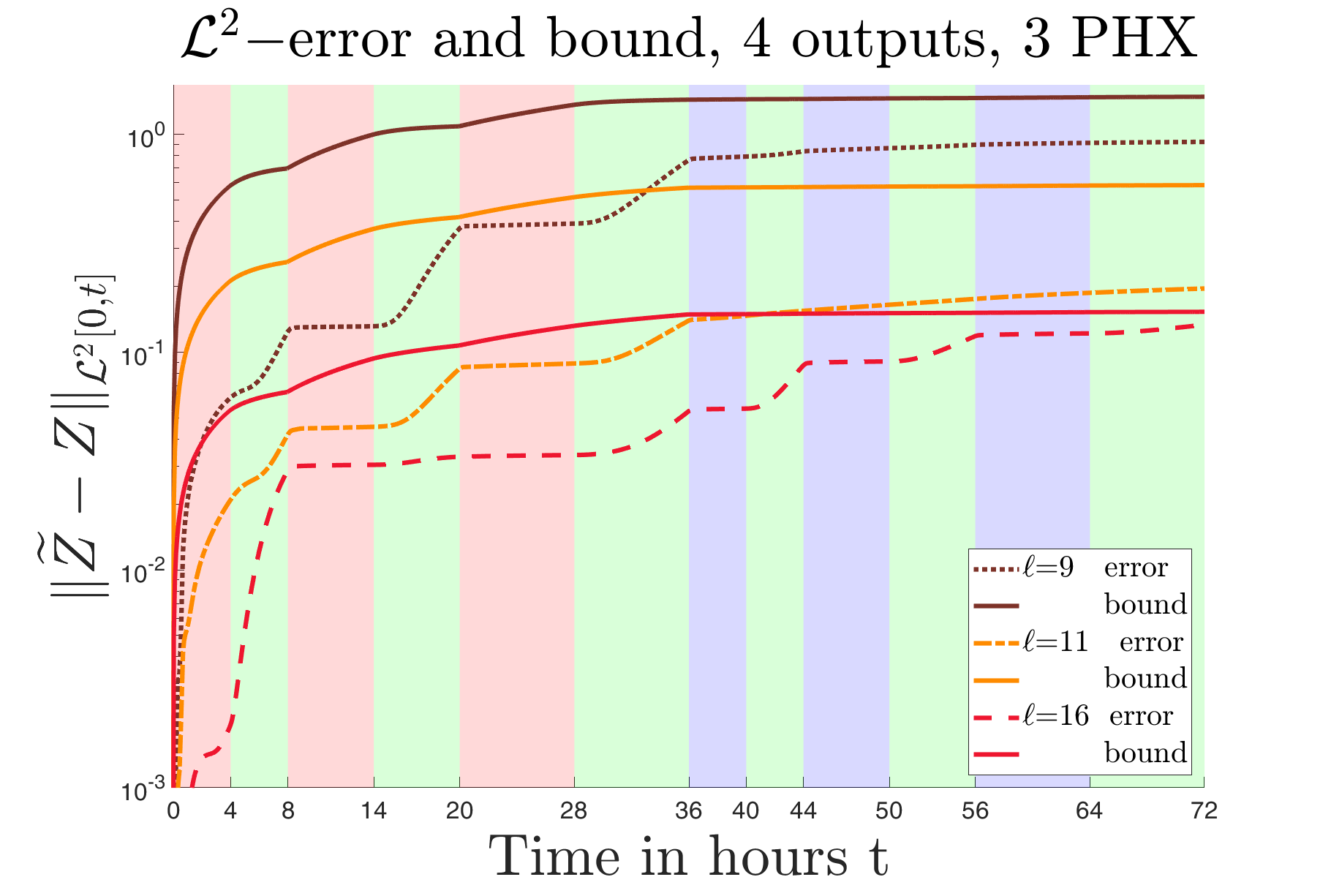}
				\caption{Model with four outputs $Z=(\Qm,\Qf,\Qout,\Qbottom)^\top$:\newline  $\Ltwo$-error and error bound for~$\dimred=9,11,16$.\quad   Left: one \phx, Right three \phxsk. 	}	
				\label{error3Cmfob2}
			\end{figure}
			
			\section{Conclusion}
			\label{conclu}
			
			In this paper we have considered the approximate description of the input-output behavior of a geothermal storage by a low-dimensional system of linear ODEs. Starting point was the mathematical modeling of the spatio-temporal temperature distribution in a two-dimensional cross section of the storage by a linear heat equation with a convection term.	
			By semi-discretization of that PDE w.r.t.~spatial variables we obtained a high-dimensional system of non-autonomous ODEs. The latter was approximated by an analogous LTI system. Reduced-order models in which the state dynamics is described by a low-dimensional system of linear ODEs were derived by the Lyapunov balanced truncation method. In our numerical experiments we considered  aggregated characteristics describing the input-output behavior of the storage  which are required for the operation of the geothermal storage within a residential heating system. The results show that that it is possible to obtain quite accurate approximations from  reduced-order systems with only a few state variables. This allows to treat the cost-optimal management  of residential heating systems as a decision making problem under uncertainty which mathematically can be formulated as a stochastic optimal control problem. Results on that will be published in an forthcoming paper.

			\newpage
			\begin{appendix}
				
				\section{List of Notations}
				\label{append_a}
				
				\begin{longtable}{p{0.3\textwidth}p{0.68\textwidth}l}
					$Q=Q(t,x,y)$ & temperature in the geothermal storage &\\	
					$T$ & finite time horizon&\\
					$l_x$,~$l_y$,~$l_z$ &width, height and depth of the storage &\\
					$\mathcal{D} =(0, l_x) \times (0,l_y)$ &domain of the geothermal storage &\\
					$\Df, ~\Dm$ & domain inside and outside the pipes &\\
					$\DInterface=\DInterfaceL \cup \DInterfaceU$ & interface between the pipes and the medium (dry soil)&\\
					$\partial \mathcal{D}$ &boundary of the domain&\\
					$\partial \Din$,~$\partial \Dout$ & inlet and outlet boundaries of the pipe&\\
					$\partial \Dleft, \partial \Dright, \partial \Dtop$,~$\partial \Dbottom$  & left, right, top and bottom boundaries of the domain&\\		
					$\mathcal{N}^*_*$  &   subsets of index pairs for grid  points  &\\
					$\mathcal{K}, \overline{\mathcal{K}}$  &   mappings $(i,j)\mapsto l$ 
					of index pairs to single indices&\\
					$v=v_0(t)(v^x, v^y)^{\top}$ & time-dependent velocity vector,&\\
					$\vconst$ &  constant velocity during pumping&\\
					$\cpf$,~$\cpm$ & specific heat capacity of the fluid  and medium &\\
					$\rhof$,~$\rhom$ & mass density of the fluid and medium&\\
					$\kappaf$,~$\kappam$ &thermal conductivity of the fluid and medium&\\
					$\af$,~$\am$& thermal diffusivity of the fluid and medium&\\
					$\heattransfer$ & heat transfer coefficient between storage and  underground &\\
					$Q_0$ &initial temperature distribution of the geothermal storage &\\
					$\Qg$ & underground temperature &\\
					$\Qin, \QinC, \QinD$ & inlet temperature of the pipe,  during charging and discharging, &\\
					$\Qm, \Qf$ & average temperature in the storage medium and fluid&\\
					$\Qout, \Qbottom$  & average temperature at the outlet   and bottom boundary&\\
					$G^*$  &  gain of thermal energy in a certain subdomain&\\
					$I_{C}, I_{W}$,  $I_{D}$  & time interval for charging, waiting, discharging  periods   &\\ 
					$\nabla$, ~~$\Delta=\nabla \cdot \nabla$ & 	 gradient, Laplace operator&\\
					\phx & pipe heat exchanger\\
					$N_x,~N_y$, &  number of grid points in $x,y$-direction&\\
					$h_x, h_y$ & step size in $x$ and $y$-direction &\\
					$n_P$ & number of pipes&\\	
					$\normalvec$& outward normal to the boundary $\partial \mathcal{D}$&\\
					$n$ &   dimension of vector $Y$ &\\	
					$\dimred$& dimension of the reduced-order system&\\
					$\mathds{I}_n$& $n \times n$ identity matrix&\\
					$\mat{A}, \mat{B}, \mat{\omatrix}$ & $n \times n$  system matrix, $n \times m$  input matrix, $n_0 \times n$ output matrix  of original  system &\\
					$\overline{\mat{A}}$, $\overline{\mat{B}}$, 	$\overline{\mat{\omatrix}}$  & $n \times n$  system matrix, $n \times m$  input matrix, $n_0 \times n$ output matrix  of transformed original  system &\\
					$\widetilde{\mat{A}}$, $\widetilde{\mat{B}}$, 	$\widetilde{\mat{\omatrix}}$  & $\dimred \times \dimred$ system matrix,  $\dimred \times m$  input matrix, $n_0  \times \dimred$  output matrix  of the reduced-order system &\\												
					$\mat{D}^{\pm}, ~\mat{A}_{L}, ~\mat{A}_{M},~\mat{A}_{R}$ & block matrices of matrix $\mat{A}$&\\	
					$Y$, 	$\overline{Y}$&$n$ dimensional state  of original and transformed original  system&\\
					$\widetilde{Y}$&$\dimred$ dimensional state   of reduced-order system&\\
					$Z$&$n_o$ dimensional output  of original  system&\\	
					$\widetilde{Z}$&$n_o$ dimensional output  of reduced-order system&\\														
					$g$& input variable of the system&\\
					$\cgram,~\ogram$&controllability and observability Gramians&\\
					$\overline{\cgram},~\overline{\ogram}$& transformed controllability and observability Gramians&\\
					$\mat{\trans}$&invariant transformation &  \\
					$\sigma_i >0$&Hankel singular values&\\
					$\Sigma$&diagonal matrix of Hankel singular values&\\
					$\mat{U,L}$& upper/ lower triangular matrix from Cholesky decomp.~of $\cgram /\ogram$&\\
					$\mat{K}$&orthogonal matrix from the eigenvalue decomposition of $\mat{U}^{{\top}} \ogram \mat{U}$&\\
					$\mat{\mat{W}},~\mat{V}$&unitary matrices from the singular value decomposition &\\											
					$\selcrit(\dimred)$ & selection criterion\\
					$\Ltwo(0,t)$ &  set of square integrable functions on $[0,t]$
				\end{longtable}

				\section{Proofs}
				\label{app_proofs}
				\subsection{Proof of Lemma~\ref{transform_lemma}}
				\label{append_d}														
				Using $\overline{Y}=\trans Y $ and $Y=\trans^{-1}\overline{Y}$ we get $ \dot{Y}=\trans^{-1}\dot{\overline{Y}}$.
				Substituting  into (\ref{sys_org}) implies  
				\begin{align*}
					\trans^{-1}\dot{\overline{Y}}(t)=\mat{A}\trans^{-1}\overline{Y}(t)+ \mat{B}g(t),\quad Z(t) =\mat{\omatrix}\trans^{-1}\overline{Y}(t)															
				\end{align*}
				Left-multiplication of the state equation by $\trans$ leads to 
				\begin{align*}
					\trans\trans^{-1}\dot{\overline{Y}}(t)=\trans\mat{A}\trans^{-1}\overline{Y}(t)+ \trans\mat{B}g(t),\quad 
					Z(t) =\mat{\omatrix}\trans^{-1}\overline{Y}(t).
				\end{align*}
				from which we obtain a transformed system  $														
				\dot{\overline{Y}}(t)=\overline{\mat{A}}~\overline{Y}(t)+ \overline{\mat{B}}g(t),~Z(t) =\overline{\mat{\omatrix}}~\overline{Y}(t)$														
				with 	$\overline{\mat{A}}=\trans\mat{A}\trans^{-1}, ~  \overline{\mat{B}}=\trans\mat{B} ~ \text{and} ~ \overline{\mat{\omatrix}}=\mat{\omatrix} \trans^{-1}.$
				
				\subsection{Proof of Theorem~\ref{theo:gram_fun}}
				\label{proof:theo:gram_fun}
				We first prove $\ofun(y) = y^{\top} \ogram \,y$.
				For zero input $g=0$ and initial state $Y(0)=y$ the state equation of system \eqref{sys_org} has a unique solution $Y(t)=e^{\mat{A}t}y$ and the output is given by  $Z(t)=\mat{\omatrix}Y=\mat{\omatrix}e^{\mat{A}t}y$ for $t \geq 0$. Hence
				\begin{align*}
					y^{\top} \ogram y&=\int_{0}^{\infty} y^{\top}e^{\mat{A}^{\top}t}\mat{\omatrix}^{\top}\mat{\omatrix} e^{\mat{A}t}y dt=\int_{0}^{\infty} (y e^{\mat{A}^{t}}\mat{\omatrix})^{\top}\mat{\omatrix} e^{\mat{A}t}ydt=\int_{0}^{\infty}\norm{\mat{\omatrix} e^{\mat{A}t}y}^2_2dt=\int_{0}^{\infty}\norm{Z(t)}^2_2dt.
				\end{align*}
				The proof of $\cfun(y) = y^{\top} \cgram^{-1} y $ can be derived from the results established in \cite{stykel2002analysis}.

				\subsection{Proof of Theorem~\ref{lyapu_theo}}
				\label{append_c}
				The proof is derived from the one sketched  in Antouslas \cite[Proposition 4.27]{antoulas2005approximation}. 
				Assume that a system matrix $\mat{A}$ is stable, then we have:	
				\begin{align*}
					\mat{A}\cgram+\cgram\mat{A}^{\top}&=\int_{0}^{\infty} \left(\mat{A} e^{\mat{A}t}\mat{B}\mat{B}^{\top} e^{\mat{A}^{\top}t}+e^{\mat{A}t}\mat{B}\mat{B}^{\top} e^{\mat{A}^{\top}t}\mat{A}^{\top}\right)dt
					=\int_{0}^{\infty} \frac{d}{dt}\left(e^{\mat{A}t}\mat{B}\mat{B}^{\top} e^{\mat{A}^{\top}t}\right)dt
					\\	&
					=\lim_{T \to \infty} e^{\mat{A}t}\mat{B}\mat{B}^{\top} e^{\mat{A}^{\top}t}\bigg|_{0}^{T}
					=\lim_{T \to \infty} e^{\mat{A}T}\mat{B}\mat{B}^{\top} e^{\mat{A}^{\top}T}-\mat{B}\mat{B}^{\top}
					=-\mat{B}\mat{B}^{\top}.\\
					\ogram\mat{A}+\mat{A}^{\top}\ogram&=\int_{0}^{\infty} \left( e^{\mat{A}^{\top}t}\mat{\omatrix}^{\top}\mat{\omatrix} e^{\mat{A}t} \mat{A}+ \mat{A}^{\top} e^{\mat{A}^{\top}t}\mat{\omatrix}^{\top}\mat{\omatrix} e^{\mat{A}t}\right)dt
					=\int_{0}^{\infty} \frac{d}{dt}\left(e^{\mat{A}^{\top}t}\mat{\omatrix}^{\top}\mat{\omatrix} e^{\mat{A}t}\right)dt
					\\	&
					=\lim_{T \to \infty}e^{\mat{A}^{\top}t}\mat{\omatrix}^{\top}\mat{\omatrix} e^{\mat{A}t}\bigg|_{0}^{T}
					=\lim_{T \to \infty}e^{\mat{A}^{\top}T}\mat{\omatrix}^{\top}\mat{\omatrix} e^{\mat{A}T}-\mat{\omatrix}^{\top}\mat{\omatrix}
					=-\mat{\omatrix}^{\top}\mat{\omatrix}
				\end{align*}

				\subsection{Proof of Lemma~\ref{transform_Gram}}
				\label{append_Gram} 
				We follow the proof in Antouslas \cite[Sec. 4.3]{antoulas2005approximation}. Let $\cgramtr$ and $\ogramtr$ be the controllability and observability Gramians of the transformed system, respectively. Then  $\cgramtr$ satisfies the following Lyapunov equation:
				\begin{align*}
					0&=\overline{\mat{A}}~\cgramtr+\cgramtr\overline{\mat{A}}^{\top}+\overline{\mat{B}}~\overline{\mat{B}}^{\top}=\trans\mat{A}\trans^{-1}\cgramtr+\cgramtr(\trans\mat{A}\trans^{-1})^{\top}+\trans\mat{B}(\trans\mat{B})^{\top}
				\end{align*}
				Multiplying by $\trans^{-1}$ from left and by $\trans^{-\top}$ from right gives
				\begin{align*}
					0	&=\mat{A}(\trans^{-1}\cgramtr\trans^{-\top})+(\trans^{-1}\cgramtr\trans^{-\top})\mat{A}^{\top}+\mat{B}\mat{B}^{\top}
				\end{align*}
				Comparing with the Lyapunov equation for the Gramian $\cgram$ of the original system which reads as $0=\mat{A}\cgram+\cgram\mat{A}^{\top}+\mat{B}\mat{B}^{\top}$  gives
				$\cgram=\trans^{-1}\cgramtr\trans^{-\top}$ and finally  $\cgramtr=\trans\cgram \trans^{\top}$.\\
				Similar reasoning gives 
				$\ogramtr=\trans^{-{\top}} \ogram \trans^{-1}$. \\Substituting into the product of the transformed Gramians yields	 $\cgramtr\ogramtr=\trans \cgram\ogram \trans^{-1}$.

				\subsection{Proof of Theorem~\ref{theo:balancing_trans}}
				\label{append_e}
				We have to prove that the system is balanced under the transformation $\trans=\Sigma^{\frac{1}{2}}\mat{K}^{{\top}}\mat{U}^{-1}$. For the Gramians of the transformed system, we obtain 
				\begin{align*}
					\cgramtr&=\trans \cgram \trans ^{{\top}}
					=\Sigma^{\frac{1}{2}}\mat{K}^{{\top}}\mat{U}^{-1} \cgram \mat{U}^{-{\top}}\mat{K}\Sigma^{\frac{1}{2}}
					= \Sigma^{\frac{1}{2}}\mat{K}^{{\top}}\mat{U}^{-1}\mat{U}\mat{U}^{\top} \mat{U}^{-{\top}}\mat{K}\Sigma^{\frac{1}{2}}, = \Sigma^{\frac{1}{2}}\mat{K}^{{\top}}\mat{K}\Sigma^{\frac{1}{2}} =\Sigma^{\frac{1}{2}}\Sigma^{\frac{1}{2}}=\Sigma.\\
					\ogramtr&=\trans^{-{\top}} \ogram \trans^{-1}
					=\Sigma^{-\frac{1}{2}}\mat{K}^{{\top}}\mat{U}^{{\top}}\ogram\mat{U}\mat{K}\Sigma^{-\frac{1}{2}}=\Sigma^{-\frac{1}{2}}\mat{K}^{{\top}}\mat{K}\Sigma^2 \mat{K}^{{\top}}\mat{K}\Sigma^{-\frac{1}{2}}=\Sigma^{-\frac{1}{2}} \Sigma^2 \Sigma^{-\frac{1}{2}}=\Sigma.
				\end{align*}
				We used $\cgram=\mat{U}\mat{U}^{\top}$, $~~\mat{U}^{\top} \mat{U}^{-{\top}}=\mat{I}_n= \mat{U}^{-1} \mat{U}$, $~~\mat{U}^{{\top}}\ogram\mat{U}=\mat{K} \Sigma^2 \mat{K}^{{\top}}$ and $\mat{K}^{{\top}}\mat{K}=\mat{I}_n$. 

			\end{appendix}	
			
			\begin{acknowledgements}														
				The authors thank   Martin Bähr (DLR),  Martin Redmann (Martin-Luther University Halle--Wittenberg), Olivier Menoukeu Pamen (University of Liverpool) and Gerd Wachsmuth (BTU Cottbus--Senftenberg) for valuable discussions	that improved this paper.\\
				P.H.~Takam gratefully acknowledges the  support by the German Academic Exchange Service (DAAD) within the project ``PeStO – Perspectives in Stochastic Optimization and Applications''.
			\end{acknowledgements}
			

		\end{document}